\ifdefined\submit
\RequirePackage{fix-cm}
\documentclass[final]{svjour3}
\smartqed
\else
\documentclass[letter,11pt]{article}
\usepackage{amsthm}
\fi
 \usepackage{amssymb,amscd}
 \usepackage{wrapfig}
 \usepackage[pdftex]{graphicx}
 \usepackage[latin1]{inputenc}
 \usepackage[framemethod=tikz]{mdframed}
 \usepackage[ruled,vlined]{algorithm2e}
 \usepackage{epsfig,epstopdf}
 \usepackage{mathrsfs}
 \usepackage{mathtools}
 \usepackage{bbm}
 \usepackage{seqsplit}
 \usepackage{multirow}
 \usepackage{caption}   
 \usepackage{float} 
 \usepackage{booktabs}
 \usepackage{bigstrut}
 \usepackage{authblk}
  \ifdefined\submit
 \usepackage[sort&compress,numbers]{natbib}
 
 \else
 \usepackage{natbib}
 \fi
\usepackage[breaklinks=true,pdfstartview=FitH,pagebackref=true]{hyperref}
\usepackage{footnotebackref}
 \usepackage[shortlabels]{enumitem}
 \usepackage{color}
\usepackage{subcaption}
\newcommand*{\lmss}{\fontfamily{lmss}\selectfont}

 \newcommand{\ds}{\displaystyle}
 \newcommand{\modify}[1]{{#1}}
 \newcommand{\T}{\mathsf{T}}
 \newcommand{\E}{\mathbb{E}}

 \newcommand{\D}{\mathcal{D}}
 \newcommand{\U}{\mathcal{U}}
 
 \newcommand{\I}{\mathcal{I}}
 \newcommand{\Id}{I} 

 \newcommand{\cpsolved}[1]{}
 \newcommand{\jhsolved}[1]{}
 \newcommand{\Tmavep}{T_{\rm MAveP}}
 \newcommand{\Tmap}{T_{\rm MAP}}
 \newcommand{\Tmarp}{T_{\rm MARP}}
 \newcommand{\bftab}{\fontseries{b}\selectfont}
 \newcommand{\Tpdmc}{T_{\rm PDMC}^{\lambda}}
 \newcommand{\Tfb}{T_{\rm FB}^{\lambda}}
 \newcommand{\Tps}{ T_{\rm PS}^{\lambda}}
\newcommand{\Tprox}{T_{\rm prox}}

 \def\duke{{\lmss duke breast-cancer}\xspace}
 \def\leu{{\lmss leukemia}\xspace}
 \def\colon{{\lmss colon-cancer}\xspace}

\graphicspath{ {code/journal/results/} }
 
 \def \rla{\rangle}
 \def \lla{\langle}
 
 \renewcommand{\Re}{{\rm I}\! {\rm R}}

 \DeclareMathOperator*{\argmin}{arg\,min}
 \DeclareMathOperator*{\rank}{rank}
 \DeclareMathOperator*{\Ker}{Ker}
 \DeclareMathOperator*{\Ran}{Ran}
 
 \DeclareMathOperator*{\Fix}{Fix}
 
 \DeclareMathOperator*{\cl}{cl}
 \DeclareMathOperator*{\dist}{dist}
 \DeclarePairedDelimiter\abs{\lvert}{\rvert}
 
 \DeclareMathOperator*{\dom}{dom}
 \newcommand{\norm}[1]{{\left\|{#1}\right\|}}
 \newcommand{\prox}{{\rm prox}}
\ifdefined\submit
\else
\usepackage[margin=0.7in]{geometry}
\fi
\usepackage[capitalize,nameinlink]{cleveref}
\crefname{assumption}{Assumption}{Assumptions}
\crefformat{equation}{\textup{#2(#1)#3}}
\crefrangeformat{equation}{\textup{#3(#1)#4--#5(#2)#6}}
\crefmultiformat{equation}{\textup{#2(#1)#3}}{ and \textup{#2(#1)#3}}
{, \textup{#2(#1)#3}}{, and \textup{#2(#1)#3}}
\crefrangemultiformat{equation}{\textup{#3(#1)#4--#5(#2)#6}}%
{ and \textup{#3(#1)#4--#5(#2)#6}}{, \textup{#3(#1)#4--#5(#2)#6}}{, and \textup{#3(#1)#4--#5(#2)#6}}

\Crefformat{equation}{#2Equation~\textup{(#1)}#3}
\Crefrangeformat{equation}{Equations~\textup{#3(#1)#4--#5(#2)#6}}
\Crefmultiformat{equation}{Equations~\textup{#2(#1)#3}}{ and \textup{#2(#1)#3}}
{, \textup{#2(#1)#3}}{, and \textup{#2(#1)#3}}
\Crefrangemultiformat{equation}{Equations~\textup{#3(#1)#4--#5(#2)#6}}%
{ and \textup{#3(#1)#4--#5(#2)#6}}{, \textup{#3(#1)#4--#5(#2)#6}}{, and \textup{#3(#1)#4--#5(#2)#6}}

\crefdefaultlabelformat{#2\textup{#1}#3}

 \numberwithin{equation}{section}
 
  \ifdefined\submit
 \newtheorem{assumption}{Assumption}
 \else
\newtheorem{theorem}{Theorem}[section]
\newtheorem{definition}[theorem]{Definition}
\newtheorem{corollary}[theorem]{Corollary}
\newtheorem{proposition}[theorem]{Proposition}
\newtheorem{lemma}[theorem]{Lemma}
\newtheorem{assumption}[theorem]{Assumption}
\theoremstyle{definition} 	
\newtheorem{example}[theorem]{Example}
\newtheorem{remark}[theorem]{Remark}
 \fi

\def\TheTitle{Global convergence and acceleration of projection methods for feasibility problems involving union convex sets}

\ifpdf
\hypersetup{
  pdftitle={\TheTitle},
  pdfauthor={Jan Harold Alcantara and Ching-pei Lee}}
\fi

\ifdefined\submit 
\title{\TheTitle\thanks{Version of \today. }}
\titlerunning{Global convergence and acceleration of projection methods for feasibility problems}
\author{Jan Harold Alcantara \and Ching-pei Lee}
\institute{*Corresponding author: Jan Harold Alcantara \at
	 Center for Advanced Intelligence Project, RIKEN  \\	\email{\url{janharold.alcantara@riken.jp}}\\ \text{} \\ 
	Ching-pei Lee \at
	Department of Statistical Inference and Mathematics, Institute of Statistical Mathematics \\ \email{\url{chingpei@ism.ac.jp}}\\
	}
\date{}
\else 
 \title{\TheTitle}
\author{Jan Harold Alcantara\thanks{\url{janharold.alcantara@riken.jp}.  Center for Advanced Intelligence Project, RIKEN.}
\qquad Ching-pei Lee\thanks{\url{chingpei@ism.ac.jp}.
Department of Statistical Inference and Mathematics, Institute of Statistical Mathematics.}
}
\fi

\begin{document}

\maketitle 
 
 \begin{abstract}
We prove global convergence of classical projection algorithms for
feasibility problems involving union convex sets, which refer to sets
expressible as the union of a finite number of closed convex sets. We
present a unified strategy for analyzing global convergence by means
of studying fixed-point iterations of a set-valued operator that is
the union of a finite number of compact-valued upper semicontinuous
maps. Such a generalized framework permits the analysis of a class of proximal algorithms for minimizing
  the sum of a piecewise
   smooth function and the difference between pointwise minimum of
   finitely many weakly convex functions and a piecewise smooth convex
   function. When realized on two-set feasibility problems, this
   algorithm class recovers alternating projections and averaged
   projections as special cases, and thus we obtain global convergence
   criterion for these projection algorithms. Using these general
   results, we derive sufficient conditions to guarantee global
   convergence for several projection algorithms for solving the
   sparse affine feasibility problem and a feasibility reformulation
   of the linear complementarity problem. Notably, we obtain global
   convergence of both the alternating and the averaged projection
   methods to the solution set for linear complementarity problems
   involving $P$-matrices. By leveraging the structures of the classes of problems we consider, we also propose acceleration algorithms with guaranteed global convergence. 
   Numerical results further exemplify that the proposed acceleration
   schemes significantly improve upon their non-accelerated
   counterparts in efficiency.\\
   
   	\ifdefined\submit
   	
   	\else 
   \fi 
 \end{abstract}

\section{Introduction}

Given two closed sets $S_1$ and $S_2$ in a Euclidean space $\E$, the
two-set feasibility problem formulated below involves finding a point in the
intersection of $S_1$ and $S_2$:
\begin{equation}
\text{find}~w\in S_1\cap S_2\tag{FP}.
\label{eq:feasibilityproblem}
\end{equation}
Given $w^0\in \E$,
the method of alternating projections \cref{eq:Map}
\begin{equation}
w^{k+1} \in (P_{S_2} \circ P_{S_1} ) (w^k),
\label{eq:Map}
\tag{MAP}
\end{equation} 
and the method of averaged projections \cref{eq:Mavep}
\begin{equation}
w^{k+1} \in \left( \frac{P_{S_1} + P_{S_2}}{2} \right) (w^k)
\label{eq:Mavep}
\tag{MAveP}
\end{equation}
are two classical projection methods for solving \cref{eq:feasibilityproblem}.
 Here, $P_S:\E \rightrightarrows S$ denotes the projector onto a closed set $S$ given by 
\begin{equation}
P_S(w) \coloneqq \{ y \in S : \norm{y-w} \leq \norm{z-w} \text{ for all }z\in S \}, \quad \forall w\in \E ,
\end{equation} 
which may contain more than one point when $S$ is nonconvex.
While global convergence of \cref{eq:Map} and \cref{eq:Mavep} to a
point in $S_1\cap S_2$ is well-understood when the sets involved are
convex \citep{Auslender1969,Bregman1965}, the global convergence 
even just to a superset of the solution set of
\cref{eq:feasibilityproblem} of MAP and MAveP for nonconvex
feasibility problems largely remains unknown. To date, only local
convergence results are known for the general nonconvex setting
(see \cite{DruLew19,LLM09}. 

\medskip 
Meanwhile, a special nonconvex structure known as \textit{union
convexity} has recently been observed in some application problems.  A
set is said to be a \emph{union convex set} if it is expressible as a
finite union of closed convex sets \citep{DT19}. A prominent example
is the problem of finding a sparse solution to a linear system $Aw=b$
with $A\in \Re^{m\times n}$ and $b\in \Re^m$ under the constraint
$\|w\|_0 \leq s$ for some $s\geq 0$. This is known as the \emph{sparse
affine feasibility problem} (SAFP), which can be cast as a feasibility
problem \cref{eq:feasibilityproblem} with
\begin{equation}
	S_1 = \{ w\in\Re^n: Aw=b\},\quad \text{and} \quad S_2 = A_s \coloneqq \{ w\in \Re^n :
	\|w\|_0\leq s\}.
	\label{eq:SAFP}
\end{equation}
$A_s$ is known as the ``sparsity set'', which is a finite union of linear
subspaces \citep{DT19,HLN14}. More recently, \cite{ACT21} studied the
\emph{general absolute value equation} (GAVE) $Ax+B|x| =c$ with
$A,B\in \Re^{m\times n}$ and $c\in \Re^m$, which can be naturally
reformulated as \cref{eq:feasibilityproblem} with $S_1=
\{(x,y)\in \Re^n\times \Re^n\} : Ax+By = c \}$, and $S_2=
\{(x,y)\in \Re^n\times \Re^n\} : y=|x| \}$ is a finite union of half-spaces. In these works, \cref{eq:Map} was used to solve the feasibility formulation, but its global convergence for GAVE is not fully understood except for homogeneous cases, while a quite restrictive assumption is used for the SAFP. 

Following the approach in \cite{ACT21}, we may also reformulate the linear complementarity problem (LCP) as a union convex set feasibility problem.  Given $M\in \Re^{n\times n}$ and $b\in \Re^n$, the LCP requires finding a point $x\in \Re^n$ that satisfies
\begin{equation}\label{eq:LCP}
x\geq 0, \qquad Mx-b \geq 0, \qquad \text{and} \qquad \lla x, Mx-b\rla = 
0.
\end{equation}
This problem encompasses many applications such as bimatrix games and
equilibrium problems, and notably includes quadratic programming as a
special case \citep{CPS92}. Many algorithms have been proposed for
solving the system \cref{eq:LCP}; we refer interested readers to
\cite{CPS92,FP03} for a comprehensive survey of theory and algorithms.
Meanwhile, a different class of algorithms can be derived through a
simple reformulation of the LCP~\cref{eq:LCP} as a feasibility
problem. Indeed, through introducing an additional variable
$y\coloneqq Mx-b$ and letting $w\coloneqq(x,y)$, the 
LCP~\cref{eq:LCP} is equivalent to \cref{eq:feasibilityproblem}
with
\begin{equation}
\begin{aligned}
S_1 &= \{ w\in \Re^{2n} ~:~ Aw = b \} \quad \text{with}\quad  
A 
\coloneqq
[M, \quad -I_n]
\in \Re^{n\times 2n},\\
S_2 &= \{ w \in \Re^{2n} ~:w_j \geq 0,\; w_{n+j}\geq 0,\;
w_jw_{n+j}=0, \;\forall 
j\in \{1,2,\dotsc,n\}\} . 
\end{aligned}
\label{eq:S1_lcp}
\end{equation}

Despite the special union convex structure of the involved sets of
these feasibility problems, determining the conditions under which the
algorithms \cref{eq:Map} and \cref{eq:Mavep} are globally convergent
remains to be an open problem, except for very specific instances of SAFP and GAVE. 
\jhsolved{removed the previous sentence as it doesn't seem necessary		} 
This work aims to show global convergence of the
classical projection algorithms applied to feasibility problems
involving union convex sets.

\subsection{Our approach}\label{sec:approach}
To prove global convergence, we interpret the projection methods \cref{eq:Map} and \cref{eq:Mavep} as proximal algorithms for solving optimization problems. In particular, we consider the structured optimization problem 
\begin{equation}
\min _{w\in \E} \quad V(w)\coloneqq f(w) + g(w) - h(w), \tag{OP}
\label{eq:minconvex_optimization}
\end{equation}
where $V$ is level-bounded, $f$ is the pointwise minimum of a finite
number of functions with Lipschitz-continuous gradients, $g$ is the
pointwise minimum of (weakly, strongly) convex functions, and $h$ is a
continuous real-valued convex function expressible as the pointwise
maximum of continuously differentiable convex functions. We denote by
$I_f$, $I_g$ and $I_h$ the collections of functions that define the
pieces of $f$, $g$ and $h$, respectively. That is, $f=\min_{f_i\in
I_f} f_i$, $g=\min_{g_i\in I_g}g_i$, and $h=\max_{h_i\in I_h} h_i$. We
say that $f_i\in I_f$ is \emph{active} at a point $w\in \E$ if $f(w) =
f_i(w)$. Note that $f$ is not necessarily smooth, and $g$ is not necessarily a convex function. For this structured problem, we introduce the following proximal-type algorithm
\begin{equation}
w^{k+1} \in T_{\rm prox}^{\lambda} (w^k)\coloneqq 
\prox_{\lambda 
	g}\left( w^k - 
\lambda f'(w^k) + \lambda h'(w^k) \right),
\label{eq:prox_alg}
\end{equation}
where $\prox_{ \lambda g}:\E \rightrightarrows \E$ is the proximal
operator, and the mappings $f',h':\E \rightrightarrows \E$ respectively map a point to the set of all gradients of functions in $I_f$ and $I_h$ that are active at the given point. 

We will show that under proper choices of the functions $f$, $g$ and
$h$ satisfying the said piecewise structures, the algorithms
\cref{eq:Map} and \cref{eq:Mavep} can be realized from the proximal
algorithm \cref{eq:prox_alg}. Hence, by deriving general conditions
under which \cref{eq:prox_alg} is globally convergent, we obtain as a
corollary the global convergence of \cref{eq:Map} and \cref{eq:Mavep}
for union convex set feasibility problems. Through this reformulation,
we can greatly simplify the task of finding sufficient conditions
for guaranteeing global convergence on the parameters for our
motivational problems of the sparse affine feasibility problem,
general absolute value equations, and the linear complementarity
problem.

Our convergence analysis of the proximal algorithm \cref{eq:prox_alg} involves studying the global convergence of the more general fixed point iterations defined by 
\begin{equation}
w^{k+1} \in T(w^k),~k\geq 0 , \tag{FPI}
\label{eq:fixedpointalgorithm}
\end{equation}
where $T:\E \rightrightarrows \E$ is a set-valued operator that
generalizes the properties of $\Tprox^{\lambda}$ for our considered
setting. In particular, we consider what we call a
\emph{union upper semicontinuous operator} $T$, which is a set-valued
map that can be decomposed as a finite union of upper semicontinuous operators, referred to as the \emph{individual operators} of $T$ (see \cref{defn:union_usc}). To establish the global convergence of \cref{eq:fixedpointalgorithm}, we assume the existence of a Lyapunov function associated with the operator $T$, in the sense defined in \cref{defn:lyapunov}. In the case of the optimization problem \cref{eq:minconvex_optimization}, we show that the objective function itself is the associated Lyapunov function for $\Tprox^{\lambda}$. 

\subsection{Contributions}
\label{sec:contributions}

The main contributions of this work include global convergence results as follows. 

\begin{itemize}
	\item[(I)] \emph{Global convergence of fixed point iterations.} Under the assumption that a Lyapunov function for an upper semicontinuous operator $T$ exists, we show in \cref{thm:global_acce} that any accumulation point of the iterations \cref{eq:fixedpointalgorithm} is a fixed point of $T$, that is, it belongs to the set 
	\[\Fix (T) \coloneqq \{ w\in \E : w\in T(w) \}.\]
	We further note in \cref{ex:not_fixed} that without the existence
	of a Lyapunov function, this result may not hold in general.
	Moreover, we also prove in \cref{thm:global_fullsequence} that
	when the individual operators of $T$ are calm at an accumulation
	point and $T$ is single-valued there, global convergence of the
	\emph{full sequence} holds.
	%

	\item[(II)] \emph{Global convergence of the proximal algorithm and
		fixed point set characterization.} 	Using the general theory,
		we establish in \cref{thm:global_pdmc_full} the global
		convergence of the proximal algorithm \cref{eq:prox_alg} to
		fixed points of $T_{\rm prox}^{\lambda}$ for suitable stepsize
		$\lambda$ by showing that the objective function of
		\cref{eq:minconvex_optimization} is a Lyapunov function for
		$T_{\rm prox}^{\lambda}$. This is stronger
		than the global \emph{subsequential} results typically
		obtained in the literature. To relate the importance of fixed
		points to the optimization problem
		\cref{eq:minconvex_optimization}, we show in
		\cref{thm:local_is_fixed} that
		\[
			\text{local minima of \cref{eq:minconvex_optimization}}
			\subset
		\Fix(T_{\rm prox}^{\lambda}).\]
		 Meanwhile, \textit{criticality} is
		a notion more traditionally used for providing necessary
		optimality conditions for \cref{eq:minconvex_optimization},
		and we show in \cref{thm:fixed_is_critical} that under a simple regularity assumption,
		\[
		\text{local minima of \cref{eq:minconvex_optimization}}
		\subset
		\Fix(T_{\rm prox}^{\lambda}) \subset 
		\text{critical points of \eqref{eq:minconvex_optimization}}.
	\]
		 Our convergence guarantee is thus stronger than
		 the traditional subsequential convergence to critical points
		 only.
	
	The setting we consider for \cref{eq:minconvex_optimization}
	subsumes the ones studied in prior works such as
	\cite{DT19,WCP18}. Consequently, our framework significantly
	extends these existing works to a wider class of optimization
	problems. More importantly, we obtain results
	concerning global convergence of the full sequence, which are
	stronger than the local or global subsequential convergence in
	existing works.
	
	\item[(III)] \emph{Global convergence of classical projection
		algorithms for union convex set feasibility problems.} Under
		certain coercivity assumptions, a consequence of the above
		general framework is that \cref{eq:Mavep} and a relaxed
		version of \cref{eq:Map}, given by 
	\begin{equation}
	w^{k+1} \in P_{S_2} ((1-\lambda)w^k + 
	\lambda P_{S_1}(w^k))
	\label{eq:Map_relaxed}
	\end{equation}
	with $\lambda \in (0,1)$, are both globally convergent to fixed points of their defining operators, as shown in \cref{sec:feasibility}.
\end{itemize}

We use the above results to determine conditions on the matrices
involved in SAFP, LCP, and GAVE\cpsolved{I added a remark on GAVE at the end of the
	LCP subsection. But I personally feel like we can also remove all
other parts mentioning GAVE. Which do you prefer?}
\jhsolved{To me, it doesn't matter that much if it's included or not in the manuscript. Actually, the only reason why I included it in the introduction is to provide more motivations for considering union convex feasibility problems. But two applications, namely SAFP and LCP, will probably suffice.}
under which the algorithms \cref{eq:Map},
\cref{eq:Map_relaxed}, and \cref{eq:Mavep} are globally convergent. We
point out that despite the availability of the above powerful tools
for the general case, the analysis for these specific problems still
requires quite some rigor, especially for proving the global convergence of \cref{eq:Map} for LCP. In particular, the following are our contributions for these feasibility problems. 

\begin{itemize} 	
	\item[(IV)] \emph{New (and old) projection algorithms for the
		sparse affine feasibility problem with global convergence
	guarantees.}
	In \cref{thm:global_pdmc_fb_sparse,thm:global_ps_affine}, we
	establish global convergence for the projected gradient algorithm
	and the relaxed method of alternating projections
	\cref{eq:Map_relaxed} applied on the sparse affine feasibility
	problem.
	The conditions we impose on the affine constraint are
	significantly looser than the ones used in existing works such as
	\cite{BT11,HLN14}, yet we still obtain global convergence to
	candidate solutions of the feasibility problem.
	One can also easily derive the same results as direct consequences
	of our analysis under the assumptions used in these works.
	In addition, our general framework is also capable of developing
	new algorithms with ease, and we thus derive several new algorithms
	for sparse affine feasibility, with similar global convergence
	guarantees.

	\item[(V)] \emph{New projection algorithms for the linear
		complementarity problem and square absolute value equations
	with global convergence guarantees.} As for the LCP
	\cref{eq:LCP}, we show in \cref{thm:global_lcp_nondegenerate} that
	\cref{eq:Mavep} and the relaxed \cref{eq:Map} given in
	\eqref{eq:Map_relaxed} are globally convergent to fixed points
	when $M$ is a \emph{nondegenerate} matrix, i.e., a matrix with
	nonzero principal minors. Moreover, local $Q$-linear convergence holds for \cref{eq:Map_relaxed}.

	We further show that for matrices with
	\textit{strictly positive} principal minors, also known as
	\emph{$P$-matrices}, global convergence of \cref{eq:Mavep} and
	relaxed MAP \cref{eq:Map_relaxed} to the actual \textit{solution
	set} of the feasibility reformulation is guaranteed.
	More significantly, we prove in \cref{cor:lcp_global_Pmatrix} that
	global convergence also holds for the \textit{original} iterations
	given by \cref{eq:Map} (as opposed to \cref{eq:Map_relaxed}),
	which is a \textit{rare result} for nonconvex feasibility
	problems. Similar to the sparse affine feasibility problem, we
	also present several other globally convergent projection-based algorithms for solving the LCP based on its feasibility reformulation.
\modify{	For GAVE involving square matrices $A$ and $B$, the results for
	the LCP can be easily adapted by reformulating the former as a
	linear complementarity problem
	with
		$M = (A^{\top} + B^{\top})(A^{\top} - B^{\top})^{-1}$,
	as discussed in \cite[Remark 2.18]{ACT21}.}
	
\end{itemize}

This work also contributes in the algorithmic side to propose
acceleration schemes that greatly improve the efficiency of
fixed-point iterations and projection algorithms.

\begin{itemize} 	
	\item[(VI)] \emph{Two Acceleration Schemes.}
We present a general acceleration scheme for the fixed point
iterations \cref{eq:fixedpointalgorithm} using the Lyapunov function
with guaranteed global subsequential convergence proved in
\cref{thm:global_acce}.
Taking advantage of the piecewise structures
of $f$, $g$ and $h$ in \cref{eq:minconvex_optimization} (or of the
union convex sets $S_1$ and $S_2$), we further derive accelerated
proximal algorithms whose global subsequential convergence follows
from the general case in \cref{sec:acce_pdmc}. In
\cref{sec:numerical}, we demonstrate empirically that
our acceleration methods significantly improve the performance of
their non-accelerated versions. The proposed acclerated algorithms
also outperform existing methods in our experiments.
\end{itemize}

\subsection{Outline}
In \cref{sec:relatedworks}, we discuss works related to the different
problem settings described above, and highlight the major
differences with and improvements over the existing works of this
paper. Mathematical preliminaries are summarized in
\cref{sec:preliminaries}. General tools concerning global convergence
of the fixed point iterations \cref{eq:fixedpointalgorithm} with union
upper semicontinuous $T$ are derived in \cref{sec:general}. Global
convergence of the proximal algorithm \cref{eq:prox_alg} and
characterization of the fixed points of $\Tprox^{\lambda}$ are
established in \cref{sec:minconvexoptimization}. We illustrate in
\cref{sec:feasibility} how to derive the projection methods
\cref{eq:Map} and \cref{eq:Mavep} from these proximal algorithms, and
we also derive another algorithm that was considered in \cite{BPW13}.
Our accleration schemes for the proximal algorithms are proposed in
\cref{sec:acce_pdmc}.
In \cref{sec:affinefeasibility}, we present a unified analysis of six projection algorithms for SAFP and LCP. \cref{sec:numerical} presents numerical experiments, and concluding remarks are given in \cref{sec:conclusion}. 

\section{Related works and further contributions}\label{sec:relatedworks}

We now compare and contrast our contributions with existing results in the literature on related topics.

\paragraph{Fixed point problems.} For the fixed point algorithm
\cref{eq:fixedpointalgorithm}, similar classes of operators $T$ that
can be expressed as a union of a finite number of set-valued operators
were studied by \cite{DT19,Tam18}. In these works,
\textit{continuous} (single-valued) individual operators, namely
nonexpansive and paracontracting maps, were considered. On the other
hand, the setting we consider involves set-valued upper semicontinuous
individual operators, and thus subsumes that in these prior works. When each individual operator is nonexpansive,  \emph{local 
	convergence} 
of \cref{eq:fixedpointalgorithm} was already established in \cite{DT19}. Our contribution described in (I) shows that a missing ingredient to extend this into a \textit{global} result is the existence of a coercive Lyapunov function (see \cref{defn:lyapunov,thm:global_acce}).
With a coercive Lyapunov function, single-valuedness of the union
operator at an accumulation point and calmness of the individual
operators at the same point are sufficient for guaranteeing global full
convergence.

\paragraph{Optimization.}For structured optimization problems of the
form \cref{eq:minconvex_optimization}, a traditional setting
considered in previous works involves a function $f$ that has a
Lipschitz continuous gradient, a proper closed convex function $g$,
and a continuous real-valued convex function $h$ \citep{LT22,WCP18}.
This setting contains a class of regularized optimization problems
that are usually motivated from statistics and machine learning. In
these applications, $f$ is a data-dependent loss function and $g-h$
represents a difference-of-convex regularizer such as the smoothly
clipped absolute deviation, minimax concave penalty, transformed
$\ell_1$, or the logarithmic penalty. However, under these 
assumptions on $f$, $g$ and $h$, it is difficult to interpret the
projection methods \cref{eq:Map} and \cref{eq:Mavep} for union convex
set \cref{eq:feasibilityproblem} as proximal algorithms for solving a
certain \cref{eq:minconvex_optimization}, as one shall see in this work. 

When $|I_f|=|I_g|=|I_h|=1$ and $g$ is a convex function, the
algorithmic operator of \cref{eq:prox_alg} reduces to a single-valued
operator $\Tprox^{\lambda} = \prox_{ \lambda g} \circ (Id - \lambda
\nabla f + \lambda \nabla h)$, which corresponds to the algorithm
studied in \cite[Section 4.2]{WCP18}, and the authors established its
global convergence under the Kurdyka-{\L}ojasiewicz (KL) assumption with a quadratic regularization on the objective function $V$. On the other hand, when $|I_f|=1$, $|I_h|=0$ and all the functions in $I_g$ are convex, $\Tprox^{\lambda}$ simplifies to $\Tprox^{\lambda} = \prox_{ \lambda g} \circ (Id-\lambda \nabla f)$, which is the forward-backward algorithm considered in \cite{DT19}, where only \textit{local convergence} to fixed points of $\Tprox^{\lambda} $ has been established. Hence, this paper provides a unifying setting for the above works, and is the first attempt to understand the global convergence of the algorithm \cref{eq:prox_alg} when $f$, $g$, and $h$ are piecewise functions described in \cref{sec:approach}. 

\paragraph{Nonconvex feasibility problems.} Due to difficulties
that come with nonconvexity, the existing body of literature on projection algorithms 
for solving nonconvex feasibility problems mainly focuses on \textit{local 
	convergence}. For instance, the local convergence of MAP for finding the intersection of union convex sets was established in \citep{DT19}. Using the same framework, one can also obtain local convergence of 
MAveP.
Global convergence for these algorithms on nonconvex sets largely
remains unknown, and
our present work shows that coercivity assumptions are sufficient to attain global convergence for the special case of union convex sets. \cpsolved{But they considered intersection of $S_i$ with all $S_i$
	being a union convex set, while technically we only have one set
being a union convex set, right?}
\jhsolved{Our specific motivating problem has one convex set and one union convex set. But the global convergence in \cref{sec:feasibility} applies to the case of two union convex sets.}

\emph{Local linear convergence} of MAP and MAveP for general nonconvex feasibility problems was studied in~\citep{LLM09} using the notion of \emph{strong regularity} of points in the solution set $S_1\cap S_2$. In the present work, we also 
establish local linear convergence of MAP (see \cref{prop:linear_ps} and \cref{sec:feasibility}) but under a Lipschitz continuity 
assumption that is more easily verifiable and potentially weaker than strong regularity. For example, for the feasibility formulation of LCP, the proof of
\cref{prop:fdecreases_nondegenerate} 
shows that if $M$ is a nondegenerate matrix and  $w^*\in S_1\cap S_2$
is a nondegenerate point (in the sense of \cref{defn:nondegenerate}), then $S_1$ and $S_2$ have a
``linearly regular intersection at $w^*$'' 
as defined in \cite{LLM09} see also the proof of \citep[Theorem 
3.19]{ACT21}. Consequently, MAP is locally linearly convergent to $w^*$ by \citep[Theorem
5.16]{LLM09}. However, it should be pointed out that nondegeneracy of $w^*$ is essential to 
guarantee this result, but this is not verifiable \emph{a priori}. \cref{thm:global_Pmatrix}, on the 
other hand, asserts 
that a linear rate is achievable whether or not $w^*$ is
nondegenerate.

\paragraph{Sparse affine feasibility problem.} Convergence analyses of
existing methods for the sparse affine feasibility problem
usually 
require near-orthonormality conditions on the matrix $A$ in
\cref{eq:SAFP}, such as the 
\emph{restricted isometry property} (RIP) introduced in \cite{CT05}. A more 
general condition subsuming the RIP is the \emph{scalable restricted isometry property}
(SRIP): A matrix $A\in \Re^{m\times n}$ is said to satisfy the SRIP of order $(d,\alpha)$ if there
exist $\mu_d \geq \nu_d >0$ with $\mu_d/\nu_d<\alpha$ such that 
\begin{equation}
\nu _d \|w\|^2 \leq \|Aw\|^2 \leq \mu_d \|w\|^2, \quad \forall w\in 
A_d.
\label{eq:SRIP}
\end{equation}
In \cite{BT11}, the authors showed that the \textit{projected gradient
algorithm} with stepsize $\lambda \in (0.5\nu_{2s}^{-1},
\mu_{2s}^{-1}]$ is globally convergent to the solution set if the SRIP
of order $(2s,2)$ holds. This algorithm coincides with
\cref{eq:prox_alg} with $f (w) = \|Aw-b\|^2/2$, $g=\delta_{A_2}$, and $h\equiv 0$.

On the other hand, \cref{eq:Map} was used in \cite{HLN14} to solve
SAFPs, and its global convergence was proved under any of the following conditions on $A$:
\begin{enumerate}
	\item[(C1)] $AA^\T = I_m$ and the SRIP of order $(2s,2)$ holds with $\mu_{2s} = 1$; or
	\item[(C2)] there exists a constant $\nu_{2s} \in (0.5,1]$ such that $\nu_{2s} \|w\|^2 \leq \|A^{\dagger}Aw\|^2$ for all $w\in A_{2s}$.
\end{enumerate}

Meanwhile, we show in \cref{thm:global_ps_affine} that we can attain global convergence to fixed points of both the projected gradient algorithm with stepsize $\lambda \in (0,1/\|A\|^2)$ and of the relaxed MAP given by \cref{eq:Map_relaxed} with $\lambda \in (0,1)$ under a significantly weaker assumption that there exists $\mu_s>0$ such that
\begin{equation}
\nu_s \|w\|^2 \leq \|Aw\|^2, \quad \forall w\in A_s.
\label{eq:srip_half}
\end{equation}
This assumption is much weaker than the SRIP of order $(2s,2)$ used in
\cite{BT11} in two ways: (i) we have no restriction on the parameter
$\nu_s$, and (ii) the inequality \cref{eq:srip_half} is required to
hold over $A_s$ only, instead of over \modify{their larger} set
$A_{2s}$. \cpsolved{So you mean insted of $2s$?}\jhsolved{Yes.} Similarly,
condition (C1) used in \cite{HLN14} for \cref{eq:Map} is much stronger
than \cref{eq:srip_half}, as it not only assumes the SRIP as in
\cite{BT11}, but also requires semi-orthogonality of $A$ and a
specific value for $\mu_{2s}$. Condition (C2), on the other hand, is
also much stronger than \cref{eq:srip_half}, since it needs
to hold over the larger set $A_{2s}$ and requires a specific range of
values for $\nu_{2s}$. Together with the fact that
$\|A^{\dagger}Aw\|^2\leq \|A^{\dagger}\|^2 \| Aw\|^2$, \modify{(C2) implies
	\cref{eq:srip_half} for some} $\nu_s$.
	Since the assumption \cref{eq:srip_half} we use is significantly
	weaker than those in \cite{BT11,HLN14}, we obtain global convergence to fixed points only.
	However, under those same stronger conditions, we can obtain easily global convergence to the solution set as a direct consequence of our framework. \cpsolved{So, the existing
	works have convergence to a solution, while ours is simply to
fixed points, which is a super set of that, right?}\jhsolved{Yes. I added two sentences to point this out. }

To our knowledge, \cref{eq:Map} and the projected gradient algorithm
discussed above are the only available methods for SAFP in the literature.
We show in \cref{thm:global_pdmc_fb_sparse} that \cref{eq:Mavep} is
also globally convergent for SAFP under the same assumption of
\cref{eq:srip_half}. Moreover, we also present other new algorithms in
\cref{sec:SAFP} that also attain global convergence under the same condition.

\paragraph{Linear complementarity problem.}  There are two well-known algorithms for LCP that, similar to \cref{eq:Map} and \cref{eq:Mavep}, are also projection-based: the \emph{basic projection algorithm} 
(BPA) and the \emph{extragradient algorithm} (EGA) \citep[see][Algorithms 12.1.1 
and 12.1.9]{FP03}). BPA is suitable when the matrix $M$ 
associated with the LCP \cref{eq:LCP} is positive definite, in the 
sense that $x^\T M x >0$ for all nonzero vector $x$. On the other hand, EGA can 
handle a positive semidefinite $M$.
Meanwhile, all the algorithms we propose in \cref{sec:lcp} are new 
projection methods for LCP with guaranteed global convergence to fixed points for 
LCPs with a nondegenerate matrix and guaranteed global convergence to the solution set for 
LCPs with a $P$-matrix. The classes of nondegenerate and $P$-matrices both include the set of 
positive 
definite 
matrices, and therefore the proposed methods can solve those LCPs that are in 
the scope of BPA. On the other hand, both the sets of nondegenerate and 
$P$-matrices contain matrices that are not positive
semidefinite,\footnote{Symmetric $P$-matrices must be positive
	definite, but nonsymmetric $P$-matrices might have all principal
	minors positive while being indefinite. See
	\cite[Example~3.3.2]{CPS92} for an example.}
and therefore lead to LCP problems solvable by our approaches but not EGA. 

\section{Notations and Definitions}\label{sec:preliminaries}
We let $w_+ \coloneqq \max \{w,0\}$, where the maximum is taken componentwise. 
$\Ran (A)$ and $\Ker(A)$ denote respectively the range and the
kernel of a 
matrix $A\in \Re^{p\times q}$. We let $\|A\|$ denote the operator norm
of $A$.
We also let 
$[q]=\{1,\dots, q\}$, and for $\Lambda \subset [q]$, we denote by 
$A_{:,\Lambda}$ the submatrix of $A$ containing all of its columns indexed 
by $\Lambda$, $A_{\Lambda,\Lambda}$ the submatrix of $A$ containing its 
rows and columns indexed by $\Lambda$, and $\Lambda^c$ the complement set $\{i: i \in [q],
i \notin \Lambda\}$. Given $w\in \Re^q$,
$w_{\Lambda}\in 
\Re^{\abs{\Lambda}}$ denote the subvector of $w$ indexed by $\Lambda$. 

Throughout this paper, $\E$ is a Euclidean space endowed with the
inner product $\lla \cdot, \cdot \rla$ and we denote its induced
norm by $\|\cdot \|$.
For a nonempty and closed set $S\subset \E$,
we denote
by $\dist(w,S)
\coloneqq \min_{z\in S} \|w-z\|$ its \emph{distance function},
${\rm conv}(S)$ its convex hull, and
$B(S,\varepsilon) \coloneqq \{ z\in \E : \dist (w,S) < \varepsilon\}$
the \emph{open ball}\cpsolved{This is actually not a ball whe $S$ is not a
	singleton, but I can't think of a better name. If you don't have a
good idea either, we can continue using it anyway.}\jhsolved{This is the terminology used in \cite{AF09}, and I think that some topology books use this term too.} around it with radius $\varepsilon>0$.
The \emph{projection operator} onto $S$, 
$P_S:\E \rightrightarrows S$, is defined by $P_S(w) \coloneqq \argmin_{z\in 
	S}\|w-z\|$, and its \emph{indicator function} $\delta_S$
is defined by 
\[\delta_S(w) = \begin{cases}
0 & \text{if}~w\in S,\\
+\infty & \text{otherwise}.
\end{cases}\]

For a finite collection of sets $\D \coloneqq \{ D_{\iota}\subset \E : 
\iota 
\in 
\I\}$,
we define the set-valued function $\phi _{\D}:\bigcup_{\iota\in\I}
D_{\iota} \rightrightarrows 2^{\I}$ by 
\begin{equation}
\phi_{\D}(w) \coloneqq \{\iota: D_\iota \in \D, w\in
D_{\iota} \}.
\label{eq:indexfunction}
\end{equation}
Let $T:\E \rightrightarrows \E$ be a set-valued operator on $\E$. If $T$ is 
single-valued at $w\in \E$, say $T(w)=\{z\}$, we slightly abuse 
the notation and write $T(w) = z$. The identity operator on $\E$ is
denoted by $Id$, while the identity matrix in $\Re^n$ is denoted by $I_n$. $T$ is said to be \emph{calm at $w$} if $T(w)\neq \emptyset$ and there exists a neighborhood $\U$ of $w$ such that 
\begin{equation}
	\label{eq:Lipschitz_at_a_point}
	T(z) \subset T(w) + \kappa \norm{z-w}B(0,1) \quad \forall z\in \U
\end{equation}
for some \emph{calmness constant} \citep{RW98} $\kappa \geq 0$. A stronger property is pointwise Lipschitz continuity: $T$ is \emph{pointwise Lipschitz continuous at $w$} if there exists $\kappa \geq 0$ (called the Lipschitz constant) and a neighborhood $\U$ of $w$ such that 
$ \|z^+ - w^+ \|\leq \kappa \| z-w\|$ for all $z\in \U $, $z^+\in T(z)$, and $w^+ \in T(w)$.
From the definition, it is clear that $T$ must be single-valued at $w$, and therefore pointwise Lipschitz continuity is equivalent to having
\[ \|z^+ - T(w) \| \leq \kappa \| z- w\| \quad \forall z\in \U,  z^+ \in T(z).\]
If $T$ is single-valued, we say that it is \emph{$\kappa$-Lipschitz
continuous} if $\|T(z)-T(w)\| \leq \kappa \|z-w\|$ for all $z,w\in \E$
and \emph{nonexpansive} when $\kappa \leq 1$.
Further, if $\kappa < 1$, $T$ is called a \emph{contraction}.

Given a set $X\subset \E$ and a point $w\in X$ such that $T(w) \neq
\emptyset$, $T$ is \emph{upper semicontinuous (usc) at $w$} if for any
neighborhood $\U$ of $T(w)$, there exists $\delta_{\U}>0$ such that
for all $z\in X$ with $\|z-w\|<\delta_{\U}$, we have $T(z) \subseteq
\U$. Moreover, $T$ 
is usc (on $X$) if it is usc at each point in $X$ \citep{AF09}. 
\begin{remark}\label{remark:usc_property}
	 Suppose that $T$ is usc at $w$, $T(w)$ is compact, and $\{w^k\} \subset X$ such that $w^k\to w$. From the definition of upper semicontinuity, it can be shown that any sequence $\{z^k\} $ such that $z^k\in T(w^k)$ is bounded, and its accumulation points belong to $T(w)$. 
\end{remark}
From usc, we further define union upper semicontinuity of an operator,
which will be central to our algorithmic and theoretical development.
\begin{definition}[Union upper semicontinuity]\label{defn:union_usc}
	An operator $T:\E \rightrightarrows \E$ is said to be \emph{union
	upper semicontinuous (union usc) on $\E$} if there exist a
	collection of nonempty closed sets $\D = \{ D_{\iota} \subset \E :
\iota \in \I\}$ and upper semicontinuous operators $\{ T_{\iota}  : \iota \in \I\}$ with $T_{\iota} :\D_{\iota} \rightrightarrows \E$ such that $\ds T(w) = \bigcup _{\iota: \iota\in \I, w\in D_{\iota}}T_{\iota}(w)$ for all $w\in \E$, $\E = \bigcup _{\iota \in \I}D_{\iota}$, $T_{\iota}(w)$ is nonempty and compact for any $w\in D_{\iota}$, and $\I$ is a finite index set. The mappings $T_{\iota}$ are called the \emph{individual operators} of $T$.
\end{definition}
Unless otherwise specified, we always use the notations in \cref{defn:union_usc} when decomposing a union usc operator $T$.

Given $g:\E \to \Re \cup \{+\infty\}$, we denote by $\dom (g) = \{ w\in \E : 
g(w)<+\infty\}$ its \emph{domain}. We say that 
$g$ is \emph{$\rho$-convex} if $g(w) - \frac{\rho}{2}\|w\|^2$ is a 
convex function. In particular, $g$ is \emph{weakly convex} if $\rho<0$, 
\emph{convex} if $\rho \geq 0$, and \emph{strongly convex} if $\rho>0$.  The \emph{subdifferential of $g$ at $w$} is defined as
\begin{equation*}
\begin{aligned}
\partial g(w) &\coloneqq \limsup_{\bar w \rightarrow w, g(\bar w)
	\rightarrow g(w)} \hat \partial g(\bar w), \quad \text{ where
}\\
\hat \partial g(\bar w)
& \coloneqq \left\{v: v \in \E, g(z) \geq g(w) + \lla v, z-w \rla + o
(\|z-w\|)\right\},
\end{aligned}
\label{eq:subgradient}
\end{equation*}
which coincides with 
\begin{equation}
\partial g(w) = \{v\in \E : g(z) \geq g(w) + 
\lla v, z-w \rla,
~\forall z\in \E \},
\label{eq:partial_h_convex}
\end{equation}
when $g$ is convex.
Given $\lambda>0$, the \emph{Moreau envelope} and the (possibly set-valued) \emph{proximal
mapping} of $g$ are respectively defined by 
\begin{align}
	M_g^{\lambda} (w) &\coloneqq \min_{z\in \E} g(z) + 
	\frac{1}{2\lambda}\|z-w\|^2,
	\label{eq:moreau}\\
	\prox_{\lambda g} (w) &\coloneqq \argmin_{z\in \E} g(z) + 
	\frac{1}{2\lambda}\|z-w\|^2.
	\label{eq:prox} 
\end{align}
If $g = \delta_S$, then $\prox_{\lambda g} (w) $ reduces to the projector operator $P_S$ for any $\lambda>0$. 
For $S \subset \E$, we define $\prox_{\lambda g}(S) \coloneqq \bigcup 
_{w\in S}\prox_{\lambda 
	g} (w)$.
If there exists a finite family of functions $\{g_j : j\in J\}$, where 
$g_j:\E 
\to \Re \cup \{ +\infty\}$ for all $j \in J$, such that for any $w\in \E$, 
we 
have $g(w) \in 
\{g_j(w): j\in J'\}$ for some $J'\subset J$, we denote
\begin{equation}
	D_j(g) \coloneqq \{ w\in \dom(g): g(w) = g_j(w)\}.
	\label{eq:g_piecewise}
\end{equation}

We list some important properties of the proximal operator and
the Moreau envelope of a function $g$ that is the pointwise minimum of
a finite number of proper functions that will be utilized in this work.
\begin{lemma}[{\cite[Proposition 5.2]{DT19}}]
	\label{lemma:moreau_prox}
	Let $g=\min_{j\in J} g_j$, where $g_j$ is a proper function for
	all $j \in J$, $J$ is a finite set, and 
	$\lambda>0$. 
	Then
	\begin{enumerate}[(a)]
		\item $M_g^{\lambda} (w) = \min _{j\in J} M_{g_j}^{\lambda}(w)$ 
		for all $w\in \E$.
		\item $\prox_{\lambda g}(w) = \ds \bigcup_{j: w\in 
			D_j ({M_{g}^{\lambda}}) } 
		\prox_{\lambda g_j}(w)$, where $D_j ({M_{g}^{\lambda}}) \coloneqq \{ 
		w \in \E : M_g^{\lambda}(w) 
		= 
		M_{g_j}^{\lambda}(w)\}$.
	\end{enumerate}
\end{lemma}



\section{Fixed point problems involving usc operators}\label{sec:general}
To establish global convergence for \cref{eq:prox_alg}, we first abstract
it as a fixed point algorithm \cref{eq:fixedpointalgorithm} associated
with a union usc operator $T$, and then obtain convergence guarantees
for \cref{eq:fixedpointalgorithm}.
In our analysis, we will make use of a Lyapunov function associated with the
operator $T$, which we define as follows. 
\begin{definition}[Lyapunov function]\label{defn:lyapunov}
	A function $V:\E \to \Re\cup \{+\infty\}$ continuous in its
	domain is a \emph{Lyapunov function} 
	for $T$ if $\inf V > -\infty$,
	\begin{equation}
	\sup _{w^+ \in T(w) }V(w^+) \leq V(w) ~for~any~w\in \E,
	\label{eq:Lyapunovcondition}
	\end{equation}
	\modify{and $w\in \Fix (T)$ whenever the equality holds}.\jhsolved{Revised the definition since we only need one direction anyway.} 
\end{definition}
Through utilizing such a Lyapunov function,
we also propose an acceleration strategy that uses the momentum
term as an easy-to-compute potential descent direction for the
Lyapunov function in \cref{alg:accelerated_FPA}.
The original fixed point algorithm \cref{eq:fixedpointalgorithm} is a
special case of \cref{alg:accelerated_FPA} by setting $t_k \equiv 0$.

We now show in \cref{thm:global_acce} that existence of a Lyapunov
function for $T$ is sufficient for guaranteeing that all accumulation
points of \cref{alg:accelerated_FPA}, and thus also of
\cref{eq:fixedpointalgorithm}, are fixed points.
A sufficient condition for the existence of such accumulation points
is that the Lyapunov function is coercive.

\begin{algorithm}[tb]
	\SetAlgoLined
	Let $V$ be a Lyapunov function for $T$.

	Choose $\sigma>0$ and $w^0 \in \E$. Set $w^{-1} = w^0$ and $k=0$.
	
	\begin{description}
		\item[Step 1.] Set $z^k \gets w^k + t_kp^k$, where $p^k \coloneqq
		w^k-w^{k-1}$
		and $t_k\geq 0$ is a stepsize such that 
		\begin{equation}
		V(z^k) \leq V(w^k)- 
		\frac{\sigma}{2}t_k^2 \| p^k\|^2.
		\label{eq:descent_acce_fpa}
		\end{equation}
		\item[Step 2.] Select $w^{k+1} \in T(z^k)$,
		$k \gets k+1$, and go back to Step 1. 
	\end{description}
	\caption{Accelerated fixed point algorithm for an operator $T$.}
	\label{alg:accelerated_FPA}
\end{algorithm}



\begin{theorem}[Global subsequential convergence of \cref{eq:fixedpointalgorithm} and
	\cref{alg:accelerated_FPA}]
	\label{thm:global_acce}
	Let $T$ be a union usc operator. If there exists a Lyapunov function $V$ for $T$, then any accumulation point of a 
	sequence generated by 
	\cref{alg:accelerated_FPA} belongs to $\Fix (T)$. In particular, any accumulation point of \cref{eq:fixedpointalgorithm} is a fixed point of $T$. 
\end{theorem}

\begin{proof}
	First, we show that if $w^*$ is an accumulation point of a sequence generated by \cref{alg:accelerated_FPA}, then there exists  $(w^*)^+ \in T(w^*)$ such that $(w^*)^+$ is 
	also an accumulation point of $\{w^k\}$. To this end, let $\{ w^{k_j}\}_{j=0}^{\infty}$	be a subsequence of 
	$\{w^k\}$  that convergens to $w^*$. Now consider
	$\{ w^{k_j+1}\}_{j=0}^{\infty}$, where $w^{k_j+1} \in 
	T(z^{k_j})$. Since the index set $\I$ is finite, there exists $\iota\in\I$ 
	and a subsequence $\{ w^{k_{j_r}+1}\}_{r=0}^{\infty}$ of $\{ w^{k_j+1}\}_{j=0}^{\infty}$
	such that 
	$w^{k_{j_r}+1} \in T_{\iota}(z^{k_{j_r}})$. By the definition of $T$, we 
	have 
	$\{z^{k_{j_r}}\}_{r=0}^{\infty}\subset D_{\iota}$. We also note from 
	\cref{eq:descent_acce_fpa} and \cref{defn:lyapunov} that
	\[\frac{\sigma}{2} t_{k_{j_r}}^2 
	\|p^{k_{j_r}}\|^2\leq  V(w^{k_{j_r}}) - 
	V(w^{k_{j_r}+1}). \]
	By summing the inequality above from $r=0$ to infinity, we see that
	the monotonicity (from the algorithm) and the lower-boundedness of $V$
	(from \cref{defn:lyapunov})
	imply $t_{k_{j_r}} 
	\|p^{k_{j_r}}\| \to 0$, so $z^{k_{j_r}} \to w^*$.
	Therefore, by the
	closedness of $D_{\iota}$, we get $w^* \in D_{\iota}$, and thus
	$T_{\iota}(w^*) \subseteq T(w^*)$. 
	Since $T_{\iota}$ is usc at $w^*$, we have from \cref{remark:usc_property} 
	that $\{ w^{k_{j_r}+1}\}_{r=0}^{\infty}$ has a subsequence converging to 
	some $(w^*)^+ \in T_{\iota}(w^*) \subseteq T(w^*)$, as desired.

Next, we will show that $(w^*)^+ = w^*$ to prove that $w^*$ is a fixed
point of $T$.
By \cref{eq:Lyapunovcondition,eq:descent_acce_fpa}, the sequence $\{V(w^k)\}$ is monotonically
decreasing and bounded below, so $\{V(w^k)\}$ converges to a finite
value.
If $w^*$ is an accumulation point of $\{w^k\}$, we have from the first
part of the proof that there exists another accumulation point
$(w^*)^+ 
	\in T(w^*)$ of $\{w^k\}$. Since $\{V(w^k)\}$ is convergent, by
	taking the corresponding subsequences of $\{w^k\}$ that 
	converge to $w^*$ and $(w^*)^+$, we must have $V(w^*) = V((w^*)^+)$ by 
	the continuity of $V$. By \cref{defn:lyapunov}, we conclude that
	$w^* \in \Fix (T)$.
		\ifdefined\submit
		\qed
		\else
		\qedhere
		\fi
\end{proof}

The existence of a Lyapunov function is crucial for the conclusion of
\cref{thm:global_acce}, as illustrated in the following example.
\begin{example}\label{ex:not_fixed}
Let $T:\Re \rightrightarrows \Re$ be a usc operator (and therefore union usc) given by $T(w)=[-2w,2w]$ if $w>0$ and $T(w)=[\tfrac{w}{2},-2w]$ if $w\leq 0$.
The sequence with terms given by $w^k = (-1)^k$ can be generated from
\cref{eq:fixedpointalgorithm}, and it is clear that no Lyapunov
function in the sense of \cref{defn:lyapunov} exists for $T$.
Meanwhile, $-1$ and $1$ are accumulation points of $\{w^k\}$, but $-1$ is not a fixed point of $T$.  
\end{example}

With some mild conditions on the individual operators $T_{\iota}$ in
addition, we are able to establish the global convergence of the full
sequence generated by \cref{eq:fixedpointalgorithm}.
\begin{theorem}
	\label{thm:global_fullsequence}
	Let $T$ be a union usc operator with an associated Lyapunov function for $T$. Let $\{w^k\}$ be 
	a sequence generated 
	by~\cref{eq:fixedpointalgorithm} with an accumulation point $w^*$,
	and suppose for 
	each $\iota\in \phi_{\D}(w^*)$, $T_{\iota}$ is calm (see \cref{eq:Lipschitz_at_a_point}) at $w^*$ with
	parameter $\kappa_{\iota}\in [0,1]$.
	If $T$ is single-valued at 
	$w^*$, then $w^k\to 
	w^*$ and  $\phi_{\D}(w^k) \subset \phi_{\D}(w^*)$ for 
	all sufficiently large $k$. Moreover, the rate of convergence is locally 
	$Q$-linear 
	if $\kappa_{\iota}< 1$ for all $\iota\in\phi_{\D}(w^*)$. 
\end{theorem}
The following lemma for component identification is needed for proving
\cref{thm:global_fullsequence}.
\begin{lemma}\label{lemma:invariantball}
	Let $\D=\{ D_{\iota} : \iota \in \I\}$ be any finite collection of 
	closed 
	sets in $\E$ and denote $\mathbb{U} \coloneqq \bigcup _{\iota\in\I} 
	D_{\iota}$. 
	Then for any $w^* \in \mathbb{U}$, there exists $\delta>0$ such that 
	$\phi_{\D}(w) \subset \phi_{\D} (w^*)$ for all 
	$w\in B(w^*,\delta)\cap \mathbb{U}$, where $\phi_{\D}$ is defined
	by \cref{eq:indexfunction}.
\end{lemma}


\begin{proof}
	Given $\iota \notin \phi_{\D}(w^*)$, there is
	$\delta_{\iota}>0$
	such that $B(w^*,\delta_{\iota}) \cap D_{\iota} =
	\emptyset$.
	Otherwise, we can construct a sequence
	$\{w^k\}\subset D_{\iota}$ converging to $w^*$.
	By the closedness of $D_{\iota}$, this implies $\iota \in
	\phi_{\D}(w^*)$, contradicting the assumption.
	Setting $\delta 
	= \min \{ \delta_{\iota}:\iota\notin\phi_{\D}(w^*)\}$, we see that 
	$B(w^*,\delta) \cap D_{\iota} = \emptyset$ for all $\iota\notin 
	\phi_{\D}(w^*)$. In other words, if $\iota \in 
	\phi_{\D}(w)$ (\textit{i.e.},  $w \in D_{\iota}$) and $w\in B(w^*,\delta)\cap \mathbb{U} $, then $\iota \in 
	\phi_{\D}(w^*)$.
		\ifdefined\submit
		\qed
		\else
		\qedhere
		\fi
\end{proof}

\begin{proof}[{\cref{thm:global_fullsequence}}]
	Since $w^*\in \Fix (T)$ by \cref{thm:global_acce} and $T$ is
	single-valued at $w^*$, we get $T_{\iota}(w^*) = w^*$ for all
	$\iota \in \phi_{\D}(w^*)$.
	Meanwhile, using 
	\cref{lemma:invariantball}, we can find $\delta>0$ such that 
	$\phi_{\D}(w)\subset \phi_{\D}(w^*)$ for all 
	$w \in B(w^*,\delta)$.
	We can then find a subsequence $\{w^{k_j}\}_{j=0}^{\infty} \subset
	B(w^*, \delta)$ of
	$\{w^k\}$ such that $w^{k_j}\to w^*$.
	Let $\iota_0 \in \phi_{\D} 
	(w^{k_0})$ be such that $w^{k_0+1} \in T_{\iota_0}(w^{k_0})$. Since 
	$w^{k_0}\in B(w^*,\delta)$, we have 
	$\iota_0\in \phi_{\D}(w^*)$ and thus $w^* \in T_{\iota_0}(w^*)$.
	By 
	\cref{eq:Lipschitz_at_a_point}, 
	\begin{equation*}
	\|w^{k_0+1} - w^*\| \leq \kappa_{\iota_0} \| w^{k_0} - w^*\| \leq
	\kappa 
	\delta,
	\label{eq:T_lipschitz}
	\end{equation*}
	where $\kappa \coloneqq \max \{ \kappa_{\iota}: \iota\in 
	\phi_{\D}(w^*) \}$. Thus, 
	$w^{k_0+1}\in  B(w^*, \delta)$ and we may proceed 
	inductively to conclude that $\|w^{k+1}-w^*\| \leq \|w^k - w^*\|$ for 
	all $k\geq k_0$ and
	\[\|w^k-w^*\| \leq \kappa^{k-k_0} \|w^{k_0}-w^*\|,\quad \forall k\geq 
	k_0. \]
	Thus, $\{\| w^k - w^*\| \}_{k=k_0}^{\infty}$ is a decreasing sequence 
	that is bounded below, and is therefore convergent. Since 
	$\|w^{k_j}-w^*\| \to 0$, it follows that $\{\| w^k - 
	w^*\| \}_{k=0}^{\infty}$ also converges to $0$, that is, $w^k \to w^*$. 
		\ifdefined\submit
		\qed
		\else
		\qedhere
		\fi
\end{proof}

%
%

\begin{remark}[Component identification]
	\label{remark:component}
	If $D\subset \E$ is any closed set such that $D\cap \Fix (T) = 
	\emptyset$ 
	and if $\{w^k\}$ generated by \cref{eq:fixedpointalgorithm} is bounded, 
	then $D$ contains at most finitely many terms of $\{w^k\}$ by 
	\cref{thm:global_acce}. Thus, only those $D_{\iota}$ 
	containing a fixed point of $T$ can possibly contain infinitely many terms 
	of 
	$\{w^k\}$. Moreover, the conclusion of \cref{thm:global_fullsequence} 
	that 
	$\phi_{\D}(w^k) 
	\subset 
	\phi_{\D}(w^*)$ for all large $k$ allows us to identify 
	the operators $T_{\iota}$ that will yield the fixed point $w^*$ of $T$. In 
	particular, this result implies that a fixed point of $T_{\iota}$ with $\iota \in \phi_{\D}(w^k)$
	corresponds to 
	a fixed point of $T$, provided that $k$ is chosen large enough.	
\end{remark}

\modify{The following example shows the essentiality of the condition
	of single-valuedness at an accumulation point
	for global convergence in \cref{thm:global_fullsequence}}.

\begin{example}
	\modify{
	Let $T=T_1\cup T_2$ where $T_i(x,y) =
	((-1)^i,y)$ for $i=1,2$. $T_i$ are nonexpansive, and the function
	$V(x,y) = x^2 + \delta_S(x,y)$ with $S = \{ (x,y)\in \Re^2 : |x| = 1 \}$
	is a Lyapunov function for $T$. Moreover, $\Fix (T) = \{ (x,y)
		\in \Re^2 :|x| = 1\}$, and $T$ is not single-valued
		anywhere. When initialized at a point $(x^0,y^0) $,
		the iterations given by
		\eqref{eq:fixedpointalgorithm} may oscillate between the fixed
		points $(-1,y^0) $ and $(1,y^0)$, showing that global
	convergence may not take place. }
	\cpsolved{Looks like $T_3$ is not needed so I removed it to simplify the
	example.}
\end{example}

\section{Applications to optimization}
\label{sec:minconvexoptimization}
We now focus on the optimization problem
\cref{eq:minconvex_optimization}
with $f$, $g$ and $h$ possibly nonconvex, and $f+g-h$
bounded from below. We consider $g$ that belong to 
the class of \emph{min-$\rho$-convex} functions defined below, 
which is a generalization of min-convex functions introduced in \cite{DT19}. 

\begin{definition}[min-$\rho$-convex function]
	\label{defn:minconvex}
	We say that $g:\E \rightarrow \Re \cup \{ +\infty\}$ is a 
	\emph{min-$\rho$-convex} function if
	there exist a finite index set $J$, and $\rho$-convex, 
	proper,
	and lower semicontinuous functions $g_j :\E \rightarrow \Re \cup \{ 
	+\infty\}$, $j\in J$, such that
	\begin{equation*}
		g(w) = \min _{j\in J}\, g_j(w), \quad \forall w\in \E.
		\label{eq:g_rho_convex}
	\end{equation*}
	We call $g$ \emph{min-convex} 
	if $\rho\geq 0$.
\end{definition}


We formalize below the assumptions on $f$, $g$ and $h$ described in \cref{sec:approach}. 

\begin{assumption} \text{} 
	\label{assum:min_convex}
	\begin{enumerate}[(a)]
		\item The functions $f$, 
		$g$ and $h$ are expressible as 
		\[f = \min _{i\in I}\, f_i, \quad g=\min_{j\in J}\, g_j, \quad 
		\text{and} 
		\quad 
		h=\max_{m\in M}\, h_m,\]
		where $I$, $J$ and $M$ are finite index sets.
		\item For each $i\in I$, $f_i$ has $L_i$-Lipschitz continuous 
		gradient in $\E$ for some $L_i>0$. 
		\item For each $j\in J$, $\dom(g_j)$ is closed, and $g_j$ is a proper 
		and $\rho$-convex 
		function continuous in $\dom(g_j)$. 
		\item For each $m\in M$, $h_m$ is a continuously differentiable 
		convex function 
		in $\E$. 
		\item For all $(i,j,m)\in I\times J\times M$, the function $f_i + 
		g_j-h_m$ is 
		coercive over $\E$. 
	\end{enumerate}
\end{assumption}

\begin{remark}
	\label{remark:consequences_assumption}
	We mention some consequences of the above assumptions.  
	\begin{enumerate}[(a)]
		\item By \cref{assum:min_convex} 
		(b), we have from the descent lemma (see, for example, \cite[Lemma 
		5.7]{Beck17}) that 
		\begin{equation}
			f_i(z) \leq f_i(w) + \lla \nabla f_i(w) , z-w \rla + 
			\frac{L_i}{2}\|z-w\|^2, \quad \forall w,z\in \E. 
			\label{eq:descentlemma}
		\end{equation} 
		
		\item With
		\cref{assum:min_convex} (b)-(d), the sets $D_i(f) $,  $D_j(g)$
		and $D_m(h)$ defined as in \cref{eq:g_piecewise} are closed for any 
		$(i,j,m)\in I\times J\times M$. 
		Hence, by \cref{assum:min_convex} (a), we may 
		write 
		$\dom(g)$ 
		as the union of a finite number of closed sets:
		\begin{equation}
			\dom(g)=\bigcup _{(i,j,m)\in 
				I\times J\times M} D_i(f) \cap D_j(g) \cap D_m(h)
			\label{eq:E=union_D}		
		\end{equation}
		
		\item From \cref{assum:min_convex} (c), $g_j(z) + 
		\frac{1}{2\lambda}\|z-w\|^2$ is a strongly convex function of $z$ 
		for any $\lambda\in (0,\bar{\lambda})$, where 
		\begin{equation}
			\bar{\lambda} = \begin{cases}
				-\frac{1}{\rho} & \text{if}~\rho <0, \\
				+\infty & \text{if}~\rho \geq 0.
			\end{cases}
			\label{eq:lambda_bar}
		\end{equation}
		Thus, $\prox_{\lambda g_j}$ defined by \cref{eq:prox}
		is single-valued for any $\lambda\in (0,\bar{\lambda})$ in
		$\E$. It is also not difficult to show that $\prox_{\lambda
		g_j}$ is $(1 + \rho \lambda)^{-1}$-Lipschitz continuous for
		all $\lambda$ in the same range,
		so $\prox_{\lambda g_j}$ is nonexpansive when $\rho\geq 0$.
		It also follows that the Moreau envelope $M_{g_j}^{\lambda}$ 
		of $g_j$ is continuous. 
		
		\item 
		By \cref{assum:min_convex} (a) and (d), $h$ is convex, so we have from 
		\cite[Theorem 3.50]{Beck17} that 
		\begin{equation*}
			\partial h(w) = {\rm conv} \left( \left\lbrace \nabla h_m(w) : m\in M 
			~\text{such that } w\in D_m(h)\right\rbrace \right).
			\label{eq:partial_h}
		\end{equation*}	
		\item[(e)]
		\cref{assum:min_convex} (a) and \cref{eq:E=union_D}
		indicate that \cref{assum:min_convex} (e) implies coerciveness of $f+g-h$.  
	\end{enumerate}
	
\end{remark}

%
We revisit the proximal algorithm \cref{eq:prox_alg}, which we recall as follows:
\begin{equation}
	w^{k+1} \in \Tpdmc (w^k) \coloneqq 
	\prox_{\lambda 
		g}\left( w^k - 
	\lambda f'(w^k) + \lambda h'(w^k) \right), \tag{PDMC}
	\label{eq:pdmc}
\end{equation}
where $\lambda\in (0,\min \{\bar{\lambda}, 1/L \})$,
 $\bar{\lambda}$ is given by 
\cref{eq:lambda_bar}, $L\coloneqq \max_{i\in I} L_i$ with $L_i$ given in 
\cref{assum:min_convex} (b),
$f',h':\E \rightrightarrows \E$ are defined by 
\begin{equation}
	\label{eq:prime}
\begin{aligned}
	f'(w) &\coloneqq \left\lbrace \nabla 
	f_i(w) : i\in I~\text{such 
	that } 
	w\in D_i(f) \right\rbrace, \\
	h'(w) &\coloneqq \left\lbrace\nabla 
	h_m (w) : 
	m\in M ~\text{such that }w\in D_m(h)\right\rbrace.
\end{aligned}
\end{equation}
We call the iterations \cref{eq:pdmc}
the \emph{proximal difference-of-min-convex algorithm} (or PDMC, for short). 

For specific settings of $g$ and/or $h$, we recover several familiar
algorithms from \cref{eq:pdmc}.
When $h\equiv 0$, we obtain the \emph{forward-backward algorithm} given by 
\begin{equation}
w^{k+1} \in \Tfb (w^k) \coloneqq 
\prox_{\lambda 
	g}\left( w^k - 
\lambda f'(w^k)  \right). \tag{FB}
\label{eq:fb}
\end{equation}
When $h\equiv 0$ and $g$ is the indicator 
function of a union convex set $S$,
\cref{eq:minconvex_optimization} reduces to a \emph{union convex set-constrained} problem given by
\begin{equation}
	\min f(w) \quad \text{subject to } w\in S.
	\label{eq:min_over_UCS}
\end{equation}
and \cref{eq:pdmc} simplifies to the \emph{projected subgradient algorithm}
\begin{equation}
w^{k+1} \in \Tps (w^k) \coloneqq P_S (w^k - \lambda 
f'(w^k)). \tag{PS}
\label{eq:PS}
\end{equation}
Since $S$ is a union convex set, there exists a finite collection of
closed convex sets $\{ R_j:j\in J\}$ such that $S = \bigcup_{j\in
J}R_j$\modify{, and thus $g=\delta_S=\min_{j\in J}\delta_{R_j}$}
is a min-convex function satisfying \cref{assum:min_convex}. 


\subsection{Global subsequential convergence to fixed points}
In the setting of optimization problems, the objective function is the natural 
choice of Lyapunov function for descent algorithms. We show this in the 
next theorem and use \cref{thm:global_acce} to establish global
subsequential convergence for \cref{eq:pdmc}.
\begin{theorem}
	\label{thm:global_pdmc}
	Let $\{w^k\}$ be any sequence generated by 
	\cref{eq:pdmc} 
	with $\lambda\in (0,\min \{\bar{\lambda}, 1/L \})$. Under 
	\cref{assum:min_convex},
	$\{w^k\}$ is bounded and its 
	accumulation points belong to $\Fix (\Tpdmc)$. 
\end{theorem}
\begin{proof}
	By \cref{thm:global_acce}, it suffices to show that 
	$T_{\rm 
		PDMC}^{\lambda}$ is a union usc operator, and that there exists a Lyapunov function for $T_{\rm 
		PDMC}^{\lambda}$. 
	First, we claim that $V\coloneqq f+g-h$ is a Lyapunov function for 
	\cref{eq:pdmc}. Simple algebraic manipulations of 
	\cref{eq:descentlemma}
	give
	\begin{equation}
		\frac{L_i}{2}\|z\|^2 - f_i(z) \geq \frac{L_i}{2}\|w\|^2 - f_i(w) + 
		\lla 
		L_iw - \nabla f_i(w), z-w\rla, \quad \forall w,z\in \E, \,
		\forall i \in I,
		\label{eq:descentlemma_to_convexity}
	\end{equation}
	and since $L\geq L_i$, we have that $\frac{L}{2}\|w\| - f_i(w)$ is convex. 
	Hence, 
	\[\frac{L}{2}\|w\|^2 - f(w) = \max _{i\in I} \left\lbrace 
	\frac{L}{2}\|w\|^2 - f_i(w)\right\rbrace\] 
	is also a convex function. By Theorem 3.50 of \cite{Beck17},
	\begin{align*}
		\partial \left( \frac{L}{2}\|w\|^2 - f(w) \right)
		= {\rm conv} 
		\left\lbrace Lw - \nabla f_i(w): i\in I~\text{s.t. }w\in 
		D_i(f) \right\rbrace \supseteq Lw - f'(w). 
	\end{align*}
	Thus, 
	\begin{equation*}
		\frac{L}{2}\|z\|^2 - f(z) \geq \frac{L}{2}\|w\|^2 - f(w) + 
		\lla 
		Lw - y, z-w\rla, \quad \forall w,z\in \E, \,
		\forall y \in f'(w).
		\label{eq:descentlemma_to_convexity2}
	\end{equation*}
	By reversing the 
	algebraic manipulations done to get \cref{eq:descentlemma_to_convexity} 
	from \cref{eq:descentlemma}, we have 
	\begin{equation}
		f(z) \leq Q_f^{\lambda}(z,w) \coloneqq f(w) + \lla y, z-w \rla + 
		\frac{1}{2\lambda}\|z-w\|^2,\quad \forall 
		w,z\in \E,
		\label{eq:f<=Qf}
	\end{equation}
	for any $y\in f'(w)$ and $\lambda \in (0,1/L]$. Now, let $w^+ 
	\in 
	\Tpdmc(w)$ for some $w\notin \Fix 
	(\Tpdmc)$, say $w^+ \in \prox_{\lambda g} (w - 
	\lambda y + \lambda v)$ for some $y\in f'(w)$ and $v\in 
	h'(w) \subset \partial h(w)$. 
	From \cref{eq:f<=Qf} and \cref{eq:partial_h_convex}, we have
	\begin{equation}
		V(z) \leq Q_f^{\lambda}(z,w) - L_h (z,w) + g(z), \quad \forall
		z\in\E,\quad L_h(z,w) \coloneqq h(w) + \lla v,z-w\rla.
		\label{eq:V<=Q_L_g}
	\end{equation}
	From that $Q^\lambda_f(w,w) - L_h(w,w) + g(w) \equiv V(w)$,
	standard calculations (see for example, \citep[Section~5]{HA13a}) yield
	for $\lambda \in (0,\bar \lambda)$ that
	
	\begin{align}
		\prox_{\lambda g}(w-\lambda y + \lambda v) &\in \argmin_{z\in \E} \,
		Q_f^{\lambda}(z,w) -L_h(z,w) + g(z),
		\label{eq:prox=argmin}\\
		V(w) - V(w^+) &\geq \frac{1-\lambda L}{2\lambda}\|w^+-w\|^2. 
		\label{eq:V_Lyapunov}
	\end{align}
	Thus, $V(w^+) < V(w)$ provided $\lambda \in (0,\min 
	\{\bar{\lambda},1/L\})$, proving that $V$ 
	is a Lyapunov 
	function 
	for 
	\cref{eq:pdmc}. 

	It remains to prove that each $D_{\iota}$ is closed to verify that
	$T$ is a union usc operator.
	For each $\iota \coloneqq (i,j,m)\in I\times J \times 
	M$, we define 
	$T_{\iota} : D_{\iota} \to \E$ by 
		\begin{align}
		\label{eq:T_ijl}
			T_{\iota} &\coloneqq \prox_{\lambda g_j} \circ (Id-\lambda \nabla 
			f_i+ 
			\lambda \nabla h_m), \quad \text{where} \\
			D_{\iota} &\coloneqq 
			\left\{ w\in \E : w\in D_i(f) \cap D_m(h), w-\lambda \nabla 
			f_i(w) +\lambda \nabla h_m(w) \in
			D_{j}(M_g^{\lambda})\right\}.
			\nonumber
		\end{align}
	By \cref{lemma:moreau_prox} (b), we obtain $\Tpdmc(w) = 
	\bigcup _{\iota: w\in 
		D_{\iota}}T_{\iota}(w)$. Since $\prox_{\lambda g_j}$ is continuous on 
		$\E$ (see 
	\cref{remark:consequences_assumption} (c)), 
	so is $T_{\iota}$. By the continuity of  $M_{g_j}^{\lambda}$ together 
	with \cref{lemma:moreau_prox} (a), 
	$D_j ({M_{g}^{\lambda}})$ is closed. Hence, the continuity of $\nabla f_i$ 
	and $\nabla h_m$ plus the closedness of $D_i(f) \cap D_m(h)$ imply 
	that $D_{\iota}$ is indeed closed.
		\ifdefined\submit
		\qed
		\else
		\qedhere
		\fi
\end{proof}

\begin{remark}
	\label{remark:lambda_full}
	When $\lambda = 1/L$, we cannot guarantee from 
	\cref{eq:V_Lyapunov} that $V(w^+)$ is strictly less than $V(w)$.
	But the result of \cref{thm:global_pdmc} can
	still be valid for $\lambda = 1/L$ if monotonicity of $V$ is ensured by 
	some other mechanisms.
	One such instance is
	when the 
	regularizer $g$ is convex, in which case $V$ is a Lyapunov
	function since the right-hand side of \cref{eq:V<=Q_L_g} is $L$-strongly
	convex for $\lambda = 1/L$.
\end{remark}

\begin{corollary}
	If \cref{assum:min_convex} (a), (b), (d), and (e) hold, and $g$ is a convex 
	function, then the 
	conclusions of \cref{thm:global_pdmc} hold for
	$\lambda = 1/L$.
	\label{cor:pdmc_g_convex}
\end{corollary}

As special cases, global subsequential convergence of the
forward-backward algorithm and projected subgradient algorithm follows
directly from \cref{thm:global_pdmc} and \cref{cor:pdmc_g_convex}.

\subsection{Global convergence to fixed points}
\label{sec:global}
We now take advantage of \cref{thm:global_fullsequence} to prove global convergence of the \emph{full sequence} generated by \cref{eq:pdmc} to a fixed point.

\begin{theorem}\label{thm:global_pdmc_full}
	Let $\{w^k\}$ be any sequence generated by \cref{eq:pdmc} with $\lambda 
	\in 
	(0,1/L)$ and suppose that \cref{assum:min_convex} holds so that an accumulation point $w^* \in \Fix(\Tpdmc)$ exists. If
	$Id-\lambda \nabla f_i$ is nonexpansive over $\E$ for all $i\in
	I$ and $\Tpdmc$ is single-valued at $w^*$, then $\{w^k\}$ converges to $w^*$ under either one of the following conditions.
	\begin{enumerate}[(a)]
		\item $\nabla h_m$ is nonexpansive for all 
		$m \in M$ and $g_j$ is $\rho$-convex with $\rho\geq 1$ for all $j\in J$; or
		\item $h\equiv 0$ and $g_j$ is $\rho$-convex with $\rho\geq 0$ for all $j\in J$.
	\end{enumerate}
	Moreover, $Q$-linear convergence to $w^*$ is achieved if $\rho >
	1$ for (a) or $\rho > 0$ for (b). 
\end{theorem} 
\begin{proof}
	By \cref{thm:global_pdmc} and \cref{thm:global_fullsequence}, it suffices 
	to show that $T_{\iota}$ given in \cref{eq:T_ijl} is calm at $w^*$ with 
	parameter $\kappa_{\iota}\in [0,1]$. To prove part (a), we have from 
	\cref{remark:consequences_assumption} (c) and the assumptions on 
	$Id-\lambda \nabla f_i$ and $\nabla h_m$ that $T_{\iota}$ is in fact 
	nonexpansive if $\rho\geq 1$, and a contraction if $\rho>1$. Thus, the claim follows. The proof for (b) is similar. The individual operators given by \cref{eq:T_ijl} reduces to $T_{\iota}=\prox_{\lambda g_j}(Id-\lambda \nabla f_i)$, which  is nonexpansive when $\rho \geq 0$, and a contraction when $\rho >0$.
		\ifdefined\submit
		\qed
		\else
		\qedhere
		\fi
\end{proof}

\begin{remark} \label{remark:specialcases} \text{}~
	\label{remark:fconvex_nonexpansive}
\begin{enumerate}[(a)]
	\item 	When $f_i$ is a convex function, it is well-known that $Id - \lambda
		\nabla f_i$ is nonexpansive, so \cref{thm:global_pdmc_full} is
		applicable for min-convex functions $f$.
	\item \cref{thm:global_pdmc_full} (b) provides sufficient
		conditions for global convergence of the forward-backward
		algorithm \cref{eq:fb}. When specialized to the case of $g_j = \delta_{R_j}$ with a closed convex set $R_j$, we obtain global convergence of the projected subgradient algorithm \cref{eq:PS}. Applying further \cref{lemma:invariantball} to the collection $\D = \{ R_j:j\in J\}$ and noting the convergence of $\{w^k\}$ to $w^*$, we see that there exists $N>0$ such that $\{w^k\}_{k=N}^{\infty} \subset 
	\bigcup _{j:w^*\in R_j} R_j$. 
\end{enumerate}
\end{remark}

The property of the projected subgradient algorithm noted in \cref{remark:specialcases} (b) has
practical consequences in the same spirit as the component 
identification result described in \cref{remark:component}. In particular, 
the locations of the iterates 
can be used to identify which $R_j$'s contain 
the convergence point $w^*$. In turn, using an identified $R_j$, we 
may reduce 
\cref{eq:min_over_UCS} to a 
convex-constrained problem, which is potentially easier to solve than the 
original one.

Another important consequence of \cref{remark:specialcases} is that if
$Id-\lambda \nabla f_i$ is a contraction when restricted to $R_j$, for
any $(i,j) \in I\times J$, we can further attain a local $Q$-linear rate of convergence.
This will be useful when we analyze sparse affine feasibility and
linear complementarity problems in \cref{sec:affinefeasibility}.

\begin{proposition}[Linear convergence]
	\label{prop:linear_ps}
	Consider the setting of \cref{thm:global_pdmc_full} (b) with $g_j = \delta_{R_j}$ where $R_j$ is a closed convex set for all $j\in J$. 
	If for some $\lambda \in (0,1/L)$, $Id - \lambda \nabla 
	f_i$ is $\kappa_{ij}$-Lipschitz continuous on $R_j$  with $\kappa_{ij} \in 
	[0,1)$ for all $(i,j)\in I\times J$, $\{w^k\}$ converges to some
		point in $\Fix (\Tps (w^k))$ with a local $Q$-linear rate. 
\end{proposition}
\begin{proof}
	We already have that $w^k \to w^*$ by \cref{thm:global_pdmc_full}. 
	From \cref{thm:global_fullsequence}, we know that there exists $N \geq 0$ 
	such that for 
	each $k\geq N$, we can find $\iota = (i,j)\in I\times J$ (dependent on $k$) 
	such that $w^{k+1} = T_{\iota }(w^k)$ and $w^* = T_{\iota}(w^*)$, where 
	$T_{\iota} = P_{R_j} \circ  (Id-\lambda \nabla f_i)$. Then
	\begin{eqnarray}
	\|w^{k+1} - w^* \| & = & \| (P_{R_j}\circ (Id-\lambda \nabla 
	f_i) )(w^k)
	-  (P_{R_j}\circ (Id-\lambda \nabla 
	f_i) )(w^*) \| \nonumber  \\
	& \leq & \| (Id-\lambda \nabla f_i)(w^k - w^*) \| .\nonumber  	
	\end{eqnarray} 
	Taking a larger $N$, if necessary, we have from 
\cref{remark:specialcases} (b) that there exists $j^*\in J$ such that
$w^k, w^* \in R_{j^*}$. By 
	hypothesis, we then obtain from the above inequality that $\|w^{k+1}-w^*\| 
	\leq 
	\kappa_{ij^*} \|w^k-w^*\| \leq \kappa \|w^k - w^*\|$, where $\kappa = \max 
	\{ \kappa_{ij}:w^*\in R_j\}$.
		\ifdefined\submit
		\qed
		\else
		\qedhere
		\fi
\end{proof}

\subsection{Illustrative examples: Applications to union-convex-feasibility problems}
\label{sec:feasibility}
We revisit the feasibility problem \cref{eq:feasibilityproblem} to
demonstrate some applications of the framework studied in the previous
section. In particular, we consider \cref{eq:feasibilityproblem} with
$S_1\cap S_2 \neq \emptyset$ and
$S_1$ and $S_2$ being union convex sets, say 
\begin{equation}
S_1 = \bigcup _{i \in I} R_i^{(1)}, S_2 = \bigcup 
_{j \in J} R_j^{(2)}, |I|,|J| < \infty,\;
R_k^{(l)} \text{ is convex for all } k,l.
\label{eq:S1S2_UCS}
\end{equation}
We establish global convergence of the methods of averaged 
projections and alternating projections.

\subsubsection{Method of averaged projections}
\label{sec:mavep}
\cref{eq:feasibilityproblem} can be 
reformulated as an optimization problem:
\begin{equation}
\min _{w\in\E}\, \frac{1}{2}\dist(w,S_1)^2 + \frac{1}{2}\dist(w,S_2)^2,
\label{eq:min_averagedprojections}
\end{equation}
and each term is a min-convex function if \cref{eq:S1S2_UCS} holds.
Indeed, 
\begin{equation}
f(w)\coloneqq \frac{1}{2}\dist (w,S_1)^2 = \min _{i\in I}\,
\left\lbrace f_i (w) \coloneqq \frac{1}{2} \dist \left(w,R_i^{(1)}\right)^2
\right\rbrace.
\label{eq:f_feasibility}
\end{equation}
By the convexity of $R_i^{(1)}$, 
$f_i$ is a convex function whose gradient, namely $\nabla f_i(w) = w - 
P_{R_i^{(1)}}(w)$, is 1-Lipschitz 
continuous, see \cite[Example 
5.5]{Beck17}. On the other hand, we also have 
\begin{equation}
\frac{1}{2}\dist(w,S_2)^2 = \frac{1}{2}\|w\|^2 - \left( 
\frac{1}{2}\|w\|^2 - 
\frac{1}{2} \dist (w,S_2)^2\right) \eqqcolon g(w) - h(w).
\label{eq:dist2_dc}
\end{equation}
Note that $h$ can also be expressed as 
\[h(w)= \max _{j\in J}\,\left\lbrace h_j(w) \coloneqq \frac{1}{2}\|w\|^2 - 
\frac{1}{2}\dist \left(w,R_j^{(2)}\right)^2\right\rbrace,\]
and is therefore convex.\footnote{Each component is the convex
	conjugate of $\norm{w}^2/2 + \delta_{R_j^{(2)}}(w)$, and the maximum
	of convex functions is convex.}
By \cref{eq:S1S2_UCS}, $g$ and $h$ 
satisfy \cref{assum:min_convex} (a), (c), and (d). Hence, we may use 
\cref{eq:pdmc} to solve 
\cref{eq:min_averagedprojections}. In this setting,
\cref{eq:prime} becomes
\begin{equation}
f'(w) 
= w-P_{S_1}(w), \quad h'(w) = P_{S_2}(w).
\label{eq:fprime_feasibility}
\end{equation}
Since $\prox_{\lambda g} (w) = \frac{1}{1+\lambda}w$, \cref{eq:pdmc}
simplifies to 
\begin{equation}
\Tmavep^{\lambda} (w^k) \coloneqq \left( 
\frac{1-\lambda}{1+ 
	\lambda }Id + \frac{\lambda}{1+\lambda}(P_{S_1} + P_{S_2})\right) 
(w^k), 
\quad \lambda \in (0,1].
\label{eq:mave_relaxed}
\end{equation}
When $\lambda = 1$, we denote $\Tmavep \coloneqq \Tmavep^1$ and the 
above algorithm further simplifies to the method of averaged projections \cref{eq:Mavep}.
Global convergence of \cref{eq:mave_relaxed} for all $\lambda \in (0,1]$ (that is, including $\lambda=1$) holds under the assumptions of \cref{thm:global_pdmc_full} (a) and \cref{cor:pdmc_g_convex}.


\subsubsection{Method of alternating relaxed projections}
\label{sec:marp}
Another projection algorithm for solving
\cref{eq:feasibilityproblem} can be obtained by applying directly the 
FB algorithm \cref{eq:fb} to \cref{eq:min_averagedprojections}. Let $f$ be 
given by \cref{eq:f_feasibility}, $g \coloneqq \dist (\cdot, 
S_2)^2/2$, and $g_j\coloneqq \dist 
(\cdot,R_j^{(2)})^2/2$. From Example~6.65 of \cite{Beck17}, we have 
\begin{equation*}
\prox_{\lambda g_j}(w) = 
\frac{\lambda}{1+\lambda}P_{R_j^{(2)}} 
(w) + 
\frac{1}{1+\lambda}w.
\label{eq:z_j}
\end{equation*}
Using the optimality condition of \cref{eq:moreau} and 
\cref{lemma:moreau_prox}, the above formula for the proximal 
mapping of $g_j$ extends to that of $g$:
\begin{equation*}
\prox_{\lambda g}(w) = \frac{\lambda}{1+\lambda}P_{S_2} (w) + 
\frac{1}{1+\lambda}w .
\label{eq:prox_dist2}
\end{equation*}
%
Thus, \cref{eq:fb} becomes
\begin{align}
\nonumber
\Tmarp^{\lambda}(w^k)
\coloneqq 
\frac{1}{1+\lambda}\left(\lambda P_{S_2} \left((1-\lambda)w^k+\lambda 
P_{S_1}(w^k)\right)  + 
((1-\lambda)w^k+\lambda P_{S_1}(w^k))\right),
\label{eq:marp}
\end{align}
recovering a special instance of the \emph{method of alternating relaxed 
	projections (MARP)} studied in 
\cite{BPW13}. Its global convergence to fixed points immediately follows from \cref{thm:global_pdmc_full} (b) for stepsizes $\lambda \in (0,1)$.

\subsubsection{Method of alternating projections}\label{sec:map}
An equivalent reformulation of \cref{eq:feasibilityproblem} is
\begin{equation}
\min _{w\in \E} ~ f(w) + \delta_{S_2} (w) ,
\label{eq:min_alternatingprojections}
\end{equation}
where $f$ is defined by \cref{eq:f_feasibility}. By 
\cref{eq:fprime_feasibility}, \cref{eq:PS} then takes the form
\begin{equation}
w^{k+1} \in \Tmap^{\lambda} (w^k)\coloneqq P_{S_2} ((1-\lambda)w^k + 
\lambda P_{S_1}(w^k))
\label{eq:map_relaxed}
\end{equation}
with $\lambda \in (0,1]$.
When $\lambda = 1$, we denote $\Tmap 
\coloneqq \Tmap^1$ and the algorithm simplifies to \cref{eq:Map}.
%
%
%
Global convergence to fixed points again follows from \cref{thm:global_pdmc_full} (b), but only for stepsizes $\lambda \in (0,1)$
because $V \coloneqq f + \delta_{S_2}$ may not be a 
Lyapunov function when $\lambda = 1$ (see \cref{eq:V_Lyapunov}). Hence, we 
cannot guarantee the global (subsequential) convergence of the MAP
scheme. This in fact is a major theoretical open problem 
for MAP.  In the 
literature, particularly for nonconvex feasibility problems, it is 
often the case that global convergence results are obtained for some 
relaxations of MAP with $\lambda < 1$ in \cref{eq:map_relaxed} only \citep{ACT21,ABRS10,HA13a,BPW13}.
In \cref{sec:ps_lcp_pmatrix}, we overcome this challenge to show that \cref{eq:Map} attains global convergence for
a union-convex-feasibility reformulation of LCP problems.

\subsection{Fixed point sets and critical points}
\label{sec:fixed_critical}
Having established the convergence of \cref{eq:pdmc} to fixed 
points, we now show its importance in view of the optimization problem \cref{eq:minconvex_optimization}. In particular, we show that being a 
fixed point is a necessary condition for optimality. 

\begin{theorem}
	\label{thm:local_is_fixed}
	Let $w$ be a local minimum of \cref{eq:minconvex_optimization}. If
	\cref{assum:min_convex} holds, then
	\begin{enumerate}[(a)]
		\item $w$ is a local minimum of $f_i+g_j-h_m$ for all $(i,j,m)\in 
		I\times J\times M$ such that $w\in D_i(f) \cap D_j(g) \cap D_m(h)$;
		\item there exists $\varepsilon\in 
		(0,1/L)$, 
		dependent on 
		$w$, such that $w\in \Fix (T_{\rm 
			PDMC}^{\lambda})$ for any $\lambda \in (0,\min\{\bar{\lambda}
		,\varepsilon\}]$; and
		\item if $w$ is a global 
		minimum, then $\Tpdmc$ is single-valued at $w$ and 
		$w\in 
		\Fix (\Tpdmc)$, for all $\lambda \in 
		(0,\min\{\bar{\lambda},1/L\})$.
	\end{enumerate} 
\end{theorem}

\begin{proof}
	Let $w$ be a local minimum of $f+g-h$, and let $(i,j,m)\in I\times J\times 
	M$ be such that $w \in D_i(f) \cap D_j(g) 
	\cap D_m(h)$. Then there exists 
	$\delta >0$ such that 
	\begin{equation}
		(f_i+g_j-h_m)(w) = (f+g-h)(w) \leq (f+g-h)(z) \leq (f_i+g_j-h_m)(z),
			\quad \forall z\in 
		B(w,\delta),
		\label{eq:localmin}
	\end{equation}
	where the last inequality follows from the definition of $f$, $g$ and $h$. 
	That is, $w$ is a local minimum of $f_i+g_j-h_m$, which proves 
	\cref{thm:local_is_fixed} (a).

	Meanwhile, consider any 
$\lambda \in (0,\min\{\bar{\lambda}) ,1/L\}]$. Computation similar 
	to that for obtaining \cref{eq:V<=Q_L_g} gives
	\ifdefined\submit 
	\begin{align}
		\nonumber
	& (f_i+g_j-h_m)(z)\\
		\nonumber
	 \leq&~
	f_i(w) + \lla \nabla f_i(w), z-w\rla + 
	\frac{1}{2\lambda}\|z-w\|^2 - h_m(w) - \lla \nabla 
	h_m(w),z-w\rla 
	+ 
	g_j(z) \\
	\eqqcolon&~ Q^\lambda_{i,j,m}(z),
	\label{eq:tobound}
	\end{align}
	
	\else 
	\begin{equation}
		\begin{aligned}
	(f_i+g_j-h_m)(z) & \leq
	f_i(w) + \lla \nabla f_i(w), z-w\rla + 
	\frac{1}{2\lambda}\|z-w\|^2 - h_m(w) - \lla \nabla 
	h_m(w),z-w\rla 
	+ 
	g_j(z) \\
	& \eqqcolon Q^\lambda_{i,j,m}(z),
\end{aligned}
\label{eq:tobound}
	\end{equation}
	\fi 
	\noindent where the inequality holds since $\lambda \geq L \geq L_i$.
	\cref{eq:localmin,eq:tobound} then lead to
	\begin{equation*}
	(f_i+g_j-h_m)(w) \leq \min _{z\in B(w,\delta)} 
	Q_{i,j,m}^{\lambda} 
	(z)
	\leq Q_{i,j,m}^{\lambda} 
	(w)
	= (f_i+g_j-h_m)(w).
	\end{equation*}
	That is, $w$ is a local minimum of $Q_{i,j,m}^{\lambda}$.
	 Since $\lambda\in (0,\bar{\lambda})$,
	$Q_{i,j,m}^{\lambda} 
	(z)$ is a strongly convex function in $z$ (see also 
	\cref{remark:consequences_assumption} (c)), and is therefore
	globally and uniquely minimized at $z=w$. Hence, we conclude that (see also 
	\cref{eq:prox=argmin})
	\begin{equation*}
	w = \prox_{\lambda g_j}(w - \lambda \nabla f_i(w) + \lambda 
	\nabla h_m(w)), \quad \forall \lambda \in (0, \min\{\bar \lambda, 1 /
	L\}],
	\end{equation*}
	for all $(i,j,m)\in I\times 
	J\times M$ such that  $w\in D_i(f) \cap D_j(g) \cap D_m(h)$.
	By 
	\cref{eq:T_ijl}, in order to prove part (b) of the theorem, it suffices to show that there exists 
	$\varepsilon>0$ 
	and some $ (i,j,m)\in I\times 
	J\times M$ such that $w\in D_i(f)\cap D_j(g) \cap D_m(h)$ and 
	\begin{equation}
		\begin{gathered}
	M_g^{\lambda} (w-\lambda \nabla f_i(w) + \lambda \nabla h_m(w)) = 
	M_{g_j}^{\lambda} (w-\lambda \nabla f_i(w) + \lambda \nabla 
	h_m(w)), \\
\quad \forall \lambda \in (0,\min\{\bar{\lambda},\varepsilon\}].
\end{gathered}
	\label{eq:Mg=Mgj}
	\end{equation}
	Let
	$(i,m)\in I\times M$ be any index such that $w\in 
	D_i(f)\cap D_m(h)$. If
	\begin{equation}
	z_{\lambda} \in 
	\prox_{\lambda g} (w - \lambda \nabla f_i(w) + \lambda \nabla 
	h_m(w)), \quad \lambda \in (0,\bar{\lambda}),
	\label{eq:z_lambda}		
	\end{equation}
	we have from \cref{eq:V_Lyapunov} that
	\begin{equation}
	\label{eq:Vw_minus_Vwstar}
	V(w) - \min_{z}\, V(z) \geq V(w) - V(z_{\lambda}) \geq \frac{1-\lambda 
		L}{2\lambda}\|w-z_{\lambda}\|^2.
	\end{equation}
	Taking $\D \coloneqq \{D_j(g): j\in J\}$ 
	and defining 
	$\phi_{\D}$ as in \cref{eq:indexfunction}, we 
	know from 
	\cref{lemma:invariantball} that there exists $\eta>0$ such that 
	$\phi_{\D}(z) \subset \phi_{\D} (w)$ for all $z\in 
	B(w,\eta) \cap \dom (g)$. Using \cref{eq:Vw_minus_Vwstar}, we can 
	find 
	$\varepsilon>0$ small enough so that $\|w-z_{\lambda}\| < \eta$ for 
	all $\lambda \in (0,\varepsilon]$. Now, fix $\lambda \in 
(0,\min\{\bar{\lambda},\varepsilon\}]$ 
	and let $j\in \phi_{\D}(z_{\lambda})$. Then $g(z_{\lambda}) = 
	g_j(z_{\lambda})$ and we have
	\begin{eqnarray*}
		&&
	M_g^{\lambda} (w - \lambda \nabla f_i(w) + \lambda \nabla
	h_m(w))\\
	& \overset{\cref{eq:z_lambda},\cref{eq:moreau}}{=} & 
	g_j(z_{\lambda}) + \frac{1}{2\lambda}\|z_{\lambda} - ( w - 
	\lambda 
	\nabla 
	f_i(w) + \lambda \nabla h_m(w)) \|^2 \nonumber \\
	& \overset{\cref{eq:moreau}}{\geq} & M_{g_j}^{\lambda} (w - \lambda \nabla 
	f_i(w) + \lambda 
	\nabla h_m(w)) \nonumber \\
	& \overset{\cref{lemma:moreau_prox} (a)}{\geq} & M_g^{\lambda} (w - \lambda \nabla f_i(w) + \lambda \nabla 
	h_m(w)). \nonumber 
	\end{eqnarray*}
	Since $z_{\lambda}\in \dom (g)$,
	$\phi_{\D}(z_{\lambda}) \subset \phi_{\D}(w)$ and thus
	$w\in D_j(g)$, proving \cref{eq:Mg=Mgj} and thus part (b).
	Finally, Part (c), follows immediately from
	\cref{eq:Vw_minus_Vwstar}.  
		\ifdefined\submit
		\qed
		\else
		\qedhere
		\fi
\end{proof}

One caveat of the local optimality condition given in 
\cref{thm:local_is_fixed} (b) is that 
a local 
minimum might not be a fixed point of $\Tpdmc$ when $\lambda 
\in (0,\bar{\lambda})$ but $\lambda 
\in (\varepsilon, 1/L]$ (see \cref{ex:local_fixed_critical}).
On the other hand, in search for 
global minima of 
\cref{eq:minconvex_optimization},  the 
above theorem provides an intuition that larger but permissible values of 
$\lambda$ 
must be chosen to avoid getting stuck at spurious local optima. 
Of course, from a 
numerical point of view, a larger stepsize is often also more desirable to 
obtain faster empirical convergence of the algorithms. 

A more standard necessary condition for optimality is \emph{criticality}.
We recall from \cite{WCP18} that $w$ is a \emph{critical point} of $f+g-h$ if 
\[0\in
\partial f(w) + \partial g(w) - \partial h(w).\]
Indeed, by \cref{assum:min_convex} (b) and (d), $f$ and $h$ are piecewise 
smooth functions in the sense of \cite[Definition 4.5.1]{FP03} and thus
locally Lipschitz continuous at any point  \citep[Lemma 
4.6.1 (a)]{FP03}. Consequently, we obtain from Exercise~10.10 of \cite{RW98} 
that $\partial (f+g-h)(w) \subset \partial f(w) + \partial g(w) - \partial 
h(w)$, where equality holds if $f$ and $h$ are differentiable at 
$w$. Hence, by \cite[Theorem 10.1]{RW98}, a local minimum of $f+g-h$ is a 
critical point. We now show that criticality is a tighter condition than being 
a fixed point.

\begin{theorem}
	\label{thm:fixed_is_critical}
	Suppose that \cref{assum:min_convex} holds and $D_{i}(f)$ is a regular 
	closed set, that is, $D_{i}(f) = \cl ({\rm int}(D_{i}(f)))$, for any $i\in 
	I$, where $\cl$ and ${\rm int}$ are respectively the closure and
	the interior of a set.
	Then any 
	fixed point of $\Tpdmc$ is a critical point of $f+g-h$.
\end{theorem}
\begin{proof}
	Let $w\in \Fix (\Tpdmc)$, say $w \in \prox_{\lambda g} 
	(w-\lambda \nabla f_i(w)+\lambda \nabla h_m (w))$ for some $(i,m)\in 
	I\times M$. Then $- \nabla f_i(w) + \nabla h_m(w) \in \partial g(w)$ by 
	\cite[Theorem 10.1]{RW98}. Since 
	$\nabla h_m(w) \in \partial h(w)$ by \cref{remark:consequences_assumption} 
	(d), it 
	suffices to show that $\nabla f_i(w) \in \partial f(w)$. Since $w\in 
	D_i(f)$ and we have from hypothesis that $D_{i}(f) = \cl ({\rm 
		int}(D_{i}(f)))$, there exists a sequence $\{w^k\}
	\subset {\rm int} (D_i(f))$ such that $w^k \to w$. Moreover, since 
	$f\equiv f_i$ on ${\rm int} 
	(D_i(f))$, $f$ is differentiable on ${\rm int} (D_i(f))$, and so 
	$\hat{\partial}f(w^k) = \{\nabla f_i(w^k)\}$ for all $k$. By the continuity 
	of 
	$\nabla f_i$, we then have $\nabla f_i(w^k) \to \nabla f_i(w)$ so that 
	$\nabla f_i(w) \in \partial f(w)$, as desired. 
		\ifdefined\submit
		\qed
		\else
		\qedhere
		\fi
\end{proof}

\begin{remark}
	\begin{enumerate}[(a)]
		\item \cref{thm:fixed_is_critical}, together with 
		\cref{thm:global_pdmc_full}, 
		guarantees that 
		\cref{eq:pdmc} is globally 
		 convergent to critical points of the objective 
		function. Note 
		that the assumption on $D_i(f)$ trivially holds when $\abs{I}=1$, as in 
		the 
		illustrative applications that we will see in 
		\cref{sec:affinefeasibility}.
		
		\item The fixed-point set $\Fix (\Tpdmc)$ may be 
		strictly contained in the set of critical points; for instance, see 
		\cref{ex:local_fixed_critical}.
	\end{enumerate}
\end{remark}

\modify{As we have seen in \cref{sec:general}, single-valuedness at a fixed
point is critical for global convergence.
The following property shows that when $|I| = 1$, namely when $f$ has
Lipschitz continuous gradient, $\Tps$ is single-valued for
a sufficiently small stepsize.}

\begin{proposition}
	\label{prop:nondegenerate_fixedpoints}
	\modify{Suppose \cref{assum:min_convex} holds.}
	If $w\in \Fix(T_{\rm PS}^{\lambda})$ for some \modify{$\lambda >
	0$},\cpsolved{$\lambda \leq 1/L$ is not used here, right?}
	then there 
	exists $\varepsilon>0$ such that $w\in \Fix(T_{\rm PS}^{\bar{\lambda}})$ 
	for all $\bar{\lambda} \in (0,\varepsilon]$. Moreover, $T_{\rm PS}^{\bar{\lambda}}$ is single-valued at $w$ for any $\bar{\lambda}\in (0,\varepsilon]$ if $|I|=1$. 
\end{proposition}
\begin{proof}
	Let $i\in I$ \modify{be} such that $w\in D_i(f)$ \modify{with}
	$w\in P_S(w-\lambda \nabla f_i(w))$. For each $j\in J$, denote
	$D_j \coloneqq \{ z\in \E: \dist(z,S)=\dist(z,R_j)\}$ and $\D
	\coloneqq \{ D_j : j\in J\}$. By \cref{lemma:invariantball}, there
	exists $\delta>0$ such that $\phi_{\D}(z)\subset \phi_{\D}(w)$ for
	all $z\in B(w,\delta)$.  Meanwhile, since $w\in  P_S(w-\lambda
	\nabla f_i(w))$ and each $R_j$ is convex, \modify{it follows from the
	definition of $P_S$ that} $w=P_{R_j}(w-\lambda \nabla f_i(w))$ for
	all $j\in \phi_{\D}(w)=\{ j\in J: w\in R_j\}$.\jhsolved{I skipped some arguments here, but we can put more details in the proof if you think that's better. To complete the argument, we note that $\|w-(w-\lambda \nabla f_i(w))\| \leq \|z-(w-\lambda \nabla f_i(w))\| $ for any $z\in S$. In particular, $\|w-(w-\lambda \nabla f_i(w))\| \leq \|P_{R_j}(w-\lambda \nabla f_i(w))-(w-\lambda \nabla f_i(w))\| $. But since $w\in R_j$, then $ \|P_{R_j}(w-\lambda \nabla f_i(w))-(w-\lambda \nabla f_i(w))\| \leq  \|w-(w-\lambda \nabla f_i(w))\|$. Since $R_j$ is closed and convex (which implies uniqueness of the projection), then $P_{R_j}(w-\lambda \nabla f_i(w)) = w$. This is the reason why we can consider degenerate fixed points as well.} Hence, $-\nabla 
	f_i(w) \in N_{R_j}(w)$, the normal cone to $R_j$ at 
	$w$, for all $j\in \phi_{\D}(w)$. 
	Now, set $\varepsilon = 
	\delta/(2\|\nabla f_i(w)\|)$ and take any $\bar{\lambda} \in 
	(0,\varepsilon]$. We have $\|w-(w-\bar{\lambda}\nabla f_i(w))\| < \delta$, so that
	$\phi_{\D}(w-\bar{\lambda}\nabla f_i(w))\subset \phi_{\D}(w)$. Then 
	\[ P_S(w-\bar{\lambda}\nabla f_i(w)) = \bigcup _{j\in \phi_{\D}(w-\bar{\lambda}\nabla f_i(w))} P_{R_j}(w-\bar{\lambda}\nabla f_i(w))=w ,\]
 where the last equality holds since $\phi_{\D}(w-\bar{\lambda}\nabla f_i(w))\subset \phi_{\D}(w)$ and $-\nabla 
	f_i(w) \in N_{R_j}(w)$ for all $j\in \phi_{\D}(w)$. It follows that $w\in T_{\rm PS}^{\bar \lambda}(w)$.
	\modify{If $|I|=1$, we further obtain $ T_{\rm PS}^{\bar
	\lambda}(w)=w$.}
		\ifdefined\submit
		\qed
		\else
		\qedhere
		\fi
\end{proof}
We demonstrate by an example the relationship among the 
sets of global/local minima, critical points of 
\cref{eq:min_alternatingprojections}, and fixed points of $T_{\rm MAP}^{\lambda}$.

\begin{example}\label{ex:local_fixed_critical}
	Consider \cref{eq:min_alternatingprojections} with $S_1 = \{ 
	(a,1):a\in \Re\}$ and $S_2=\{(a,b):a,b\geq 0, ab=0 \}$. 
	The set of local minima of 
	\cref{eq:min_alternatingprojections} and $\Fix (T_{\rm 
		MAP}^{\lambda})$ are given respectively
	by $S^
	* \coloneqq \{ (0,1)\}\cup \{ (t,0) : t>0\}$ and $ \Fix (T_{\rm 
		MAP}^{\lambda}) = \{(0,1)\} \cup \{ 
	(t,0):t\geq \lambda\}$ for any $\lambda\in (0,1]$.
	Clearly, $S^*$ is a subset of $C^*$, the set of critical points
	of $f+\delta_{S_2}$. 
	Thus, we have $\{(0,1)\} = S_1\cap S_2 \subsetneq 
	\Fix (\Tmap^{\lambda})\subsetneq C^*$ for any $\lambda \in (0,1]$. It 
	is then not difficult to verify the claims of \cref{thm:local_is_fixed} and 
	\cref{thm:fixed_is_critical}. 
	Moreover, observe that each point in $\Fix(T_{\rm 
		MAP}^{\lambda}) \subset 
	\Fix(\Tmap^{\bar{\lambda}})$ whenever $ \bar{\lambda} \leq \lambda$ and $\Tmap^{\bar{\lambda}}$ is single-valued on $\Fix(T_{\rm 
		MAP}^{\lambda})$ if $\bar{\lambda}<\lambda$, 
	demonstrating 
	\cref{prop:nondegenerate_fixedpoints}. 
\end{example}

\subsection{Acceleration schemes for PDMC}
\label{sec:acce_pdmc}
\begin{wrapfigure}{r}{0.27\textwidth}
	\centering 
	\includegraphics[scale=.75]{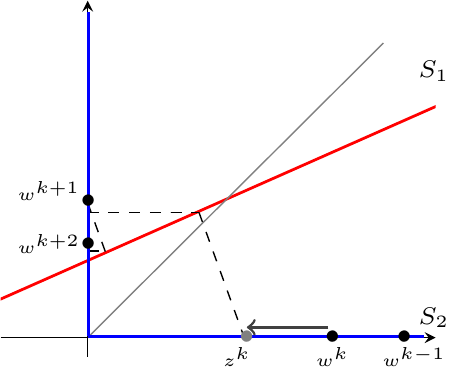}
	\caption{Illustration of accelerated MAP
		for a two-set feasibility problem with union convex sets.
		$P_{S_2}$ is multivalued on the gray line.}
	\label{fig:motiv}
\end{wrapfigure}
In this section, We follow the scheme described in 
\cref{alg:accelerated_FPA} and the discussion on component
identification in \cref{remark:component} to propose two such
acceleration techniques.
The first one is motivated by the following example.
\begin{example}
	\label{ex:acce_motive}
	Let $\E = \Re^2$, $S_1$ be any straight line with a positive slope, and
	$S_2=A\cup B$, where $A \coloneqq \{(a,0): a \geq 0\}, B \coloneqq \{(0,b): 
	b \geq
	0\}$;
	see
	\cref{fig:motiv}. Consider
	\cref{eq:minconvex_optimization} with $f(w) = \frac{1}{2} \dist 
	(w,S_1)^2$, 
	$g(w)= \delta_{S_2} = \min \{ \delta_A(w), \delta_B(w)\}$, and $h \equiv 
	0.$ Then it can be shown that the PDMC 
	iterates with stepsize $\lambda = 1$ coincide with the MAP
	iterates; see also \cref{sec:map}. Notice that 
	this algorithm generates points confined in the union convex set $S_2$. To 
	speed up the convergence of the algorithm to the solution, we conduct 
	extrapolation if two consecutive iterates lie on the same convex set. As 
	illustrated in \cref{fig:motiv}, if $w^{k-1}$ and $w^k$ both 
	lie on $A$ or $B$, we extrapolate along the direction $w^k - w^{k-1}$ 
	to get an intermediate point $z^k$ before conducting alternating 
	projections to obtain $w^{k+1}$. 
	Intuitively, the iterates generated by this procedure tend to get closer to 
	$S_1$ faster than when (non-accelerated) MAP only is used. 
\end{example}

Inspired by the above example,
we propose to proceed with the 
extrapolation step in Step 1 of \cref{alg:accelerated_FPA} only when two consecutive iterates ``activate'' the same 
components in $f$, $g$ and $h$. Formally, let
\begin{equation}
	\hat{\D} \coloneqq \{  D_i(f) \cap 
	D_j(g)\cap D_m(h) : (i,j,m)\in I\times J\times M\}		
	\label{eq:Dhat}
\end{equation}
and define
\begin{equation}
	\chi_k \coloneqq \begin{cases}
		1 & \text{if}~\phi_{\hat{\D}}(w^k) \cap 
		\phi_{\hat{\D}}(w^{k-1}) \neq \emptyset \text{ and}~k\geq 
		1, \\
		0 & \text{otherwise},
	\end{cases}
	\label{eq:chi_k}
\end{equation}
where  $\phi_{\hat \D}$ is defined in \cref{eq:indexfunction}.
Then, as summarized in \cref{alg:PDMC_accelerated}, we simply replace the 
step $p_k$ in Step 1 
of \cref{alg:accelerated_FPA} by $\chi_k p_k$ to take into account the 
described restriction.
It is clear that global subsequential convergence of
\cref{alg:PDMC_accelerated} to a fixed point of $\Tpdmc$ directly
follows from \cref{thm:global_acce}. 

\begin{algorithm}[tb]
	Let $V=f+g-h$. Choose 
	$\sigma>0$, $\lambda \in (0,1/L]\cap (0,\bar{\lambda})$, and $w^0 \in \E$. 
	Set $w^{-1} = w^0$ and $k=0$.
	\begin{description}
		\item[Step 1.] Set $z^k = w^k + t_kp^k$, where $p^k = \chi_k (
		w^k-w^{k-1})$, $\chi_k$ is given by \cref{eq:chi_k}, and
		$t_k\geq 0$ satisfies \cref{eq:descent_acce_fpa}.
		\item[Step 2.] Set $w^{k+1} \in \Tpdmc(z^k)$,
		$k=k+1$, and go back to Step 1. 
	\end{description}
	\caption{Accelerated PDMC algorithm
		for \cref{eq:minconvex_optimization}.}
	\label{alg:PDMC_accelerated}
\end{algorithm}

In the same spirit as \cref{remark:component}, applying
\cref{lemma:invariantball} to \cref{eq:Dhat} suggests that latter
iterates of the \cref{eq:pdmc} algorithm indicate which components of
the objective function are activated by a fixed point. Using this
observation, we propose to identify and safeguard the activated
component by checking consecutive component changes in \cref{alg:PDMC_plus}.
Our algorithm has a spirit similar to the heuristics for manifold
identification in \citep{YSL20a,LeeCP20,LeeW12a}
but is with theoretical tools thoroughly different from these works.

%

\begin{algorithm}[tb]
	Choose  $w^0\in \E$, $N\in \mathbb{N}$. Set Unchanged $= 0$, $k=0$.
	\begin{description}
		\item[Step 1.] Set Unchanged $= $ $\chi_k$(Unchanged + 1), where 
		$\chi_k$ is given by \cref{eq:chi_k}.
		
		\item[Step 2.] Compute $w^{k+1}$ according to the following
			rules:
		\begin{itemize}
			\item[2.1.] If Unchanged $<N$: set $w^{k+1}\in T_{\rm 
				PDMC}^{\lambda}(w^k)$.
			\item[2.2.] If Unchanged $=N$: set Unchanged = $-1$, pick $(i,j,m)\in
			\phi_{\hat{\D}} (w^k)$, and solve 
			\begin{equation*}
				w^{k+1}\in \argmin_{z\in \E} f_i(z)+g_j(z)-h_m(z).
			\end{equation*}
		\end{itemize}
		\item[Step 3.] Terminate if $w^{k+1} = w^k$; otherwise
				set $k = k+1$ and go back to Step 1.
	\end{description}
	\caption{PDMC with component 
		identification for \cref{eq:minconvex_optimization}.}
	\label{alg:PDMC_plus}
\end{algorithm}

\section{Affine-union convex set feasibility 
problems}\label{sec:affinefeasibility}

In this section,
we establish global convergence of several algorithms for solving
\cref{eq:feasibilityproblem} involving an affine set
	\begin{equation}
		S_1 = \{ w\in \Re^q : Aw = b\},
		\label{eq:S1_general}
	\end{equation}
where $A\in \Re^{m\times q}$ is a matrix with full row rank, and a
union convex set $S_2$. 
Specifically, we consider the sparse affine feasibility problem and a 
feasibility reformulation of the linear complementarity problem in 
\cref{sec:SAFP,sec:lcp}, respectively.
\modify{The results for LCPs are then applicable to GAVE following
	\cite{ACT21} as discussed in (V) of \cref{sec:contributions}.}
Recall that in 
general, the feasibility problem \cref{eq:feasibilityproblem}
can be reformulated as an optimization problem, either as 
\cref{eq:min_averagedprojections} or
\cref{eq:min_alternatingprojections}. Other than these reformulations, the 
affine structure of $S_1$ given by \cref{eq:S1_general} enables recasting the feasibility problem as
	\begin{align}
		\min_{w\in \E}\, &\frac{1}{2}\|Aw - b\|^2 +
		\frac{1}{2}\dist(w,S_2)^2,\quad \text{ or }
		\label{eq:opt_sparseaffine}\\
		\min_{w\in \E}\, &\frac{1}{2}\|Aw-b\|^2 + \delta_{S_2} (w).
		\label{eq:opt_sparseaffine2}
	\end{align}	
To unify the analyses of algorithms for these four optimization reformulations, 
we first note that the projection onto 
$S_1$ is given by $P_{S_1}(w) = w - A^{\dagger}(Aw-b)$ \cite[Lemma 4.1]{BK04},
where, $A^{\dagger}$ is the Moore-Penrose inverse of $A$, given by 
$A^{\dagger}  = A^\T (AA^\T)^{-1}$ since $A$ has full row rank. With this, we 
have
	\begin{equation*}
		\dist (w,S_1)^2  =
		\|A^{\dagger}(Aw-b)\|^2 
		 =  w^\T A^\T \hat Q Aw - 2w^\T A^\T \hat Q b + b^\T  \hat Q b,
		\label{eq:f_sparseaffine}
	\end{equation*}
where $\hat Q\coloneqq (AA^\T)^{-1}$. By denoting 
	\begin{equation}
		f_Q(w) \coloneqq \frac{1}{2}w^\T A^\T Q Aw - w^\T A^\T Q b + 
		\frac{1}{2}b^\T  Q b ,
		\label{eq:f_Q}
	\end{equation}
	we get
$f_Q(w) = \|Aw-b\|^2/2$ if $Q=\Id$, and $f_Q(w) = \dist(w,S_1)^2/2$ if $Q=(AA^\T)^{-1}$.
Thus, we may unify the convergence analyses 
of algorithms for \cref{eq:min_averagedprojections,eq:opt_sparseaffine} through varying
$Q$ in
\begin{equation}
	\min _{w\in \Re^n} f_Q(w) + \frac{1}{2}\dist(w,S_2)^2 ,
	\label{eq:unified_sparseaffine}
\end{equation}
and similarly for \cref{eq:min_alternatingprojections,eq:opt_sparseaffine2}, we may
consider
\begin{equation}
	\min _{w\in \Re^n} f_Q(w) + \delta_{S_2}(w).
	\label{eq:unified_sparseaffine2}
\end{equation}

Note that $f_Q$ is a convex function with gradient 
	\begin{equation}
		\nabla f_Q(w) = A^\T Q (Aw-b),
		\label{eq:grad_f_Q}
	\end{equation}
	which is Lipschitz continuous with parameter
	\begin{equation}
		L_Q = \begin{cases}
			1 & \text{if}~Q=(AA^\T)^{-1}, \\
			\|A\|^2 & \text{if}~Q=\Id.
		\end{cases}
		\label{eq:LQ}
	\end{equation}
Moreover, we have the following:
	\begin{enumerate}[(i)]
		\item As noted in \cref{sec:mavep}, $g$ and $h$
		given in \cref{eq:dist2_dc} 
		satisfy  \cref{assum:min_convex} (c) and (d) since $S_2$ is a union 
		convex set. By using this decomposition in \cref{eq:unified_sparseaffine} and then applying
		\cref{eq:pdmc}, we get
			\begin{equation}
				\Tpdmc (w^k) = \frac{1}{1+\lambda} 
				\left(w^k - \lambda \nabla f_Q(w) + 
				\lambda P_{S_2}(w^k)\right).
				\label{eq:pdmc_affine_ucs}
			\end{equation}
	
		\item A direct application of \cref{eq:fb} to 
		\cref{eq:unified_sparseaffine} with $g = \dist (\cdot, 
		S_2)^2/2$
		leads to
			\begin{equation}
				\Tfb (w^k)  = 
				\frac{\lambda}{1+\lambda}P_{S_2}(w^k - \lambda \nabla 
				f_Q(w^k)) + \frac{1}{1+\lambda}(w^k - \lambda \nabla f_Q(w^k)).
				\label{eq:fb_affine_ucs}
			\end{equation}
		
		\item The PS algorithm \cref{eq:PS} for solving 
		\cref{eq:unified_sparseaffine2} is given by
		\begin{equation}
			 T_{\rm PS}^{\lambda} (w^k) = P_{S_2}(w^k - \lambda 
			\nabla 
			f_Q(w^k)).
			\label{eq:ps_affine_ucs}
		\end{equation}
	\end{enumerate}
When
$Q=(AA^\T)^{-1}$, the operators $\Tpdmc$, 
$\Tfb$ and $\Tps$ above respectively coincide with 
$\Tmavep^{\lambda}$, $\Tmarp^{\lambda}$ and $\Tmap^{\lambda}$
presented in \cref{sec:feasibility}.

\begin{remark}
	Except for \cref{assum:min_convex} (e), 
	all the other assumptions are satisfied. Together with the convexity of $f_Q$ and \cref{remark:specialcases} (a), we obtain from \cref{thm:global_pdmc_full} that the algorithms \cref{eq:pdmc_affine_ucs,eq:fb_affine_ucs,eq:ps_affine_ucs} are globally convergent to fixed points if we can show that the objective functions are coercive. 
\end{remark}

\subsection{Sparse affine feasibility}
\label{sec:SAFP}
We consider the \emph{sparse affine feasibility problem} (SAFP), which 
involves 
solving \cref{eq:feasibilityproblem} with \cref{eq:SAFP},
where $0<s\leq n$, $A\in \Re^{m\times n}$ has full row rank and $b\in \Re^m$.
\citet{HLN14} have shown
that $S_2 = A_s$ can be decomposed as
\begin{equation}
	A_s = \bigcup _{\iota \in \I_s} R_{\iota}, \quad
	\I_s \coloneqq \{ \iota \subset [n] : \iota \text{ has 
	}s~\text{elements}\},\quad R_{\iota} \coloneqq \Ran \left(\Id_{:,\iota}\right),
	\label{eq:I_Ri_sparse}
\end{equation} 
so $S_2$ is indeed a union convex 
set
and the projection onto $S_2$ is given by
\begin{equation*}
	P_{S_2}(w) = \{ P_{R_{\iota}}(w): \iota \in \I_s \text{ such that } \min 
	_{j\in \iota} |w_j| \geq \max _{j\in \iota^c} |w_j| \}.
\end{equation*}
In turn, we can use the algorithms \crefrange{eq:pdmc_affine_ucs}{eq:ps_affine_ucs} to
solve the sparse affine feasibility problem.

We now show that these algorithms are globally convergent under conditions significantly weaker than those used in prior works \cite{BT11,HLN14}.
To establish our convergence results, we 
note the following simple but useful lemma. 


\begin{lemma}
	\label{lemma:rank}
	Let $A\in \Re^{m\times n}$ and $Q \in \Re^{m\times m}$ be with $\rank (A) = 
	\rank (Q) = m$, and let $\Lambda \subset [n]$. If $\rank (A_{:,
	\Lambda}) = \abs{\Lambda}$, then $ \rank 
	((A^\T Q A)_{:,\Lambda}) = \abs{\Lambda}$. Consequently, 
	$\lambda_{\min}((A^\T Q A)_{:,\Lambda}^\T (A^\T Q A)_{:,
	\Lambda})>0$. 
\end{lemma}
\begin{proof}
	Let $E = \Id_{:,\Lambda}$, then 
	$A_{:,\Lambda} = AE$ and $(A^\T 
	Q A)_{:,\Lambda} = A^\T QAE$. With the rank assumptions, the 
	result immediately follows. 
		\ifdefined\submit
		\qed
		\else
		\qedhere
		\fi
\end{proof}

\begin{theorem}
	\label{thm:global_pdmc_fb_sparse}
	Consider \cref{eq:feasibilityproblem} with \cref{eq:SAFP}.
	Let $A\in \Re^{m\times n}$ be of full row rank, $Q\in \{ (AA^\T), \Id\}$, and 
	$f_Q$ and $L_Q$ 
	be given by \cref{eq:f_Q} and \cref{eq:LQ}, respectively. Suppose there 
	exists $\nu_s>0$ such 
	that 
	\begin{equation}
		\nu_s \|w\|^2 \leq \|Aw\|^2, \quad \forall w\in A_s	.
		\label{eq:SRIP_half}	
	\end{equation}
	Then any sequence $\{w^k\}$ generated by the PDMC algorithm 
	\cref{eq:pdmc_affine_ucs} with $\lambda \in (0,1/L_Q]$ has an accumulation point $w^*\in \Fix (\Tpdmc)$, and full sequence convergence holds if $\Tpdmc$ is single-valued at $w^*$. The same conclusion holds for a sequence generated by the FB algorithm \cref{eq:fb_affine_ucs} with stepsizes $\lambda \in (0,1/L_Q)$. 
\end{theorem} 
\begin{proof}
	Using \cref{thm:global_pdmc_full}, 
it suffices to prove that $V_{\iota} 
	\coloneqq f_Q+\dist 
	(\cdot,R_{\iota})^2/2$ is 
	coercive for all $\iota \in \I_s$.
	That is, given any
	$\{w^k\}$ such that $\|w^k\|\to \infty$, we need to show 
	that $V_{\iota}(w^k)\to \infty$. Suppose otherwise, then 
	$\{V_{\iota}(w^k)\}$ must have a bounded subsequence, and 
	we assume without loss of generality that the whole sequence is 
	bounded. Since
		\begin{equation}
		V_{\iota}(w^k)  =  f_Q (w^k)+ 
		\frac{1}{2}\|w^k - 
		P_{R_{\iota}}(w^k)\|^2
		=  f_Q(w^k) + 
		\frac{1}{2}\|(w^k)_{\iota^c}\|^2,
		\label{eq:lyapunov_mave_coercive_sparse}
	\end{equation} 
	$\{(w^k)_{\iota^c}\}$ must be bounded, and hence
	$\|(w^k)_{\iota}\|\to \infty$ since we are given that $\|w^k\|\to \infty$. 
	Meanwhile, for $Q=(AA^\T)^{-1}$, we have 
	\begin{align}
		\nonumber  f_Q(w^k) & = \frac{1}{2}	\|(A^{\dagger}A)_{:,\iota} 
		(w^k)_{\iota} + (A^{\dagger}A)_{:,\iota^c} 
		(w^k)_{\iota^c} - A^{\dagger}b\|^2 
		\\ \nonumber
		& \geq
		\frac{1}{2} 	\lambda_{\min} \left( (A^{\dagger}A)_{:,\iota} ^\T 
			(A^{\dagger}A)_{:,\iota} \right) \|(w^k)_{\iota}\|^2 
			+\frac{1}{2} \| (A^{\dagger}A)_{:,\iota^c} 
			(w^k)_{\iota^c} - A^{\dagger}b\|^2  
			\\ 
		\label{eq:distS1>=_sparse}
			&\quad - \|(A^{\dagger}A)_{:,\iota} \| \cdot 
			\|(w^k)_{\iota}\| 
			\cdot 
			\| (A^{\dagger}A)_{:,\iota^c} 
			(w^k)_{\iota^c} - A^{\dagger}b \|
	\end{align}
	On the other hand, if $Q=\Id$, we obtain by a similar computation that 
		\ifdefined\submit 
		\begin{multline}
		f_Q(w^k) \geq \frac{1}{2}\lambda_{\min}(A_{:,\iota}^\T A_{:,
			\iota})\|(w^k)_{\iota}\|^2
		+ \frac{1}{2} \| A_{:, \iota^c} (w^k)_{\iota^c} -b\|^2
		\\ -\|A_{:,\iota}\| \cdot 	
		\|(w^k)_{\iota}\| 
		\cdot \| A_{:,\iota^c}  (w^k)_{\iota^c} - b \|.
		\label{eq:distS1>=_sparse2}
		\end{multline}
		\else 
		\begin{equation}
			f_Q(w^k) \geq \frac{1}{2}\lambda_{\min}(A_{:,\iota}^\T A_{:,
			\iota})\|(w^k)_{\iota}\|^2
		+ \frac{1}{2} \| A_{:, \iota^c} (w^k)_{\iota^c} -b\|^2
			-\|A_{:,\iota}\| \cdot 	
			\|(w^k)_{\iota}\| 
		\cdot \| A_{:,\iota^c}  (w^k)_{\iota^c} - b \|.
		\label{eq:distS1>=_sparse2}
		\end{equation}
		\fi 
	By \cref{eq:SRIP_half}, it is clear that $\rank (A_{:,
	\iota})=\abs{\iota}$. Thus, by \cref{lemma:rank}, $	\lambda_{\min} \left( 
	(A^{\dagger}A)_{:,\iota} ^\T 
(A^{\dagger}A)_{:,\iota} \right) >0$ and $\lambda_{\min}(A_{:,\iota}^\T 
A_{:,\iota})>0$. Letting $k\to \infty$ in \cref{eq:distS1>=_sparse} 
and \cref{eq:distS1>=_sparse2}, we then obtain that $f_Q(w^k) \to \infty$, and so 
by \cref{eq:lyapunov_mave_coercive_sparse}, $V_{\iota}(w^k)\to \infty$, which 
is a contradiction. Hence, $V_{\iota}$ is coercive, as desired.
		\ifdefined\submit
		\qed
		\else
		\qedhere
		\fi
\end{proof}

We now show $Q$-linear convergence of the PS algorithm for solving
\cref{eq:unified_sparseaffine2}.

\begin{theorem}
	\label{thm:global_ps_affine}
	Consider the setting of \cref{thm:global_pdmc_fb_sparse}.
	Then any sequence $\{w^k\}$ generated by
	\cref{eq:ps_affine_ucs} with $\lambda \in (0,1/L_Q)$ has an accumulation point $w^*\in \Fix (\Tps)$, and if $\Tps$ is single-valued at $w^*$, then the algorithm converges to
	$w^*$ at a local $Q$-linear rate.
\end{theorem} 
\begin{proof}
	Given any $\iota\in \I_s$ and any sequence 
	$\{w^k\} $ that lies in 
	$R_{\iota}$ such that $\|w^k\|\to \infty$, clearly $(w^k)_{\iota^c} = 0$ 
	for all $k$. Consequently, by noting that $\lambda_{\min}\left( 
	(A^{\dagger}A)_{:,
		\iota}^\T 
	(A^{\dagger}A)_{:,\iota} \right) $ and $\lambda_{\min}\left( (A^\T 
	A)_{:,\iota}^\T (A^\T 
	A)_{:,\iota} \right)$ are both strictly positive from the 
	proof of \cref{thm:global_pdmc_fb_sparse}, we obtain from 
	\cref{eq:distS1>=_sparse,eq:distS1>=_sparse2} that $f_Q(w^k)\to \infty$. Thus, $f_Q$ 
	is coercive over $R_{\iota}$, showing that \cref{assum:min_convex} (e) is 
	fulfilled. To complete the proof, by \cref{prop:linear_ps}, it suffices to show 
	that $Id-\lambda \nabla f_Q$ is a contraction over 
	$R_{\iota}$. Suppose that $Q=(AA^\T)^{-1}$ and $w,w'\in 
	R_{\iota}$, then 
	\begin{align}
		\| \nabla f_Q(w) - \nabla f_Q(w')\|^2 &
		=  \| (A^{\dagger}A)_{:,\iota} (w-w')_{\iota}\|^2 \nonumber  
		\\
		& \geq  \lambda_{\min}\left( (A^{\dagger}A)_{:,\iota}^\T 
		(A^{\dagger}A)_{:,\iota} \right)  
		\|(w-w')_{\iota}\|^2  \nonumber  \\
		 & =  \lambda_{\min}\left( (A^{\dagger}A)_{:,\iota}^\T 
		(A^{\dagger}A)_{:,\iota} \right)  
		\|w-w'\|^2.
		 \label{eq:Lipconst1}
	\end{align}
	Similarly, for $Q=\Id$ and $w,w'\in R_{\iota}$, we have
	\begin{equation}
		\| \nabla f_Q(w) - \nabla f_Q(w')\|^2 
		\geq \lambda_{\min}\left( (A^\T A)_{:,\iota}^\T (A^\T 
		A)_{:,\iota} \right)\|w-w'\|^2.
		\label{eq:Lipconst2}
	\end{equation}
	By \cref{eq:LQ} and the Lipschitz continuity of $\nabla f_Q$,
	\cref{eq:Lipconst1,eq:Lipconst2} further lead to
	\[ \| \left( w - \lambda \nabla f(w)\right) - \left(w' - \lambda \nabla 
	f(w') \right)\| \leq \kappa_{\iota} \|w-w'\|,  \quad \forall w,w'\in 
	R_{\iota} ,\]
	where 
	\[\kappa_{\iota}^2 = \begin{cases}
		1+(\lambda^2-2\lambda) \lambda_{\min}\left( (A^{\dagger}A)_{:,
			\iota}^\T 
		(A^{\dagger}A)_{:,\iota} \right)   & \text{if }Q=(AA^\T)^{-1}, \\
		1+(\lambda^2-2\lambda \|A\|^{-2})  \lambda_{\min}\left( (A^\T 
		A)_{:,\iota}^\T (A^\T ,
		A)_{:,\iota} \right) & \text{if }Q=\Id.
	\end{cases}\]
	Since the second term is negative for $\lambda \in (0,1/L_Q)$, $\kappa_{\iota}\in
	[0,1)$ and the conclusion follows. 
		\ifdefined\submit
		\qed
		\else
		\qedhere
		\fi
\end{proof}

\subsection{Linear complementarity problems and general absolute value
equations}\label{sec:lcp}
We now turn our attention to  the \emph{linear complementarity problem 
(LCP)} described in \cref{eq:LCP} and consider the feasibility problem
reformulation \cref{eq:feasibilityproblem} with
\cref{eq:S1_lcp}.
We note that $A$ given in \cref{eq:S1_lcp} has full row rank for any matrix 
$M$. 
Observe that $S_1$ is an affine set 
and $S_2$  also has a sparsity structure
such that $S_2 \subset
A_n$. However, $S_2$ has additional 
properties that distinguishes it from $A_n$,  including the nonnegativity of 
its vectors as well as the complementarity between $(w_1,\dots,w_n)$ and 
$(w_{n+1},\dots,w_{2n})$. 
	
As 
shown in~\cite[Proposition 2.2]{ACT21}, $z\in P_{S_2}(w)$ if and 
only if
\begin{equation*}\label{eq:PC2}
(z_j,z_{n+j}) \in \begin{cases}
\{ (0,(w_{n+j})_+)\} & \text{if}~ w_j<w_{n+j}, \\
\{ ((w_{j})_+,0)\} & \text{if}~ w_j>w_{n+j}, \\
\{ (0,(w_{n+j})_+), ((w_{j})_+,0)\} & \text{if}~ w_j=w_{n+j},\\
\end{cases} \quad \forall j \in [n].
\end{equation*}
We also get from 
\cite[Section 3.1]{ACT21} that
$S_2$ can be decomposed as a union of closed convex sets:
\begin{equation}
	S_2 = \bigcup _{\iota \in \I} R_{\iota}, \quad \text{ where }
	R_{\iota} \coloneqq 
\Ran(\Id_{:,\iota}) \cap \Re^{2n}_+,
	\label{eq:S2unionconvexlcp}
\end{equation}
where $\Re^{2n}_+$ denotes the set of nonnegative vectors in $\Re^{2n}$,
and $\I$ is the set of all  
$\iota \subset [2n]$ 
expressible as $\iota = \Lambda_1 
\cup \Lambda_2$ for some $\Lambda_1 \subset [n]$ and $\Lambda_2 = \{ n+j : 
j \in [n], j \notin \Lambda_1\}$.
It is also clear
that for any $\iota \in \I$ and $w\in \Re^{2n}$, the projection 
$z$ of $w$ onto $R_{\iota}$ is given by
\begin{equation}\label{eq:PRtau}
(z_j, z_{n+j}) = \begin{cases}
((w_j)_+,0) & \text{if}~j\in \iota \\
(0,(w_{n+j})_+) & \text{if}~j \notin \iota 
\end{cases}, \quad \forall j\in [n].
\end{equation}

\subsubsection{LCPs involving nondegenerate and $P$-matrices}
\label{sec:lcp_nondegenerate}
In \cref{sec:SAFP}, the
condition \cref{eq:SRIP_half} relaxed from SRIP was used to establish
the convergence of 
the algorithms \crefrange{eq:pdmc_affine_ucs}{eq:ps_affine_ucs}. For the feasibility reformulation of LCP, a property similar to \cref{eq:SRIP_half} 
can be obtained 
through assumptions on $M$ that are conventional in the LCP literature. 

\begin{definition}
	A matrix $M\in \Re^{n\times n}$ is said to be \emph{nondegenerate} 
	if all of its principal minors are 
	nonzero.
\end{definition}

\begin{lemma}
	\label{lemma:SRIP_lcp}
	Let $M\in \Re^{n\times n}$ be a nondegenerate matrix, and
	$A=\begin{bmatrix}
	M & -\Id
	\end{bmatrix}$. Then
	for $S_2$ given by \cref{eq:S1_lcp}, there exists $\nu>0$ such that 
		\[\nu \|w\|^2 \leq \| A w\|^2, \quad \forall w\in S_2\cup (-S_2).\]
\end{lemma}
\begin{proof}
	If $\iota \in \I$ and $w\in R_{\iota}\cup (-R_{\iota})$,
		\begin{equation*}
			\|Aw\|^2 = \|A_{:,\iota} w_{\iota}\|^2 \geq 
			\lambda_{\min} (A_{:,\iota}^\T A_{:,
				\iota})\|w_{\iota}\|^2 = \lambda_{\min} 
				(A_{:,
				\iota}^\T A_{:,
				\iota})\|w\|^2 = \nu_{\iota} 
				\|w\|^2,
		\end{equation*}
	where $ \nu_{\iota}  
	=\lambda_{\min} (A_{:,
	\iota}^\T A_{:,\iota})$.  Meanwhile, nondegeneracy 
	of $M$ implies that the square matrix $A_{:,\iota}$ is nonsingular by 
	\cite[Lemma 2.10]{ACT21}, so $\nu_{\iota} > 0$. By taking $\nu = 
	\min_{\iota \in \I}\nu_{\iota}$ and noting
	\cref{eq:S2unionconvexlcp}, we get the desired inequality.
		\ifdefined\submit
		\qed
		\else
		\qedhere
		\fi
\end{proof}
With the above lemma, we can easily obtain convergence for
\cref{eq:pdmc_affine_ucs,eq:fb_affine_ucs,eq:ps_affine_ucs} on the feasibility reformulation of LCPs.

\begin{theorem}
	\label{thm:global_lcp_nondegenerate}
	Let $M\in \Re^{n \times n}$ be a nondegenerate matrix, $b\in \Re^n$, $A=\begin{bmatrix}
		M & -\Id
	\end{bmatrix}$, $Q\in \{ (AA^\T)^{-1},\Id\}$, and $f_Q$ and $L_Q$ be 
	given by \cref{eq:f_Q} and \cref{eq:LQ}, respectively, then
	for \cref{eq:feasibilityproblem} with \cref{eq:S1_lcp}:
		\begin{enumerate}[(a)]
			\item  Any 
			sequence $\{w^k\}$ generated by
			\cref{eq:pdmc_affine_ucs} with $\lambda \in (0,1/L_Q]$ has an accumulation point $w^*\in \Fix (\Tpdmc)$, and full sequence convergence holds if $\Tpdmc$ is single-valued at $w^*$.
			
			\item Any sequence $\{w^k\}$ generated by
			\cref{eq:fb_affine_ucs} with $\lambda \in (0,1/L_Q)$ has an accumulation point $w^*\in \Fix (\Tfb)$, and full sequence convergence holds if $\Tfb$ is single-valued at $w^*$.
			
			\item Any 
			sequence $\{w^k\}$ generated by
			\cref{eq:ps_affine_ucs} with $\lambda \in (0,1/L_Q)$ has an accumulation point $w^*\in \Fix (\Tps)$, and full sequence convergence holds if $\Tps$ is single-valued at $w^*$. Moreover, the convergence rate is locally linear.
		\end{enumerate}
\end{theorem} 
\begin{proof}
	Let $\iota\in\I$.
	We define $V_{\iota} \coloneqq f_Q + \dist(\cdot, 
	R_{\iota})^2/2$, and see from \cref{eq:PRtau} that 
		\begin{equation}
			V_{\iota}(w) = f_Q (w)+ \frac{1}{2}\|w_{\iota} - 
			[w_{\iota}]_+\|^2 + \frac{1}{2} 
			\|w_{\iota^c}\|^2 
			\geq f_Q(w) + \frac{1}{2} \|w_{\iota^c}\|^2.
			\label{eq:Vcoercive_nondegenerate}
		\end{equation}
	Using \cref{eq:Vcoercive_nondegenerate} and \cref{lemma:SRIP_lcp}, the rest of the proof follows from 
	arguments analogous to those 
	in the proofs of \cref{thm:global_pdmc_fb_sparse,thm:global_ps_affine}. 
		\ifdefined\submit
		\qed
		\else
		\qedhere
		\fi
\end{proof}

For a special class of nondegenerate matrices, known as $P$-matrices,
we can obtain finer results.
\begin{definition}
	\label{def:P}
	A matrix $M\in \Re^{n\times n}$ is said to be a \emph{$P$-matrix} if all of
	its principal minors are positive. 
\end{definition}

It is known that \cref{eq:LCP} has a unique solution for any $b\in 
\Re^n$ when $M$ is a $P$-matrix \citep[Theorem 3.3.7]{CPS92}. Consequently, 
$S_1\cap S_2$ contains a single point when $M$ is a $P$-matrix for
$S_1$ and $S_2$ defined in \cref{eq:S1_lcp}. Some important 
applications of LCP involving $P$-matrices can be found in \cite{Schafer04}. 
For $P$-matrices, we derive the following nice result on the
characterization of fixed points. The proof of this result is quite
technical and heavily relies on a special property of $P$-matrices
described in \cref{lemma:Pmatrix}, and thus we defer it to
\cref{app:fixed_Pmatrix}.

\begin{theorem}
	\label{thm:fixedpoints_Pmatrix}
	Consider the setting of \cref{thm:global_lcp_nondegenerate}.
	If $M$ is a $P$-matrix, then
	\[ \Fix \left(\Tpdmc\right) = \Fix \left(T_{\rm
		FB}^{\lambda}\right) = \Fix 
	\left(T_{\rm PS}^{\lambda}\right)  = S_1\cap S_2, \quad \forall \lambda \in (0, 1/L_Q].\]
\end{theorem}

Combining \cref{thm:fixedpoints_Pmatrix} and
\cref{thm:global_lcp_nondegenerate}, we obtain global convergence of
the algorithms to the solution set of the problem, for both the
non-accelerated and the accelerated versions. 

\begin{corollary}[Global convergence to solution set]
	\label{cor:lcp_global_Pmatrix}
		The algorithms given in \cref{thm:global_lcp_nondegenerate} and their accelerated versions via \cref{alg:PDMC_accelerated} are globally convergent to
		$S_1\cap S_2$ if $M$ is a $P$-matrix. Moreover, the projected
		subgradient algorithm converges $Q$-linearly to $S_1\cap S_2$. 
\end{corollary}

%
\begin{proof}
	In the
	proof of \cref{thm:global_lcp_nondegenerate}, we have shown the
	coercivity of the corresponding Lyapunov functions of the
	algorithms. Hence, by \cref{eq:descent_acce_fpa},
	\cref{alg:PDMC_accelerated} generates a bounded sequence, and accumulation points are fixed points by \cref{thm:global_acce}. Together
	with \cref{thm:fixedpoints_Pmatrix}, any sequence generated by
	\cref{alg:PDMC_accelerated} must converge to the unique point in
	$S_1\cap S_2$. Setting $t_k\equiv 0$ in \cref{alg:PDMC_accelerated} gives the desired result for the non-accelerated algorithms. Local linear convergence of the projected subgradient algorithm follow from \cref{thm:global_lcp_nondegenerate} (c).\cpsolved{I don't think here sbusequencial convergence has
	already implied full convergence. Am I missing something?}
	\jhsolved{Since the sequence is bounded and all subsequential limits are equal to the \textit{unique} point in $S_1\cap S_2$, the full sequence is convergent. We can actually merge Corollary 6.10 and Corollary 6.11. What do you think?}
		\ifdefined\submit
		\qed
		\else
		\qedhere
		\fi
\end{proof}

\begin{remark}
	In the same spirit as in the discussions in \cref{remark:component}, \cref{remark:specialcases} (b), and \cref{sec:acce_pdmc}, we note that latter 
	iterations of 
	the algorithms \cref{eq:pdmc_affine_ucs,eq:fb_affine_ucs,eq:ps_affine_ucs} 
	indicate which $\iota\in \I$ can be used 
	to reduce the original problems \cref{eq:unified_sparseaffine,eq:unified_sparseaffine2} into 
	the simpler problem of finding a point in $S_1\cap R_{\iota}$. For the LCP, finding $S_1\cap R_{\iota}$ 
	is equivalent to solving the system $Aw^* = b$ and 
	$w^*_{\Lambda_{\iota^c}}=0$, which is 
	simply an $n\times n$ system of linear equations. If the obtained solution satisfies 
	$w^*_{\iota}\geq 0$, then $w^*$ is indeed a solution of the original feasibility problem. 
	For 
	\cref{eq:pdmc_affine_ucs,eq:fb_affine_ucs,eq:ps_affine_ucs} and their
	extrapolation-accelerated versions by \cref{alg:PDMC_accelerated},
	\cref{cor:lcp_global_Pmatrix}
	guarantees that these algorithms will converge to 
	the unique point $w^*$ in $S_1\cap S_2$ when $M$ is a $P$-matrix. Thus, 
	theoretically, we know that \cref{alg:PDMC_plus} will indeed output the solution 
	$w^*$. Similarly, for the sparse affine feasibility problem, the reduced feasibility problem of finding a point in $S_1\cap R_{\iota}$ amounts to solving 
	the linear system $A_{:,\iota}w_{\iota}=b$. These remarks will be used in the numerical implementation of \cref{alg:PDMC_plus} in \cref{sec:numerical}. 
	%
	
	
	\label{remark:linearsystem}
\end{remark}

\subsubsection{Special properties of the projected subgradient algorithm for LCP}\label{sec:ps_lcp_pmatrix}

We already know from \cref{cor:lcp_global_Pmatrix} that the projected subgradient algorithm \cref{eq:ps_affine_ucs} is globally convergent to $S_1\cap S_2$ for stepsizes $\lambda \in (0,1/L_Q)$. In this section, we show that the result also holds for $\lambda = 1/L_Q$. This in turn shows the global convergence of the method of alternating projections by setting $Q=(AA^\T)^{-1}$, which is a rare result in the nonconvex setting. Indeed, proving such a result for the LCP requires a number of technical lemmas, an indication that global convergence for MAP is indeed difficult to obtain for nonconvex problems in general.

In addition to \cref{thm:fixedpoints_Pmatrix}, the following
proposition is needed for proving the desired global convergence result.
As the proof needs many other technical lemmas, we defer its
presentation to \cref{app:lyapunov_pmatrix}.

\begin{theorem}\label{theorem:fdecreases_Pmatrix}
	Consider the setting of \cref{thm:global_lcp_nondegenerate}.
	Let $w\in S_2 \setminus \Fix (T_{\rm PS}^{\lambda})$ and $w^+ \in  
	T_{\rm PS}^{\lambda}(w)$ where $\lambda \in (0,1/L_Q]$. If $M$ is a 
	$P$-matrix, then $f_Q(w^+) < f_Q(w)$. Consequently, $V \coloneqq f_Q+\delta_{S_2}$ is a Lyapunov function for
	\cref{eq:PS} for any $\lambda \in (0,1/L_Q]$.
\end{theorem}


We now state our main result showing the convergence of \cref{eq:ps_affine_ucs} with stepsize $\lambda = 1/L_Q$. We highlight that for a sequence generated by the PS algorithm, we obtain an 
additional property that the 
objective function values $\{f_Q(w^k)\}$ decreases to zero $Q$-linearly as well. 

\begin{theorem}
	\label{thm:global_Pmatrix}
	Let $M$ be a $P$-matrix, $b\in \Re^n$, $A = [M~-\Id]$,
	and consider \cref{eq:feasibilityproblem} with
	\cref{eq:S1_lcp}.
	Denote by $w^*$ the unique point in $S_1\cap S_2$ and let $Q\in \{ 
	(AA^\T)^{-1},\Id\}$, and $f_Q$ and $L_Q$ be given by \cref{eq:f_Q} and 
	\cref{eq:LQ}, respectively. Any sequence generated by
		\cref{eq:ps_affine_ucs} 
		with $\lambda\in (0,1/L_Q]$ 
		converges to $w^*$ with a local $Q$-linear rate. Moreover, the objective function
		\cref{eq:unified_sparseaffine2} converges to the global
		optimum of $0$ with a local $Q$-linear rate. 
\end{theorem}
\begin{proof}
		Linear convergence of 
		$\{w^k\}$ to $w^*$ when $\lambda\in (0,1/L_Q)$ is already provided in \cref{thm:global_lcp_nondegenerate}. Linear convergence for $\lambda = 1/L_Q$ can be proved using the fact from \cref{theorem:fdecreases_Pmatrix} that $V=f_Q+\delta_{S_2}$ is a Lyapunov function when $\lambda = 1/L_Q$, together with \cref{remark:lambda_full}, \cref{thm:fixedpoints_Pmatrix}, and the techniques used in
		\cref{thm:global_lcp_nondegenerate}. Thus, it remains to show that $f_Q(w^k) 
		\to 
		0$ at a 
		$Q$-linear rate. 
		
		From \cref{remark:specialcases} (b), we know that there exists $N \geq 0$ such 
		that 
		$w^k \in R_{\iota}$ for all $k \geq N$ and for any $\iota \in \I$ such that $w^*\in 
		R_{\iota}$. 
		 It then follows that 
		$w^k-w^* 
		\in S_2 \cup (-S_2)$ for all $k\geq N$. Suppose now that $Q=(AA^\T)^{-1}$. By using 
		\cref{lemma:SRIP_lcp} and noting that $AA^{\dagger}=\Id$, we get
			\[\nu \|w^k-w^*\|^2 \leq \|A(w^k-w^*)\|^2 = \| 
			(AA^{\dagger})A(w^k-w^*)\|^2 \leq \|A\|^2 
			\cdot \|A^{\dagger}A(w^k-w^*)\|^2 ,\]
		that is, 
		\begin{equation}
			\frac{\nu}{\|A\|^2}\|w^k-w^*\|^2  \leq \|A^{\dagger} A(w^k-w^*)\|^2,
			\quad \forall k\geq N.
			\label{eq:SRIP_lcp_1}
		\end{equation}
		On the other hand, if $Q=\Id$, we immediately get from 
		\cref{lemma:SRIP_lcp} that 
		\begin{equation}
			 \nu \|w^k-w^*\|^2  \leq \|A(w^k-w^*)\|^2,
			\quad \forall k\geq N.
			\label{eq:SRIP_lcp_2}
		\end{equation}
		Since $w^* \in S_1$, with \cref{eq:SRIP_lcp_1,eq:SRIP_lcp_2}, we obtain that 
		\begin{equation}
			\frac{1}{2} \|w^k-w^*\|^2 \leq \eta f_Q(w^k), \,
			\forall k\geq N, 
			\,\text{where} \,
			\eta=  \begin{cases}
				\frac{\|A\|^2}{\nu} & \text{if}~Q = (AA^\T)^{-1} \\
				\frac{1}{\nu} & \text{if}~Q=\Id.
			\end{cases}
			\label{eq:f_Q>=normsquared}
		\end{equation}
		Then, 
		\begin{align}
			\nonumber
			f_Q\left( w^{k+1} \right) & \overset{\cref{eq:f<=Qf}}{\leq} \min_{z 
			\in 
				S_2} \;
			Q_{f_Q}^{\lambda}(z,w^k) \\ 
			&\leq 
			\min_{\alpha \in [0,1]}\; f_Q\left( \alpha w^* + (1 - \alpha) w^k
			\right) + \frac{1}{2\lambda} \left\| w^k - \left(\alpha w^* + (1 - 
			\alpha)
			w^k\right)
			\right\|^2 \nonumber \\
			& \leq (1-\alpha) f_Q(w^k)  + 
			\frac{\alpha^2}{2\lambda}
			\left\| w^k - w^*\right\|^2,\quad  \forall \alpha \in [0,1], 
			\nonumber 
		\end{align}
		where the last inequality is from the convexity of $f_Q$ and the fact that $f_Q(w^*) = 0$.
		Applying \cref{eq:f_Q>=normsquared} to the inequality above then gives
		\begin{equation}
			f_Q\left( w^{k+1} \right) \leq \left(1 - \alpha + 
			\frac{\eta}{\lambda}\alpha^2 \right) f_Q(w^k), \quad \forall \alpha
			\in [0,1].
			\label{eq:rate}
		\end{equation}
		The claim now follows by minimizing the right-hand side of 
		\cref{eq:rate} with respect to $\alpha$.
		\ifdefined\submit
		\qed
		\else
		\qedhere
		\fi
\end{proof}

\section{Numerical experiments}
\label{sec:numerical}
This section presents numerical experiments on sparse 
affine feasibility and linear 
complementarity problems to support the established theoretical convergence of 
the proposed algorithms and \modify{to demonstrate the efficiency of the
	acceleration schemes \cref{alg:PDMC_accelerated,alg:PDMC_plus}.}
For simplicity, we focus on the projected gradient algorithm
\cref{eq:ps_affine_ucs} with $Q\in \{I_n, (AA^\T)^{-1} \}$. We keep
the \modify{notations PS for $Q=I$ and MAP for $Q=(AA^\T)^{-1}$}. We include a prefix ``A'' and/or a
suffix ``$+$'' to signify that \cref{alg:PDMC_accelerated} and/or 
\cref{alg:PDMC_plus} are incorporated in the algorithms.
All experiments are conducted on a machine running Ubuntu 20.04 and
MATLAB R2021b with 64GB memory and an Intel Xeon Silver 4208 CPU with
8 cores and 2.1 GHz.
%
\modify{%
	To satisfy \cref{eq:descent_acce_fpa} in \cref{alg:PDMC_accelerated},  since $f_Q$ is a quadratic function,
	a closed form stepsize is obtained by taking $t_k = \max \{ 0,\min
		\{t_k^{(1)},t_k^{(2)} \} \}$ for the LCP and $t_k = \max \{
			0,t_k^{(1)}\}$ for the SAFP, where $t_k^{(1)}
			=\frac{-2\nabla f_Q(w^k)^\T p^k}{(Ap^k)^\T Q(Ap^k)+\sigma
			\|p^k\|^2}$ and $t_k^{(2)} = \min \{ -w_j^k/p_j^k: p_j^k<0
			\}$.
		}\cpsolved{Moved your remark here and slightly updated.}


\subsection{Sparse affine feasibility problem}
We consider SAFP with synthetic and real datasets described below and
compare our methods with
the proximal gradient method by \citet{BT11}, which we denote by
PG-BT.
%
For PS/MAP, we set
$\lambda 
= \tau/L_Q$ with $\tau = 0.999$ (see
\cref{thm:global_ps_affine}) and
$\sigma = 10^{-2}$. The parameter $N$ in \cref{alg:PDMC_plus} 
is 
set to $50$ for MAP and $100$ for PS.
	When \cref{alg:PDMC_plus} is used in combination with 
\cref{alg:PDMC_accelerated}, we set $N$ to half its specified value when only \cref{alg:PDMC_plus} 
is used.
The linear system described in 
\cref{remark:linearsystem} for dealing with Step 2.2 of
\cref{alg:PDMC_plus} is handled by solving $	A_{:,\iota}^{\top} A_{:,\iota} w_{\iota} = A_{:,\iota}^{\top} b$
using the conjugate gradient (CG) method
(see, for example, \citep[Chapter~5]{NoceWrig06}).
	For the SAFP problems, the cost of one CG iteration
	is $O(ms)$ and the number of CG iterations in one round of Step
	2.2 of \cref{alg:PDMC_plus} is upper bounded by $s$. Hence, the
	overall cost of invoking the CG procedure once is at most
	$O(ms^2)$, although we often observe that CG terminates
	within few iterations in practice, while one step of \cref{eq:ps_affine_ucs} is $O(mn)$.
	We also observe that empirically the CG procedure takes an almost
	negligible amount of running time in the whole procedure.
All algorithms are initialized with 
$w^0 = 
A^\T b$, and the residual is measured by
\begin{equation}
	\text{Residual} \coloneqq \frac{1}{2}\|Aw^k-b\|^2 + \frac{1}{2}\dist 
	(w^k,S_2)^2, 
	\label{eq:stop_saf}
\end{equation}
which is $0$ if and only if $w^k$ is a solution to the SAFP.

\medskip 
\noindent \textbf{Synthetic data.} We follow \cite{BBC11} to generate standard random test problems involving a 
matrix 
$A\in 
\Re^{m\times n}$ with entries
sampled from the standard normal distribution, and a sparse signal 
$w^*\in A_s$ such that the nonzero entries $w_i^*$ are generated as $w_i^* = 
\eta_1 10^{\alpha \eta_2}$ with $\alpha = 5$, $\eta_1 = \pm 1$ with probability 0.5, and 
$\eta_2 $ uniformly sampled from $[0,1]$. After generating $A$ 
and 
$w^*$, we set $b=Aw^*$ so that $w^*$ is a solution of the SAFP.  The
running time and
total iterations required for reducing \cref{eq:stop_saf} below
$10^{-6}$ with $n=10000$, $m=2500$, and
$s=625$ over ten independent trials are summarized in \cref{tab:safp}.

We see from \cref{tab:safp} that
the acceleration schemes \cref{alg:PDMC_accelerated,alg:PDMC_plus} 
reduce both the running time and the number of iterations of the 
algorithms.
The non-accelerated MAP algorithm is already more efficient than PG-BT,
but when $Q=\Id$, only the accelerated versions of PS have 
better 
performance than PG-BT.\cpsolved{But we don't have nonaccelerated PS here?} Finally, we observe that for this experiment, \cref{alg:PDMC_accelerated} has 
faster convergence than \cref{alg:PDMC_plus}, and incorporating component identification to 
\cref{alg:PDMC_accelerated} only resulted to minimal improvements in convergence time.
Component identification in this experiment only helped to
reduce the residual to a much lower level after the stopping criterion
of $10^{-6}$ is almost reached.

\begin{table}[tbh!]
	\centering
	\caption{Performance of algorithms on ten independent trials of 
		SAFP with synthetic data.
		For the running time and residual of each method, we report
		their average$\pm$standard deviation. 
Ave. CI Iters. refers to the average 
number of times Step 2.2 in \cref{alg:PDMC_plus}) is executed, while Ave. CI Time is the average amount 
of time required to finish \emph{one} CI iteration. \modify{PS and PS+ are
omitted as they} failed to solve the problems after 10000 iterations. } 
	\begin{tabular}{lrrrrcr}
\toprule 
\multirow{2}*{Method} & \multicolumn{1}{c}{Ave.} & \multicolumn{1}{c}{Time} & \multicolumn{1}{c}{Ave.}  & \multicolumn{1}{c}{Ave.}  & \multirow{2}*{Residual}  \\ 
 & \multicolumn{1}{c}{Iters} & \multicolumn{1}{c}{(seconds)} & \multicolumn{1}{c}{CI Iters} & \multicolumn{1}{c}{CI Time} &   \\ 
\midrule 
MAP & 673.6 & 10.6 $\pm$ 0.2 & NA & NA & 9.3e-07 $\pm$ 4.5e-08 \\ 
AMAP & 263.4 & 4.5 $\pm$ 0.3 & NA & NA & 7.8e-07 $\pm$ 8.6e-08 \\ 
MAP+ & 600.1 & 9.5 $\pm$ 0.2 & 1.2 & 0.014 & \bftab{1.4e-10 $\pm$ 4.5e-11} 
\\ 
AMAP+ & \bftab{250.1} & \bftab{4.3 $\pm$ 0.3} & 1 & 0.014 & \bftab{1.4e-10 
$\pm$ 4.5e-11} \\ 
APS & 417.5 & 5.7 $\pm$ 0.3 & NA & NA & 8.4e-07 $\pm$ 6.7e-08 \\ 
APS+ & 402.9 & 5.5 $\pm$ 0.3 & 1 & 0.014 & \bftab{1.4e-10 $\pm$ 4.5e-11} \\ 
\midrule 
PG-BT & 847.0 & 15.3 $\pm$ 0.6 & NA & NA & 9.5e-07 $\pm$ 4.2e-08 \\ 
\bottomrule 
\end{tabular}

	\label{tab:safp}
\end{table}

\medskip 
\noindent \textbf{Real-world datasets.} We then consider three public real-world 
datasets:\footnote{Downloaded from
	\url{http://www.csie.ntu.edu.tw/~cjlin/libsvmtools/datasets/}.} \colon ($m=62$, $n=2000$), \duke ($m=44$, $n=7129$) and \leu ($m=38$, 
$n=7129$). We set $s$ to 5\% of the total number of features $n$. 
	The results are 
summarized in \cref{fig:safp_real_time}.

Similar to the results on synthetic datasets,
the acceleration schemes \cref{alg:PDMC_accelerated,alg:PDMC_plus} 
reduce both the running time and the number of iterations of the 
algorithms, except that component identification in \cref{alg:PDMC_plus} did 
not take place for \duke.
For the other two datasets, component identification greatly reduced
both the running time and the number of iterations.
The algorithms corresponding to $Q=(AA^\T)^{-1}$ also provided performance better than 
those with $Q=\Id$ and PG-BT.

\begin{figure}[tb!]
	\centering
	\begin{tabular}{@{}ccc@{}}
		\includegraphics[width=.23\linewidth]{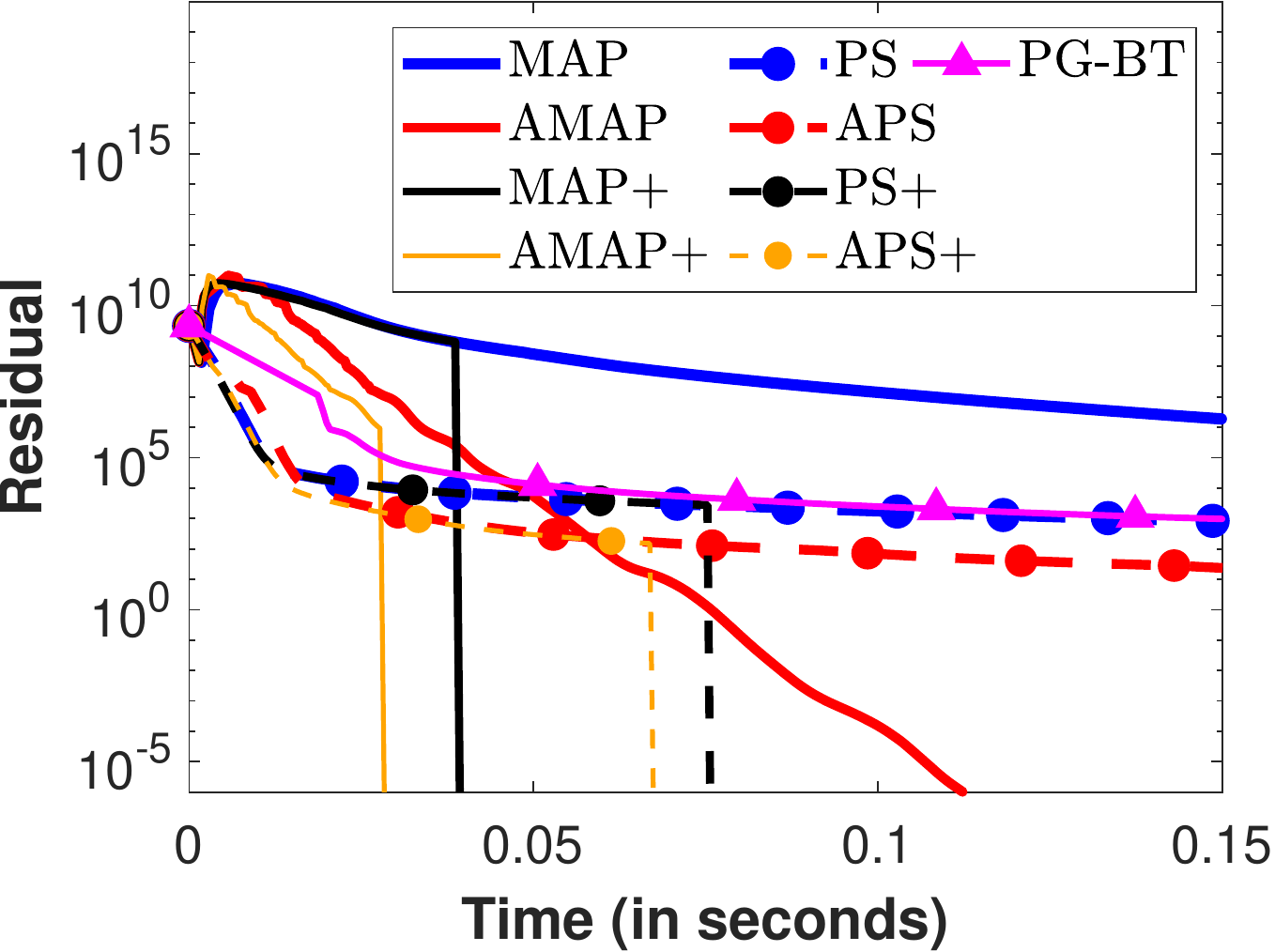}&
		\includegraphics[width=.23\linewidth]{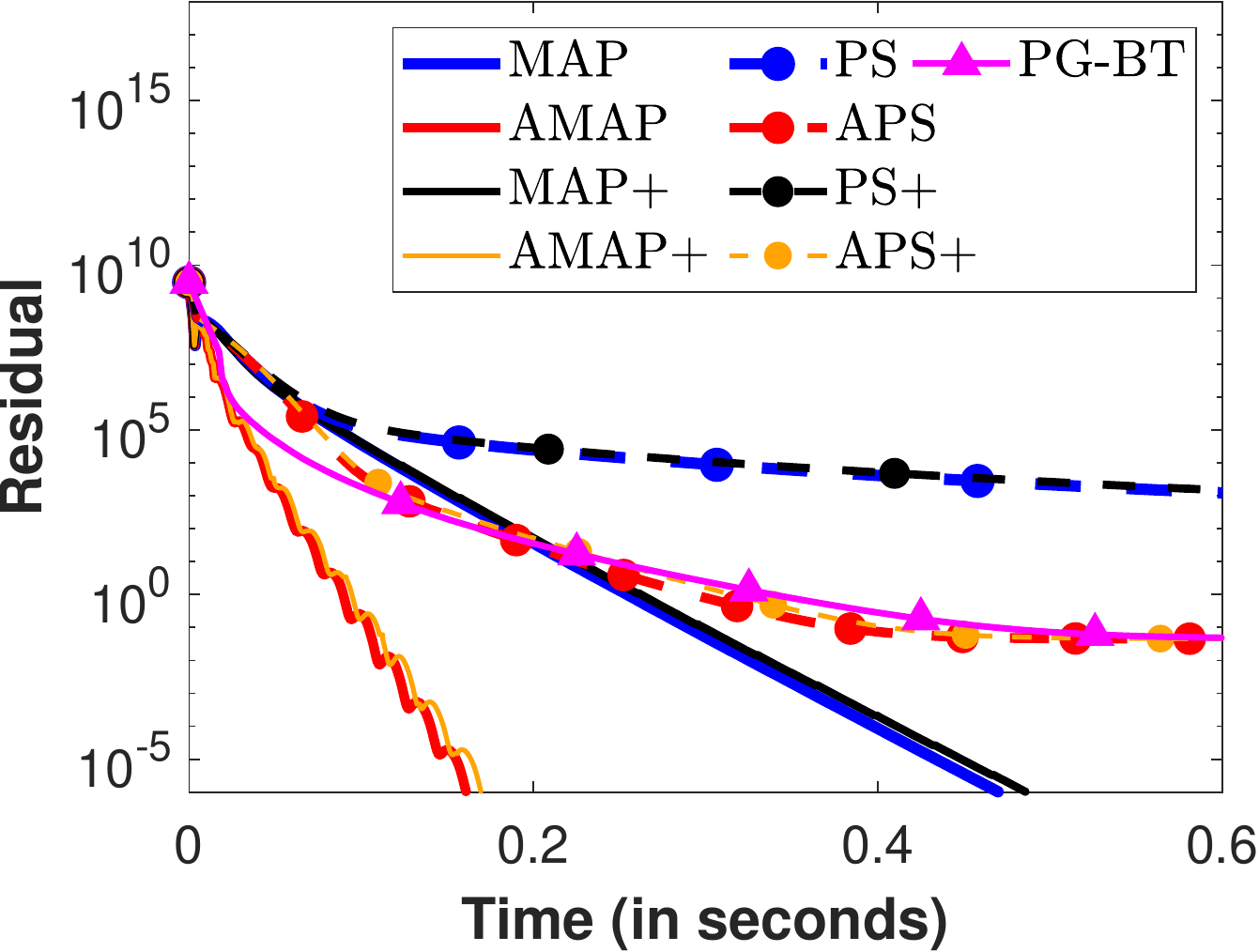}& 
		\includegraphics[width=.23\linewidth]{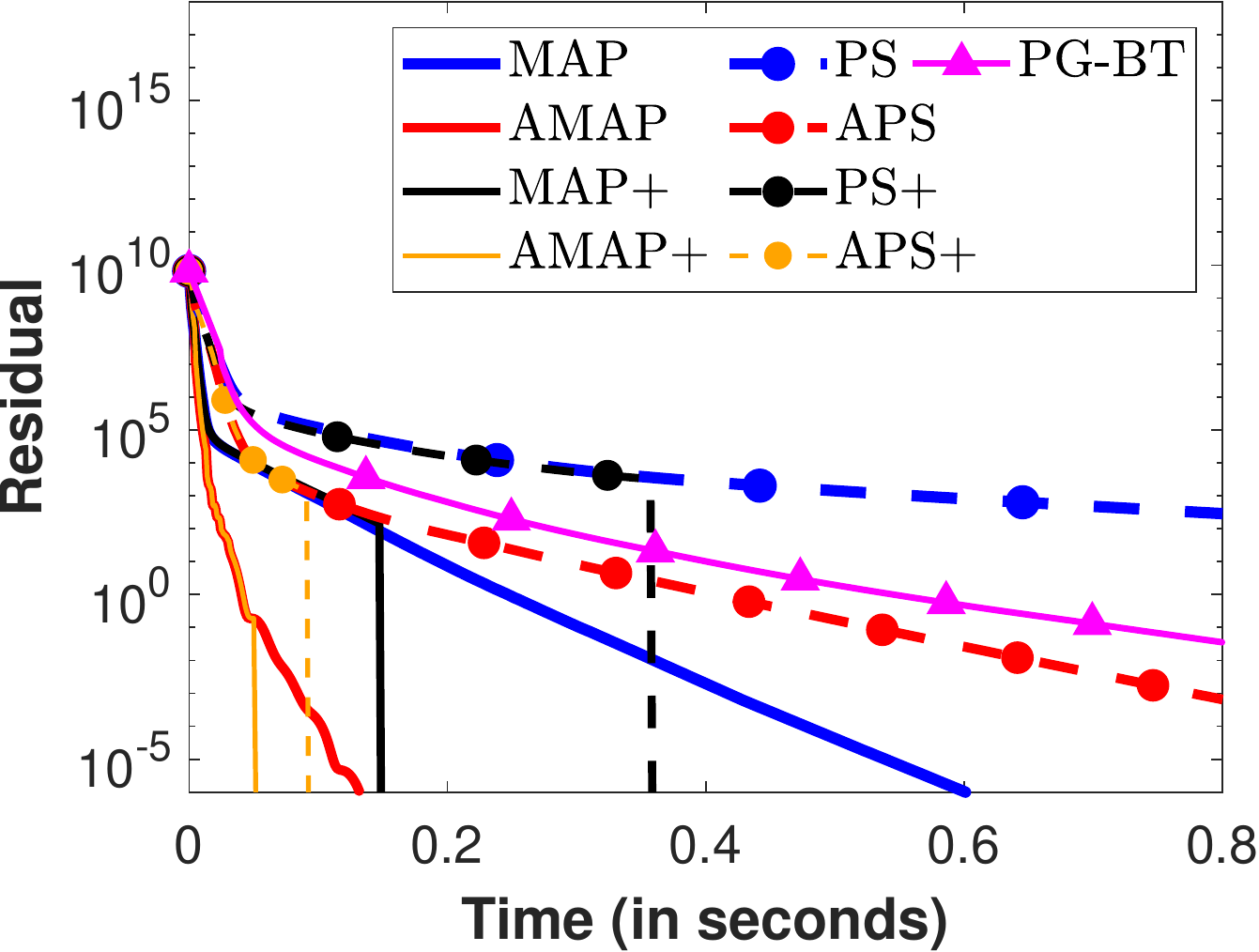} \\
		\includegraphics[width=.23\linewidth]{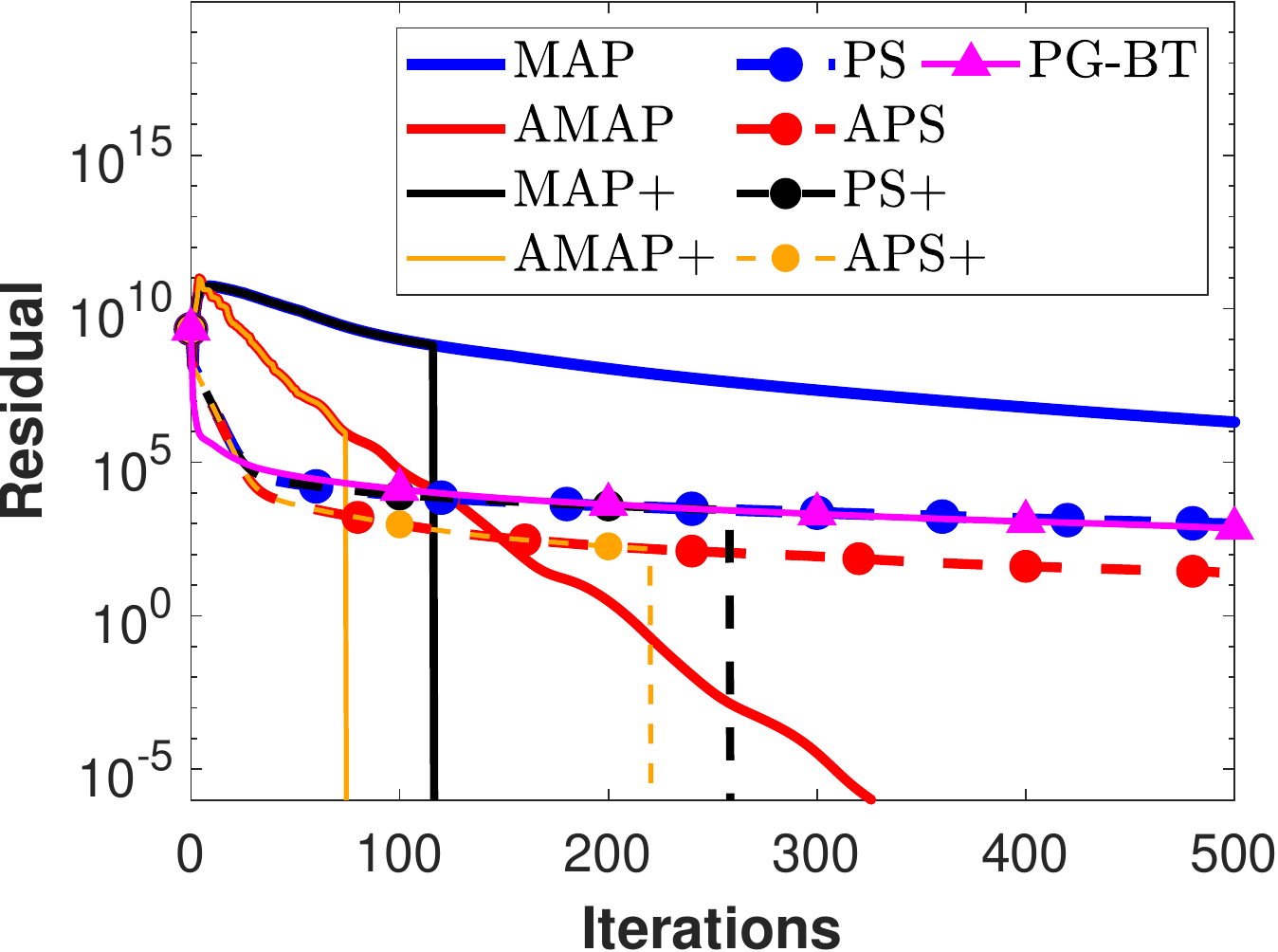}&
		\includegraphics[width=.23\linewidth]{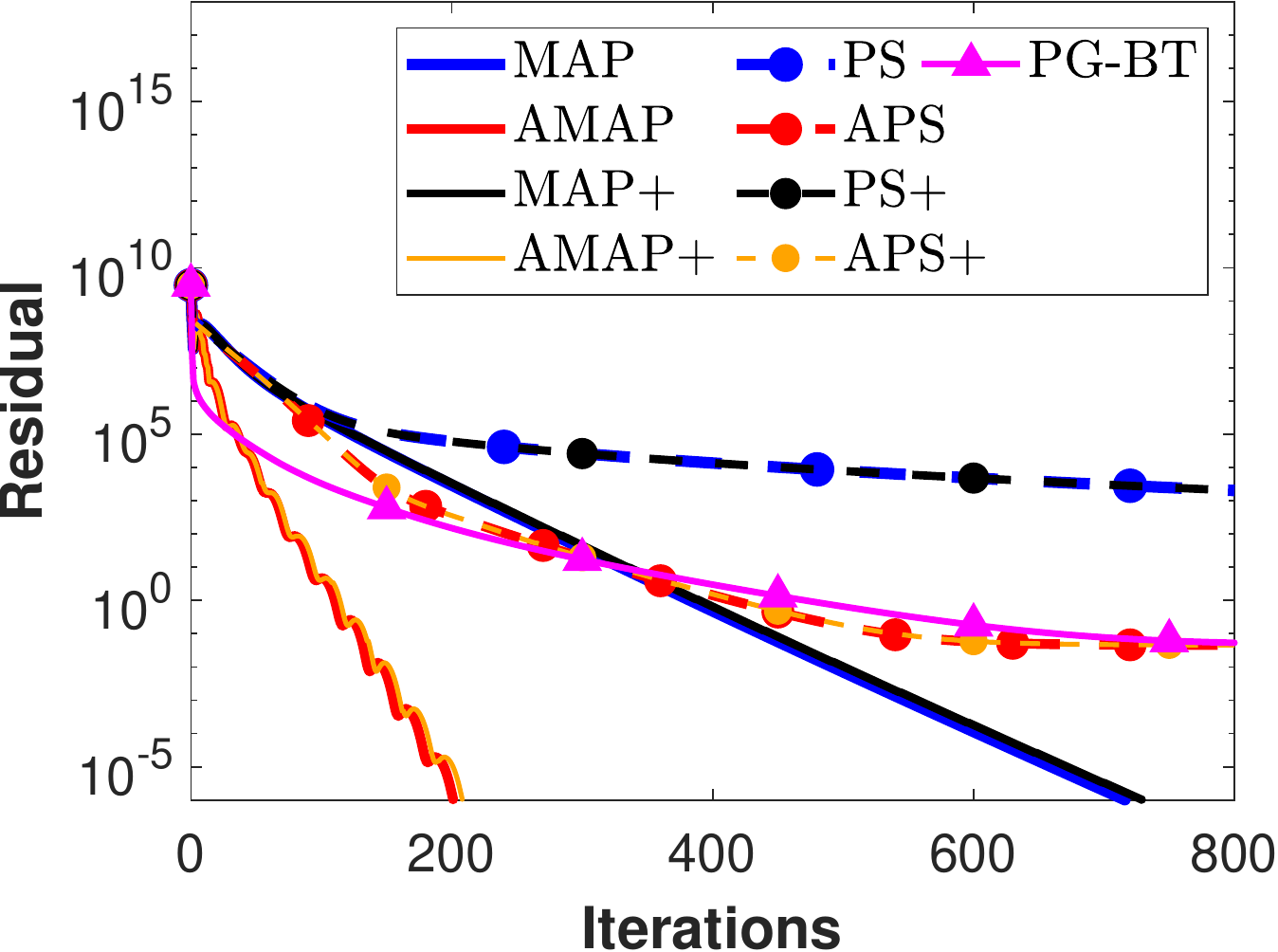}&
		\includegraphics[width=.23\linewidth]{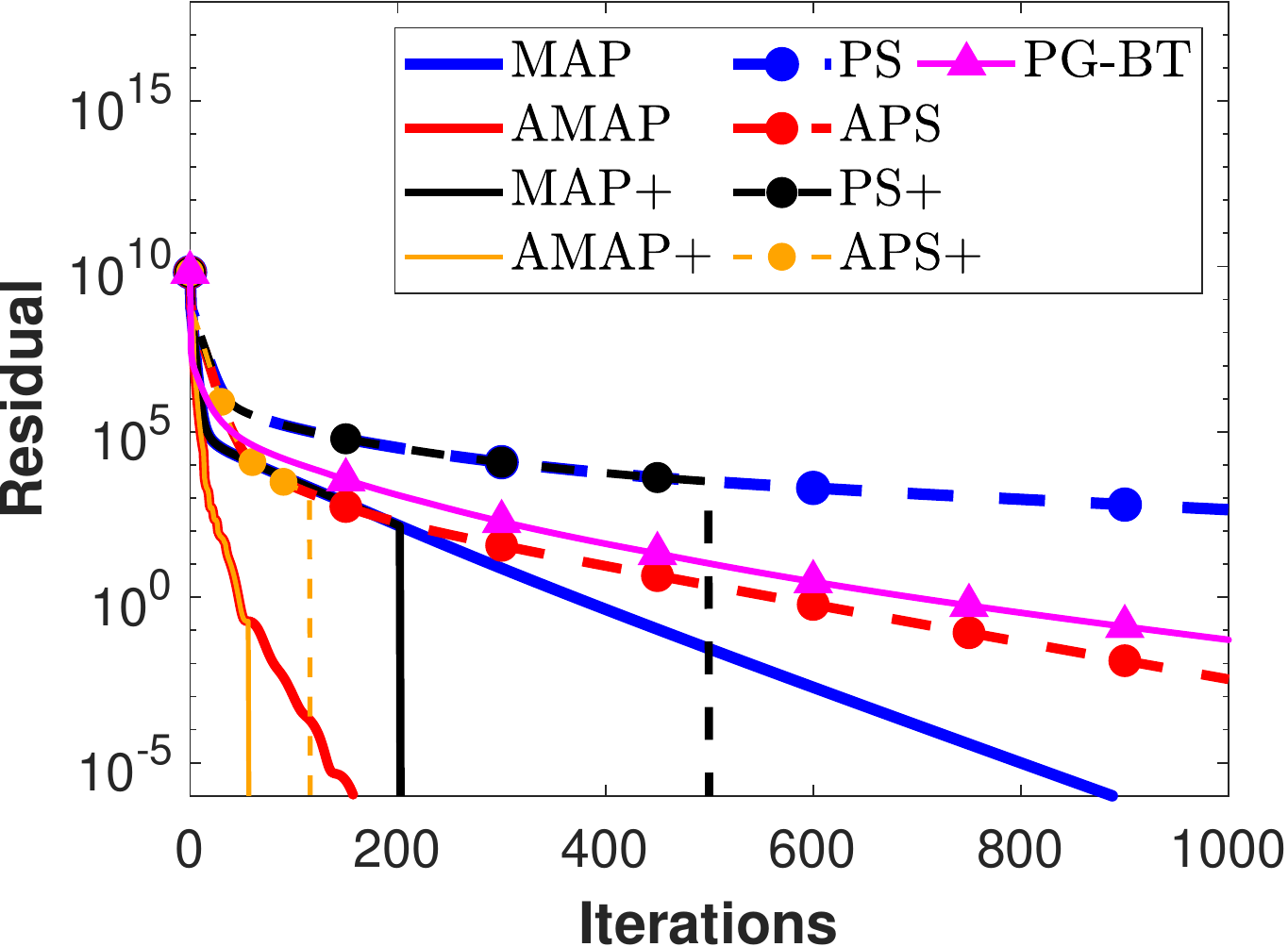} \\ 
		& & \\
		\colon  & \duke  & \leu \\
	\end{tabular}
	\caption{Comparisons of running time and iteration number of algorithms for 
		solving 
		SAFP with 
		real-world 
		datasets.}
	\label{fig:safp_real_time}
\end{figure}

\subsection{Linear complementarity problem}
We consider standard LCP test problems as follows and compare our
algorithms with BPA and EGA mentioned in \cref{sec:relatedworks}.

\medskip 
\noindent \textbf{LCP1.} \citep[Example 2]{QSZ00}  $M$ is a tridiagonal matrix 
with 
$M_{ii}=4$ 
for all $i\in [n]$ and $M_{ij}=-1$ when $|i-j|=1$, and 
$b=(1,1,\dots,1)^\T$.

\medskip 
\noindent \textbf{LCP2.} \citep[Example 7.1]{Kanzow96} $M$ is an upper 
triangular matrix 
with $M_{ii}=1$ for 
all $i\in [n]$, $M_{ij}=2$ for all $i<j$, and $b=(1,1,\dots,1)^\T$.
BPA is excluded for this case as it is applicable only when $M+M^\T$ is positive definite \citep[Theorem 
12.1.2]{FP03}.

\medskip 
\noindent \textbf{LCP3.} \citep[Example 7.3]{Kanzow96}
The entries of $b$ are independently
sampled from uniform random with range $(-500,500)$. $M$ is a $P$-matrix given by $M=A_1^\T A_1 + 
A_2+{\rm 
diag}(\eta)$, where $A_1, A_2\in \Re^{n\times n}$ are matrices with entries 
independently sampled from uniform random in $(-5,5)$, $A_2$ is skew-symmetric, and
each entry of $ \eta \in 
\Re^n$ is independently taken
from uniform random of $(0,0.3)$.

We set $\lambda =1$ for MAP according to \cref{thm:global_Pmatrix}.
$\sigma$ and $N$ follow the setting in the preceding section. Matlab's backslash operator is used to handle the linear 
system described in \cref{remark:linearsystem}, with the cost of
$O(n^3)$ for our problem, which is of the same order as the overhead
of computing $Q$ and $L_Q$.
On the other hand, the cost of one iteration of \cref{eq:pdmc} is
$O(n^2)$.
We will see in the experimental results that although component
identification in this case is slightly more expensive than its
counterpart in the SAFP experiment,
it still takes only a small portion of the overall running time of the
algorithms.
We set $n=5000$ in all of the experiments, 
and divide both $M$ and 
$b$ by the same scalar $\|M\|_1/\sqrt{n}$. This normalization is due to the geometric observation 
that for $n=1$, projection algorithms tend to converge faster to a solution 
when the slope is in a moderate range. 
%
Instead of \cref{eq:stop_saf}, we use the following standard measure of 
residual 
in LCP \citep[Proposition 1.5.8]{FP03} to
facilitate fair comparisons with BPA and EGA.
\begin{equation}
	\text{Residual} \coloneqq \| \min (x^k, Mx^k-b) \|. 
	\label{eq:residual}
	\end{equation}
We report the running time and iterations required for reducing
\cref{eq:residual} below $10^{-6}$.
For the feasibility reformulation of the LCP (see 
\cref{sec:lcp}), the first $n$ coordinates of $w^k$ correspond to 
$x^k$.

We see from \cref{fig:lcp_time} and \cref{tab:lcp} that 
indeed, the proposed acceleration schemes significantly reduce the required 
time and number of iterations to solve the generated LCPs.\cpsolved{Looks
like the 2nd row has some legends missing. You might also want to
increase the iter count for LCP2 in the figure. Blue solid line is
also missing or invisible for LCP1.}
The only exception is LCP1, on which the non-accelerated algorithms
terminate before reaching the specified $N$ for \cref{alg:PDMC_plus},
so component identification is not executed at all.
The residual of MAP+ presented in \cref{tab:lcp} tends to be much lower
than our stopping condition,
as the linear system solver is non-iterative and
cannot be terminated exactly at the point where the required residual
tolerance is met.
On the other hand, for AMAP+, component identification does
not change the residual much.
A closer examination revealed that in this case, the algorithm
sometimes terminates without triggering component identification.

Overall speaking, the proposed acceleration scheme in
\cref{alg:accelerated_FPA} using extrapolation is indeed very
effective in reducing the running time and iterations of fixed-point
maps, while the component identification part in \cref{alg:PDMC_plus}
is more useful when highly accurate solutions are required.

\begin{figure}[tb!]
	\centering
	\begin{tabular}{@{}cccc@{}}
		\includegraphics[width=.22\linewidth]{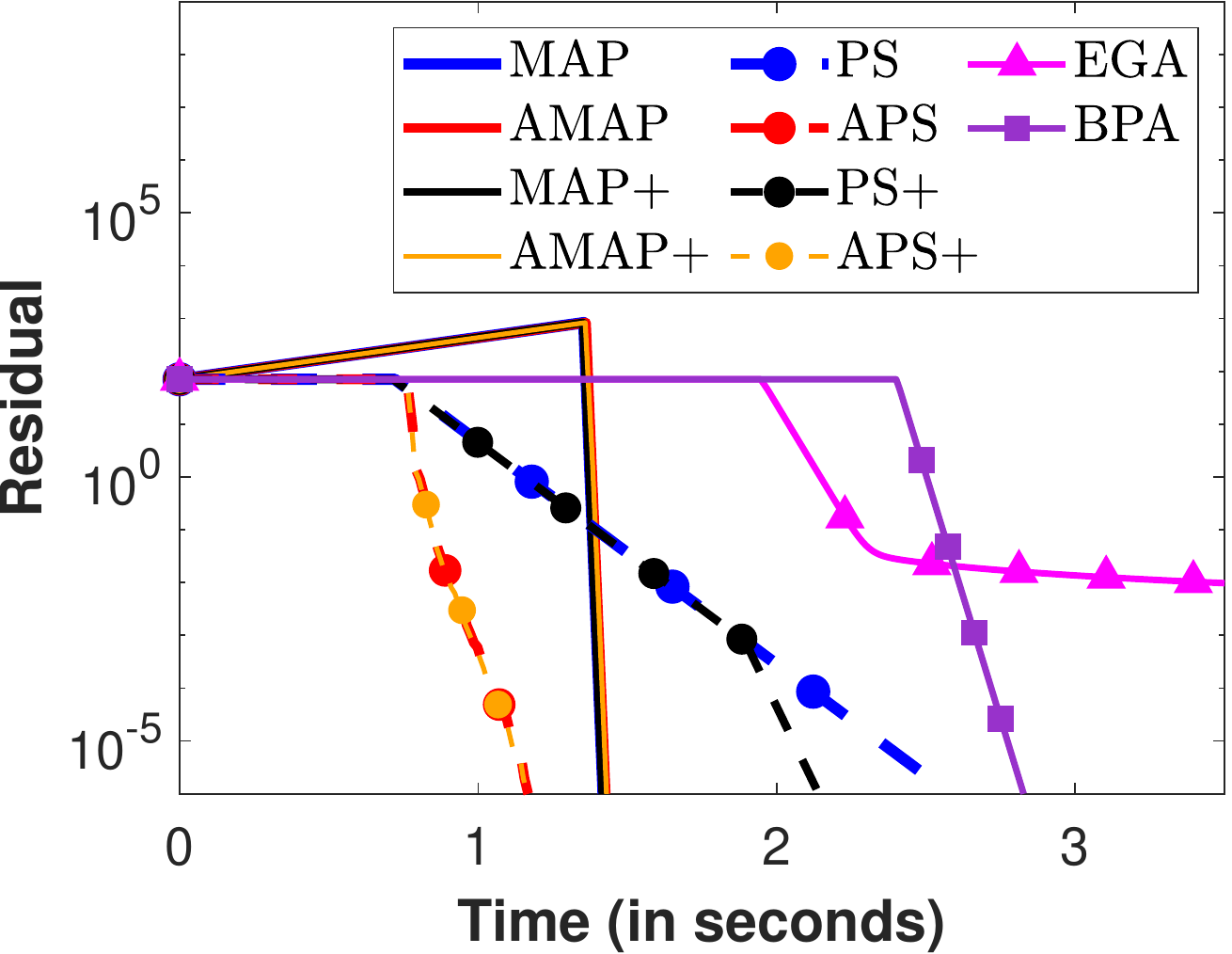}&
		\includegraphics[width=.22\linewidth]{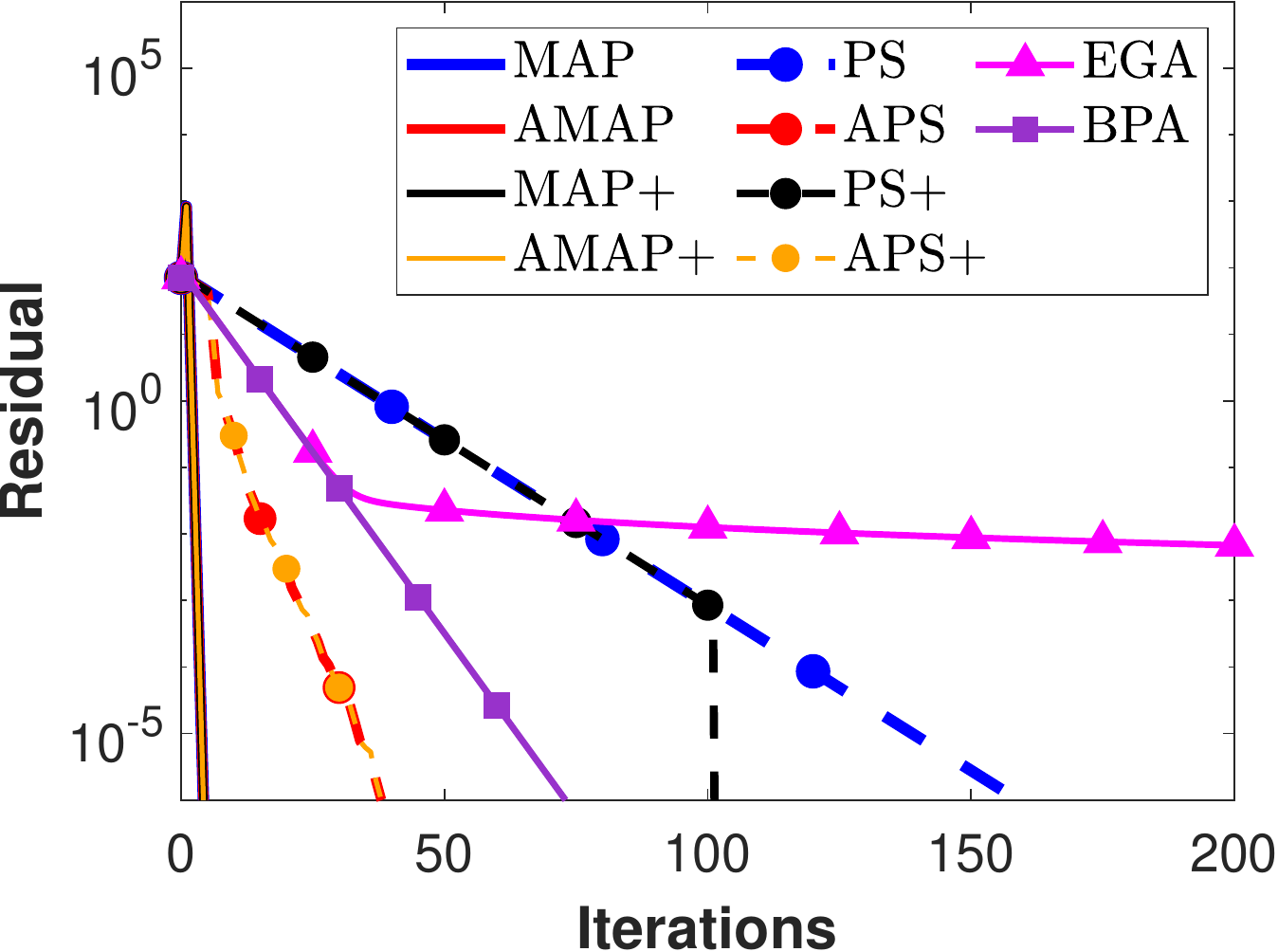}&
		\includegraphics[width=.22\linewidth]{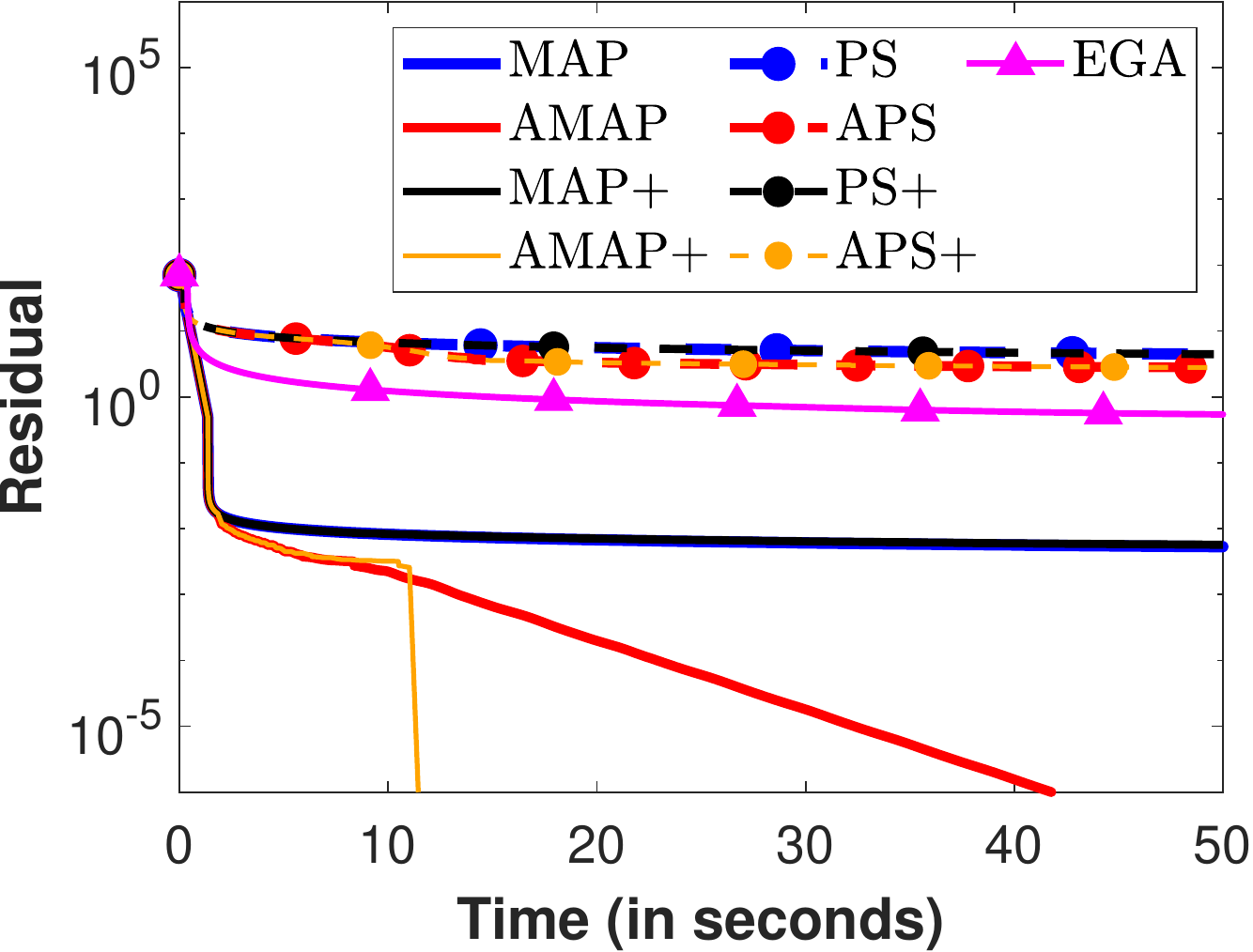} &
		\includegraphics[width=.22\linewidth]{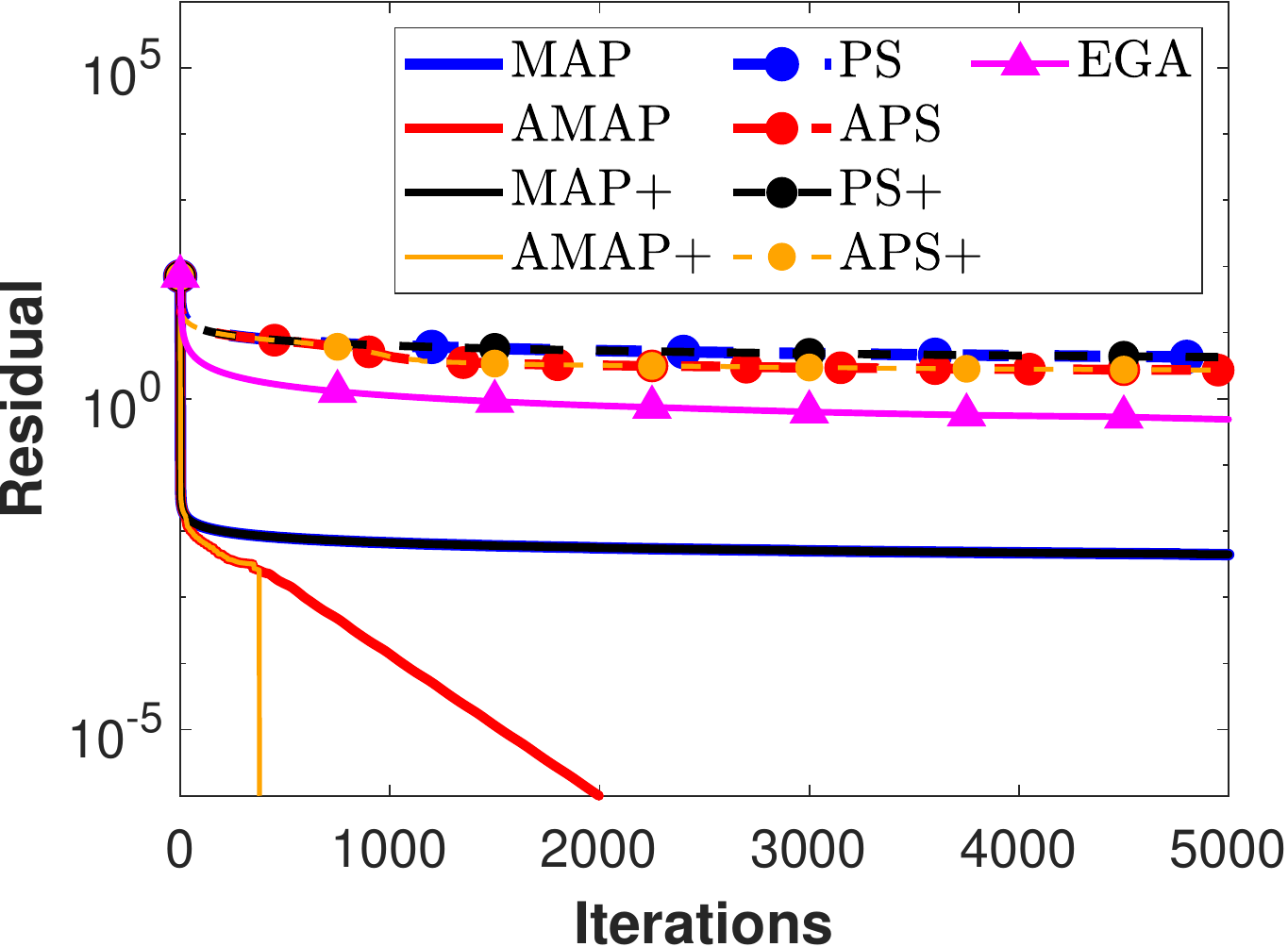}\\
		\multicolumn{2}{c}{LCP1} & \multicolumn{2}{c}{LCP2}
	\end{tabular}
	\caption{Comparisons of running time and iteration number of algorithms for 
		solving LCP1 and LCP2. For LCP1, the algorithms MAP, AMAP and APS overlap with MAP+, AMAP+ and APS+. Comparison with BPA on LCP2 is omitted because BPA is not
	applicable.}
	\label{fig:lcp_time}
\end{figure}

\begin{table}[tbh!]
	\centering 
	\caption{Performance of algorithms on ten independent trials of
		LCP3 for reducing \cref{eq:residual} to below $10^{-6}$. BPA
		and algorithms corresponding to $Q=\Id$ are omitted as 
		all of them failed to make the residual below $10^{-6}$ in 10000 iterations.
		Ave. CI Iters. refers to the average number of times Step 2.2 in
		\cref{alg:PDMC_plus}) is executed, while Ave. CI Time is the average amount
		of time required to finish \emph{one} CI iteration.}
		\begin{tabular}{lrrrrcr}
\toprule 
\multirow{2}*{Method} & \multicolumn{1}{c}{Ave.} & \multicolumn{1}{c}{Time} & \multicolumn{1}{c}{Ave.}  & \multicolumn{1}{c}{Ave.}  & \multirow{2}*{Residual}  \\ 
 & \multicolumn{1}{c}{Iters} & \multicolumn{1}{c}{(seconds)} & \multicolumn{1}{c}{CI Iters} & \multicolumn{1}{c}{CI Time} &   \\ 
\midrule 
MAP & 979.0 & 20.9 $\pm$ 1.3 & NA & NA & 1.0e-06 $\pm$ 2.3e-09 \\ 
AMAP & 244.1 & \bftab{6.3 $\pm$ 0.4} & NA & NA & 8.8e-07 $\pm$ 1.2e-07 \\ 
MAP+ & 577.1 & 16.1 $\pm$ 2.1 & 2.8 & 1.2 & \bftab{2.2e-15 $\pm$ 1.0e-16} \\ 
AMAP+ & \bftab{238.0} & 7.1 $\pm$ 0.5 & 0.8 & 1.2 & 1.4e-07 $\pm$ 2.9e-07 \\ 
\midrule 
EGA & 914.4 & 11.8 $\pm$ 0.4 & NA & NA & 9.9e-07 $\pm$ 7.5e-09 \\ 
\bottomrule 
\end{tabular}

\label{tab:lcp}
\end{table}

\section{Conclusion}
\label{sec:conclusion}
In this work, we analyzed the global subsequential convergence of fixed point 
iterations of union upper semicontinuous operators, and prove global 
convergence under a local Lipschitz condition. We show that this class of 
fixed 
point algorithms in fact covers several iterative methods for solving 
optimization and feasibility problems alike, and therefore global 
convergence 
of these methods is a consequence of the derived theory for the general 
setting 
of fixed point problems. In particular, we establish global convergence 
of 
proximal algorithms for minimizing a class of nonconvex nonsmooth functions, 
specifically those that can be expressed as the sum of
a piecewise smooth mapping and a function that is the difference of a 
min-$\rho$-convex and a convex function.
Linear convergence is also proven under a mild calmness condition.
We also prove global convergence of 
traditional projection methods for solving feasibility problems involving 
union 
convex sets. Acceleration methods via extrapolation and component 
identification are proposed by utilizing the special structure of the 
defining 
operators of the algorithms.
Numerical evidence illustrated 
that our
proposed acceleration schemes provide significant improvement over 
the 
non-accelerated ones in terms of both the running time and the number of 
iterations 
required to solve the problems. 
Another interesting future work is to obtain an iteration bound for
the component identification result,
and then to further develop global iteration complexities of the
discussed algorithms on top of the identification bound.

\ifdefined\submit
\section*{Declarations}
 CP's research was supported in part by NSTC of R.O.C. grant 109-2222-E001-003. 
\section*{Compliance with Ethical Standards}
The authors declare that they have no conflict of interest. 
\else 
\fi 
\bibliography{bibfile}
\appendix

\section{Proof of \cref{theorem:fdecreases_Pmatrix}}\label{app:lyapunov_pmatrix}
We will first develop necessary tools for separately considering different cases of $w$
and $w^+\in T_{\rm PS}^{\lambda}(w)$.
The following lemma will be our key tool to proving the desired result.
\begin{lemma}\label{lemma:fdecreases_notfixedpoint}
	Consider the setting of \cref{thm:global_lcp_nondegenerate}.
	Let $w\in S_2 \setminus \Fix (T_{\rm PS}^{\lambda})$ and $w^+ \in  
	T_{\rm PS}^{\lambda}(w)$ where $\lambda \in (0,1/L_Q]$. If $w\in 
	R_{\iota}$, 
	$\hat{w} \coloneqq P_{R_{\iota}}(w-\lambda \nabla f_Q(w))\neq 
	w$, 
	then $f_Q(w^+) < f_Q(w)$. 
\end{lemma}

\begin{proof}
From \cref{eq:prox=argmin}, we have
	\begin{equation}
	T_{\rm PS}^{\lambda} (w) = P_S(w-\lambda \nabla f_Q(w)) = \argmin_{z\in 
	S_2}\, Q_{f_Q}^{\lambda}(z,w).
	\label{eq:T_PS=argmin}
	\end{equation}
	By the convexity of $R_{\iota}$ and the definition of $\hat{w}$, we have $\lla 
	w-\lambda \nabla f_Q(w)- \hat{w} , w - \hat{w}\rla \leq 0. $ Therefore,
	\begin{eqnarray*}
		\lambda  \lla \nabla f_Q(w) , \hat{w}-w \rla =	\lla \lambda 
		\nabla 
		f_Q(w) - (w-\hat{w}), \hat{w}-w \rla + 	\lla w-\hat{w} , 
		\hat{w}-w \rla \leq - \|\hat{w}-w\|^2. 
		\label{eq:sameside_inequality} 
	\end{eqnarray*}
	Thus, from the definition of $Q_{f_Q}^{\lambda}$ in 
	\cref{eq:f<=Qf}, 
	\begin{equation}
		Q_{f_Q}^{\lambda} (\hat{w},w) \leq f_Q(w) - 
		\frac{1}{\lambda}\|\hat{w}-w\|^2 + \frac{1}{2\lambda}\|\hat{w}-w\|^2 = 
		f_Q(w) - \frac{1}{2\lambda}\|\hat{w}-w\|^2 < 
		f_Q(w),
		\label{eq:bounding1}
	\end{equation}
	where the last inequality holds because $w\neq \hat{w}$.
	Further, we have from \cref{eq:f<=Qf} and \cref{eq:T_PS=argmin} that
	$f_Q(w^+) \leq 
	Q_{f_Q}^{\lambda}(w^+,w) \leq Q_{f_Q}^{\lambda}(\hat{w},w)$, 
	combining which with \cref{eq:bounding1} then
	gives the desired result.
	\ifdefined\submit
	\qed
	\fi
\end{proof}

In particular, the above lemma shows that if $w\notin \Fix (T_{\rm 
PS}^{\lambda})$ and $w^+ \in T_{\rm PS}^{\lambda}(w)$ belong to the same convex 
set 
$R_{\iota}$, then $f_Q(w^+) < f_Q(w)$. In the other case, we will show that 
there exists some $\iota\in \I$ such 
that $w\in 
R_{\iota}$ and the vector $\hat{w}= P_{R_{\iota}} (w-\lambda \nabla f_Q(w))$ is 
distinct from $w$. First, we consider the instance when $w$ is a nondegenerate 
point of $S_2$. 

	
	\begin{definition}[Nondegenerate point]
		Let $\{R_{\iota}\}_{\iota\in \I}$ be a finite collection of closed convex sets, $S = 
		\bigcup_{j\in J}R_j$, and $w\in S$. We say that $w$ is a 
		\emph{nondegenerate} point of $S$ if there exists a unique $\iota\in \I$ such 
		that $w\in R_{\iota}$. Otherwise, it is called a \emph{degenerate} point of $S$. 
		\label{defn:nondegenerate}
	\end{definition}\cpsolved{Looks like this part is not needed? For the
		definition of nondegenerate points, we only use it in the appendix
		now, like twice, so might not be necessary to give a name (or at least
		we move the name to the appendix). We can just
		say only one $R_j$ contains $w$ in those two places maybe. What do you
		think?}


\begin{proposition}\label{prop:fdecreases_nondegenerate}
	Consider the setting of \cref{thm:global_lcp_nondegenerate} (c).
	Let $w\in S_2 \setminus \Fix (T_{\rm PS}^{\lambda})$ and $w^+ \in T_{\rm 
	PS}^{\lambda} (w)$ where $\lambda \in (0,1/L_Q]$. If $M$ is a
	nondegenerate matrix and $w$ is a nondegenerate point, then $f_Q(w^+) < f_Q(w) .$
\end{proposition}

\begin{proof}
	Let $\iota\in \I$ such that $w\in R_{\iota}$, and let $\hat{w}$ be as 
	in \cref{lemma:fdecreases_notfixedpoint}. To prove the claim, we only need 
	to show that $w \neq \hat{w}$, and then the result follows from
	\cref{lemma:fdecreases_notfixedpoint}. Suppose to the contrary that $w=\hat{w}$, 
	that is, $w = P_{R_{\iota}}(\bar{w})$ for $\bar{w} \coloneqq w-\lambda 
	\nabla f_Q(w)$. Since $w = P_{R_{\iota}}(\bar{w})$ and $(w_j,w_{n+j})\neq 
	(0,0)$ for all $j\in [n]$ by nondegeneracy of $w$, we have from \cref{eq:PRtau} that $w_j = 
	\bar{w}_j >0$ 
	and $w_{n+j} =0 $ if $j\in \iota$, and $w_{n+j} = \bar{w}_{n+j} 
	>0$ and 
	$w_j = 
	0$ otherwise. Thus, 
	\[(\bar{w}_j, \bar{w}_{n+j}) - (w_j, w_{n+j}) = \begin{cases}
		(0, \bar{w}_{n+j}) & \text{if}~j\in \iota,\\ 
		(\bar{w}_j , 0) & \text{if}~j\notin \iota,
	\end{cases}  \quad \forall j\in [n]. \] 
	Hence, 
		\begin{equation}
		(\bar{w}_j - w_j) (\bar{w}_{n+j} - w_{n+j}) = 0 ,\quad 	\forall 
		j\in [n].
		\label{eq:wbar-w_in_C2hat}
		\end{equation}
	Since $\bar{w}  - w =  -\lambda \nabla f_Q(w)$, we have 
	$\bar{w}-w \in \Ran (A^\T) =\Ker (A)^{\perp}$
	from \cref{eq:grad_f_Q}.
	By this together with 
	\cref{eq:wbar-w_in_C2hat} and the nondegeneracy of $M$, we
	have from \cite[Proposition 2.11]{ACT21} that $\bar{w}-w = 0$. 
	Consequently, we have $\nabla f_Q(w) = 0$, and since 
	$A$ is of full row rank, it follows from \cref{eq:grad_f_Q} that 
	$Aw-b=0$. That is, $w\in S_1$, and 
	in turn, we get $w\in S_1\cap S_2$. This is a contradiction since $w\notin 
	\Fix (T_{\rm PS}^{\lambda})$. Hence, $w\neq \hat{w}$, as desired.
	\ifdefined\submit
	\qed
	\fi 
\end{proof}

If $w$ is degenerate,
\cref{eq:wbar-w_in_C2hat} does not hold anymore, which prohibits the use of 
Proposition 2.11 of \cite{ACT21}. For such a case, we need the following 
lemmas. 

\begin{lemma}\cite[Theorem~3.3.4]{CPS92}\label{lemma:Pmatrix}
	$M\in \Re^{n\times n}$ is a $P$-matrix if and only if whenever 
	$x_j(Mx)_j\leq 0$ for all $j\in [n]$, we have $x=0$.
\end{lemma}
\begin{lemma}\label{lemma:fdecreases_degenerate}
	Consider the setting of \cref{thm:global_lcp_nondegenerate}.
	Let $w\in S_2 \setminus \Fix (T_{\rm PS}^{\lambda})$ and $w^+ \in T_{\rm 
	PS}^{\lambda} (w)$ for $\lambda \in (0,1/L_Q]$, and suppose that $w$ is 
	degenerate. Let $\iota \in \I$ be
	such that $w\in R_{\iota}$ and suppose that $w = \hat{w}$, where 
	$\hat{w}=P_{R_{\iota}}(w-\lambda \nabla f_Q(w))$. Denote
	$\bar{w}\coloneqq w - \lambda \nabla f_Q(w)$ and
	\begin{equation}\label{eq:Lambda(w)}
		\Gamma (w)\coloneqq \{ j\in [n] ~:~w_j = w_{n+j} = 0 
		~\text{and}~(\bar{w}_j,\bar{w}_{n+j}) \notin \Re^2_-\},
	\end{equation}
		where $\Re^2_- \coloneqq \left\{ (x_1,x_2): x_1,x_2 \leq 0
		\right\}$,
	then $\Gamma (w) \neq \emptyset$ implies $f_Q(w^+) < f_Q(w)$.
\end{lemma}

\begin{proof}
	Define $\Gamma (w)_1  \coloneqq  \Gamma (w) \cap \iota $ 
	and $\Gamma (w)_2 \coloneqq \Gamma (w) \cap \iota^c$.
	Note that since $\Gamma (w) \neq \emptyset$, either $\Gamma(w)_1$ or 
	$\Gamma (w)_2$ is nonempty. Now, since $w = 
	P_{R_{i}}(\bar{w})$, we obtain from \cref{eq:PRtau} that 
	\begin{equation}\label{eq:wbar_location}
		\begin{cases}
			\bar{w}_j < 0 , ~\bar{w}_{n+j} >0& \text{if}~ j\in \Gamma(w)_1 \\
			\bar{w}_j > 0 , ~\bar{w}_{n+j} <0  & \text{if}~j \in \Gamma(w)_2
		\end{cases} .
	\end{equation}
	Let $\iota ' \in \I$ be given by $\iota' = \Lambda_1 \cup \Lambda_2$, where
	$\Lambda_1 = (\iota  \cap \Gamma (w)^c) \cup \Gamma(w)_2$ and $\Lambda_2 = 
	\{ n+j: 
	j \in [n], j\notin \Lambda_1\}$. 
	Namely, for all $j \in [n]$ with $(w_j, w_{n+j}) = (0,0)$,
	$\iota'$ picks the one in $\{j, n+j\}$ not included in
	$\iota$.
	Then $w\in R_{\iota'}$, and by setting $w' \coloneqq P_{R_{i'}} (w-\lambda 
	\nabla f_Q(w)) = 
	P_{R_{i'}}(\bar{w})$, we have from the definition of $\iota'$ 
	and~\cref{eq:wbar_location} that 
	\[ (w_j', w_{n+j}') = \begin{cases}
		(w_j, w_{n+j}) & \text{if}~j\notin \Gamma (w)\\
		(0,\bar{w}_{n+j}) & \text{if}~ j\in \Gamma (w)_1 \\ 
		(\bar{w}_j, 0) &   \text{if}~j \in \Gamma (w)_2
	\end{cases} .\]
	Since $\bar{w}_{n+j}\neq 0$ for $j\in \Gamma(w)_1$ and $\bar{w}_j \neq 0$ 
	for $j\in \Gamma(w)_2$, we see that $w\neq w'$. By 
	\cref{lemma:fdecreases_notfixedpoint}, $f_Q(w^+) < f_Q(w)$.  
	\ifdefined\submit
	\qed
	\fi     
\end{proof}


\begin{proof}[\cref{theorem:fdecreases_Pmatrix}]
	If $w$ is nondegenerate, the result follows from 
	\cref{prop:fdecreases_nondegenerate} since $P$-matrices are
	nondegenerate. Assume that $w$ is degenerate and let $\iota \in \I$ such 
	that $w\in R_{\iota}$. If 
	$w\neq \hat{w}$, where $\hat{w} \coloneqq  P_{R_{\iota}}(w-\lambda \nabla 
	f_Q(w))$, 
	the result immediately follows 
	from 
	\cref{lemma:fdecreases_notfixedpoint}.

	Suppose now that $w=\hat{w}$. We 
	claim that $\Gamma (w)$ given by~\cref{eq:Lambda(w)} is 
	nonempty. To this end, consider the following index sets:
		\begin{align}
			\label{eq:I1}
			I_1(w) &\coloneqq  \{ j\in [n] ~:~ w_j=w_{n+j} = 0  \},
			\\
			\nonumber
			I_2(w) &\coloneqq  \{ j\in [n] ~:~ w_j > 0~\text{and}~w_{n+j}=0  
		\},\\
			\nonumber
		I_3(w) & \coloneqq \{ j\in [n] ~:~ w_j =0~\text{and}~w_{n+j} > 
	0  \}.
\end{align}
	Since $w\in S_2$, $I_1 \cup I_2 \cup I_3 = [n]$. If $j\in I_2 \cup 
	I_3$, we have from the equation $w = P_{R_{i}}(\bar{w})$ 
	and \cref{eq:PRtau} that 
	\begin{equation}\label{eq:I2I3}
		\begin{cases}
			w_j = \bar{w}_j >0 & \text{if}~j\in I_2 \\
			w_{n+j} = \bar{w}_{n+j} >0 & \text{if}~j\in I_3 
		\end{cases} . 
	\end{equation}
	Meanwhile, as in the proof of \cref{prop:fdecreases_nondegenerate}, we have 
	$z\coloneqq \bar{w}-w \in \Ran (A^\T)$. From the formula of $A$, it is not 
	difficult to verify that $\Ran (A^\T) = \Ker ([\Id~M^{\T}])$. Thus, by letting 
	$z = (u,v) \in \Re^n \times \Re^n$, we have $u+M^\T v = 0$, that is,
		\[(M^\T v)_j = - u_j, \quad \forall j\in [n].\]
	By multiplying both sides by $v_j$, we then obtain from
	\cref{eq:I1,eq:I2I3} that
		\begin{equation}
			 v_j (M^\T v)_j = -u_j v_j = -(\bar{w}_j - w_j)(\bar{w}_{n+j} - 
			w_{n+j}) = \begin{cases}
				-\bar{w}_j \bar{w}_{n+j} & \text{if}~j\in I_1 \\ 
				0 & \text{if}~j\in I_2 \cup I_3
			\end{cases}.
		\label{eq:products}
		\end{equation}
	Now, if $\Gamma (w) = \emptyset$, then 
	$\bar{w}_j,\bar{w}_{n+j} \leq 0$ for all $j\in I_1$, and the above equation 
	implies that $ v_j (M^\T v)_j  \leq 0$ for all $j\in [n]$. By 
	\cref{lemma:Pmatrix}, $v=0$ 
	since $M$ is a $P$-matrix, which in turn gives $u=0$. That is, we have 
	$\bar{w}-w = 0$. 
	As shown in the proof of \cref{prop:fdecreases_nondegenerate}, this implies 
	that $w\in S_1\cap S_2$, which is a contradiction since $w\notin \Fix 
	(T_{\rm PS}^{\lambda})$. Hence, we must have $\Gamma (w) \neq \emptyset$ 
	and by \cref{lemma:fdecreases_degenerate}, we get $f_Q(w^+) < f_Q(w)$.
	\ifdefined\submit
	\qed
	\fi
\end{proof}

\section{Proof of \cref{thm:fixedpoints_Pmatrix}}\label{app:fixed_Pmatrix}
	First, we show that
	\begin{equation}
		\exists \bar{w}: 
	w\in P_{S_2}(\bar{w}) \quad \text{and} \quad \bar{w}-w \in 
	\Ran(A^\T) \quad \Longrightarrow \quad w=\bar{w}\in S_1\cap S_2.
	\label{eq:regularity}
	\end{equation}
	Indeed, let $z\coloneqq \bar{w}-w $ and denote $z=(u,v)\in \Re^n\times 
	\Re^n$. From the proof of \cref{theorem:fdecreases_Pmatrix}, we know 
	that \cref{eq:products} holds. 
	Since $w\in P_{S_2}(\bar{w})$, we have that $(\bar{w}_j,\bar{w}_{n+j}) \in 
	\Re^2_-$ for all $j\in I_1$. Hence, we obtain from the same arguments in 
	the proof of \cref{theorem:fdecreases_Pmatrix} that $w=\bar{w}\in S_1\cap 
	S_2$. We now consider the three cases separately:
	\begin{enumerate}[(i)]
		\item Suppose $w\in \Fix (\Tpdmc)$. Then from 
			\cref{eq:pdmc_affine_ucs}, it 
		can be verified that $w = -\nabla f_Q(w) + w'$, where $w'\in 
		P_{S_2}(w)$. By \cref{eq:grad_f_Q} and \cref{eq:regularity}, we get 
		the desired result.
		
		\item Let $z\in \Fix (\Tfb)$ and denote
		$\bar{w} \coloneqq z - \lambda \nabla f_Q(z)$ and $w \coloneqq ((1+\lambda)z -
		\bar{w})/\lambda$.
		From the formula of 
		$\Tfb$ in \cref{eq:fb_affine_ucs}, we can derive that 
		$w \in P_{S_2}(\bar{w})$.
		We then have
		$\bar{w}-w = (1+\lambda)(\bar{w}-z) / \lambda
		= - (1+\lambda)\nabla f_Q(z)$, and thus $\bar{w}-w\in \Ran(A^\T)$ by 
		\cref{eq:grad_f_Q}. By \cref{eq:regularity}, we have $w=\bar{w}\in 
		S_1\cap S_2$. From the formula of $w$, we obtain that $z=w$, so
		$z\in S_1\cap S_2$. 
		
		\item If $w\in \Fix (T_{\rm PS}^{\lambda})$, from 
		\cref{eq:ps_affine_ucs}, we have $w \in P_{S_2}(\bar{w})$, where 
		$\bar{w} = w - \lambda \nabla f_Q(w)$. Thus,
		 we obtain from \cref{eq:grad_f_Q} that $\bar{w}-w\in \Ran
		 (A^\T)$, so
		 $w=\bar{w}\in S_1\cap S_2$ by \cref{eq:regularity}.
		\ifdefined\submit
		\qed
		\else 
		\qedhere
		\fi
	\end{enumerate}
\end{document}